\documentclass{amsart}

\usepackage{bc}
\usetikzlibrary{calc,decorations.pathmorphing,patterns}
\pgfdeclaredecoration{penciline}{initial}{
    \state{initial}[width=+\pgfdecoratedinputsegmentremainingdistance,
    auto corner on length=1mm,]{
        \pgfpathcurveto%
        {
            \pgfqpoint{\pgfdecoratedinputsegmentremainingdistance}
                      {\pgfdecorationsegmentamplitude}
        }
        {
        \pgfmathrand
        \pgfpointadd{\pgfqpoint{\pgfdecoratedinputsegmentremainingdistance}{0pt}}
                    {\pgfqpoint{-\pgfdecorationsegmentaspect
                     \pgfdecoratedinputsegmentremainingdistance}%
                               {\pgfmathresult\pgfdecorationsegmentamplitude}
                    }
        }
        {
        \pgfpointadd{\pgfpointdecoratedinputsegmentlast}{\pgfpoint{5pt}{5pt}}
        }
    }
    \state{final}{}
}

\newcommand\Forget{{\mathcal F}}
\newcommand\Antipush{{\mathcal E}}
\newcommand\Push{{\operatorname{push}}}
\newcommand\Braid{{\mathbf{Braid}}}
\newcommand\knBraid{{\mathbf{knBraid}}}

\newcommand\Julia{{\mathcal J}}
\newcommand\wtf{{\widetilde f}}
\newcommand\wtg{{\widetilde g}}
\newcommand\wtA{{\widetilde A}}
\newcommand\wtB{{\widetilde B}}
\newcommand\wtC{{\widetilde C}}
\newcommand\wtG{{\widetilde G}}
\newcommand\wtH{{\widetilde H}}
\newcommand\wtord{{\widetilde\ord}}
\newtheorem{mainprop}[mainthm]{Proposition}
\begin{document}
\title[Erasing maps, orbispaces, and the Birman exact sequence]{Algorithmic aspects of branched coverings III/V.\\ Erasing maps, orbispaces, and the Birman exact sequence}
\author{Laurent Bartholdi}
\email{laurent.bartholdi@gmail.com}
\address{\'Ecole Normale Sup\'erieure, Paris \emph{and} Mathematisches Institut, Georg-August Universit\"at zu G\"ottingen}
\author{Dzmitry Dudko}
\email{dzmitry.dudko@gmail.com}
\address{Jacobs University, Bremen}
\thanks{Partially supported by ANR grant ANR-14-ACHN-0018-01, DFG grant BA4197/6-1 and ERC grant ``HOLOGRAM''}
\date{February 8, 2018}
\begin{abstract}
  Let $\wtf\colon(S^2,\wtA)\selfmap$ be a Thurston map and let
  $M(\wtf)$ be its mapping class biset: isotopy classes rel
  $\wtA$ of maps obtained by pre- and post-composing $\wtf$ by
  the mapping class group of $(S^2,\wtA)$. Let
  $A\subseteq\wtA$ be an $\wtf$-invariant subset, and let
  $f\colon(S^2,A)\selfmap$ be the induced map. We give an analogue of
  the Birman short exact sequence: just as the mapping class group $\Mod(S^2,\wtA)$ is
  an iterated extension of $\Mod(S^2,A)$ by fundamental groups of
  punctured spheres, $M(\wtf)$ is an iterated extension of $M(f)$ by
  the dynamical biset of $f$.

  Thurston equivalence of Thurston maps classically reduces to a
  conjugacy problem in mapping class bisets.  Our short exact sequence
  of mapping class bisets allows us to reduce in polynomial time the
  conjugacy problem in $M(\wtf)$ to that in $M(f)$.  In case $\wtf$ is
  geometric (either expanding or doubly covered by a hyperbolic torus
  endomorphism) we show that the dynamical biset $B(f)$ together with
  a ``portrait of bisets'' induced by $\wtA$ is a complete conjugacy
  invariant of $\wtf$.

  Along the way, we give a complete description of bisets of
  $(2,2,2,2)$-maps as a crossed product of bisets of torus
  endomorphisms by the cyclic group of order $2$, and we show that
  non-cyclic orbisphere bisets have no automorphism.

  We finally give explicit, efficient algorithms that solve the
  conjugacy and centralizer problems for bisets of expanding or torus
  maps.
\end{abstract}
\maketitle

\section{Introduction}
A \emph{Thurston map} is a branched covering $f\colon S^2\selfmap$ of
the sphere whose \emph{post-critical set}
$P_f\coloneqq\bigcup_{n\ge1}f^n(\text{critical points of }f)$ is
finite.

Extending~\cite{nekrashevych:ssg}, we developed
in~\cites{bartholdi-dudko:bc0,bartholdi-dudko:bc1,bartholdi-dudko:bc2,bartholdi-dudko:bc4}
an algebraic machinery that parallels the topological theory of
Thurston maps: one considers the \emph{orbisphere} $(S^2,P_f,\ord_f)$,
with $\ord_f\colon P_f\to\{2,3,\dots,\infty\}$ defined by
\[\ord_f(p)=\lcm\{\deg_q(f^n)\mid n\ge0, f^n(q)=p\},
\]
and the \emph{orbisphere fundamental group} $G=\pi_1(S^2,P_f,\ord,*)$,
which has one generator of order $\ord_f(p)$ per point $p\in P_f$, and
one relation. Then $f$ is encoded in the structure of a
\emph{$G$-$G$-biset} $B(f)$: a set with commuting left and right
actions of
$G$. By~\cite{bartholdi-dudko:bc2}*{Theorem~\ref{bc2:thm:dehn-nielsen-baer++}}
(see also Corollary~\ref{cor:OrbiIsComplInvar}), the isomorphism class
of $B(f)$ is a complete invariant of $f$ up to isotopy.

There is a missing element to this description, that of ``extra marked
points''. In the process of a decomposition of spheres into smaller
pieces, one is led to consider Thurston maps with extra marked points,
such as periodic cycles or preimages of post-critical points. They can
be added to $P_f$, but they have order $1$ under $\ord_f$ so are
invisible in $\pi_1(S^2,P_f,\ord_f)$. The orbisphere orders can be
made artificially larger, but then other properties, such as the
characterization of expanding maps as those having a ``contracting''
biset (see~\cite{bartholdi-dudko:bc4}*{Theorem~\ref{bc4:thm:main}})
are lost.

We resolve this issue, in this article, by introducing \emph{portraits
  of bisets} and exhibiting their algorithmic properties. As an
outcome, the conjugacy and centralizer problems for Thurston maps with
extra marked points reduces to that of the underlying map with only
$P_f$ marked.

This allows us, in particular, to understand algorithmically maps
doubly covered by torus endomorphisms.

\subsection{Maps and bisets}
The natural setting is an orbisphere $(S^2,\wtA,\wtord)$ and a
sub-orbisphere $(S^2,A,\ord)$; namely one has $A\subseteq\wtA$ and
$\ord(a)\mid\wtord(a)$ for all $a\in A$. There is a
corresponding morphism of fundamental groups
$\pi_1(S^2,\wtA,\wtord,*)\eqqcolon\wtG\twoheadrightarrow
G\coloneqq\pi_1(S^2,A,\ord,*)$, called a \emph{forgetful morphism}. We
considered in~\cite{bartholdi-dudko:bc2}*{\S\ref{bc2:ss:orbispheres}}
the \emph{inessential} case in which
$\ord(a)>1\Leftrightarrow\wtord(a)>1$, so that only the type
of singularities are changed by the forgetful functor.  Here we are
more interested in the \emph{essential} case, in which
$A\subsetneqq\wtA$.

In this Introduction, we restrict ourselves to self-maps of
orbispheres; the general non-dynamical case is covered
in~\S\ref{ss:forgetful}. An orbisphere self-map
$f\colon(S^2,A,\ord)\selfmap$ is a branched covering
$f\colon S^2\selfmap$ with $\ord(p)\deg_p(f)\mid\ord(f(p))$ for all
$p\in S^2$. Given a map $f\colon(S^2,A)\selfmap$, there may exist
different orbisphere structures on $(S^2,A)$ turning $f$ into an
orbisphere self-map; in particular, the maximal one, in which
$\ord(a)=\infty$ for every $a\in A$, and the minimal one, in which
$\ord(a)=\lcm\{\deg(f^n@x)\mid n\ge0, x\in f^{-n}(a)\}$. The self-map
$f$ is encoded by the biset
\begin{equation}\label{def:B(f)}
  B(f)\coloneqq\{\beta\colon[0,1]\to S^2\setminus A\mid
  \beta(0)=*=f(\beta(1))\}/{\approx_{A,\ord}},
\end{equation}
where $\approx_{A,\ord}$ denotes homotopy rel $(A,\ord)$, and the
commuting left and right $\pi_1(S^2,A,\ord,*)$-actions are given
respectively by concatenation of paths or their appropriate $f$-lift.

Bisets form a convenient category with products, detailed
in~\cite{bartholdi-dudko:bc1}. A \emph{self-conjugacy} of a
$G$-$G$-biset $B$ is a pair of maps
$(\phi\colon G\selfmap,\beta\colon B\selfmap)$ with
$\beta(h b g)=\phi(h)\beta(b)\phi(g)$; and an \emph{automorphism} of
$B$ is a self-conjugacy with $\phi=\one$. Cyclic bisets (that of maps
$z\mapsto z^d\colon (\widehat{\C},\{0,\infty\})\selfmap$) have a special status; for the others,
\begin{mainprop}[= Corollary~\ref{cor:NoGhostAut}]\label{mainprop:NoGhostAut}
  If $\subscript G B_G$ is a non-cyclic orbisphere biset then
  $\Aut(B)=1$.
\end{mainprop}

Consider $\wtA=A\sqcup\{d\}$, obtained by adding a point to
the orbisphere $(S^2,A,\ord)$, resulting in an orbisphere
$(S^2,\wtA,\wtord)$. There is a short exact sequence,
called the \emph{Birman exact sequence},
\begin{equation}\label{eq:birmanses}
  1\to\pi_1(S^2\setminus A,d)\to\Mod(S^2,\wtA)\to\Mod(S^2,A)\to1
\end{equation}
(note that the orbisphere structure is ignored.) Let
$f\colon(S^2,A,\ord)\selfmap$ be an orbisphere map, and recall that its
\emph{mapping class biset} is the set
\[M(f)\coloneqq\{m' f m''\mid m',m''\in\Mod(S^2,A)\}/{\approx_A},
\]
with natural left and right actions of $\Mod(S^2,A)$. Assume first
that the extra point $d$ is fixed by $f$, and write
$\wtf\colon(S^2,\wtA,\wtord)\selfmap$ for the map $f$ acting
on $(S^2,\wtA,\wtord)$.  There is then a natural map
$M(\wtf)\to M(f)$, and we will see that it is an \emph{extension of
  bisets} (see Definition~\ref{defn:ses}):
\begin{mainthm}[= Corollary~\ref{cor:thm:extension}]\label{mainthm:extension}
  Subordinate to the exact sequence of groups~\eqref{eq:birmanses},
  there is a short exact sequence of bisets
  \[\subscript{\pi_1(S^2\setminus A,d)}{\bigg(\bigsqcup_{f'\in M(f)}B(f')\bigg)}_{\pi_1(S^2\setminus A,d)}\hookrightarrow M(\wtf)\twoheadrightarrow M(f)
  \]
  where $B(f')$ denotes the biset of $f'\colon (S^2,A)\selfmap$ rel
  the base point $d$.
\end{mainthm}

The statement can be easily extended to
$\wtA=A\sqcup\{\text{periodic points}\}$: for a cycle of length
$\ell$, the fibres in the short exact sequence are of the form
$B(f)^\ell$ with left- and right-action twisted along an
$\ell$-cycle. This case is at the heart of our reduction of the
conjugacy problem from $M(\wtf)$ to $M(f)$. Approximately the same
picture applies if $D$ contains preimages of points in $A$, but there
are subtle complications which are taken care of
in~\S\ref{ss:forgetful}, see Theorem~\ref{thm:4Ext of MCB}. In
essence, the presence of preperiodic points imposes a finite index
condition on centralizers, and splits conjugacy classes into finitely
many pieces, see the remark after Lemma~\ref{lem:Zportrait=Zhpo}.

\subsection{Portraits of bisets}
Portraits of bisets emerge from a simple remark: fixed points of $f$
naturally yield conjugacy classes in $B(f)$. Indeed if $f(p)=p$ then
choose a path $\ell\colon[0,1]\to S^2\setminus A$ from $*$ to $p$, and
consider $c_p\coloneqq\ell\#\ell^{-1}\lift f p\in B(f)$, which is
well-defined up to conjugation by $G$, namely a different choice of
$\ell$ would yield $g^{-1}c_p g$ for some $g\in G$. Conversely, if $f$
expands a metric, then every conjugacy class in $B(f)$ corresponds to
a unique repelling $f$-fixed point.

A \emph{portrait of bisets} in $B(f)$, see
Definition~\ref{dfn:PrtrOfBst}, consists of a map
$f_*\colon\wtA\selfmap$ extending $f\restrict A$; a collection of
\emph{peripheral} subgroups $(G_a)_{a\in\wtA}$ of $G$: they
are such that every $G_a$ is cyclic and generated by a ``lollipop''
around $a$; and a collection of cyclic $G_a$-$G_{f_*(a)}$-bisets
$(B_a)_{a\in\wtA}$.  Two portraits of bisets
$(G_a,B_a)_{a\in\wtA}$ and $(G'_a,B'_a)_{a\in\wtA}$
parameterized by the same $f_*\colon\wtA\selfmap$ are
\emph{conjugate} if there exist $(\ell_a)_{a\in\wtA}$ in $G^\wtA$
such that $G_a=\ell_a^{-1}G'_a\ell_a$ and
$B_a=\ell_a^{-1}B'_a\ell_{f_*(a)}$. The set of self-conjugacies of a
portrait is called its \emph{centralizer}.

In case $A=\wtA$, every biset admits a unique \emph{minimal}
portrait up to conjugacy, which may be understood geometrically as
follows. Consider a branched covering $f\colon (S^2,A)\selfmap$. For
every $a\in A$ choose a small disk neighbourhood $\mathcal D_a$ of it; up
to isotopy we may assume that
$f\colon \mathcal D_a \setminus\{a\}\to\mathcal D_{f(a)} \setminus\{f(a)\}$ is a
covering. A choice of embeddings
$\pi_1(\mathcal D_a\setminus\{a\})\hookrightarrow\pi_1(S^2,A)$ yields a family
$(G_a)_{a\in A}$ of peripheral subgroups; and the corresponding
embeddings
$B(f\colon \mathcal D_a\setminus\{a\} \to\mathcal D_{f(a)}\setminus
\{f(a)\})\hookrightarrow B(f)$ yields a minimal portrait of bisets
$(G_a,B_a)_{a\in A}$ in $B(f)$.

In case $\wtA=A\sqcup\{e_1,\dots,e_n\}$ with $(e_1,\dots,e_n)$
a periodic cycle, the bisets $B_{e_i}$ consist of single points, and
almost coincide with Ishii and Smillie's notion of \emph{homotopy
  pseudo-orbits}, see~\cite{ishii-smillie:shadowing}
and~\S\ref{ss:shadowing}: imagine that $(e_1,\dots,e_n)\subset S^2$ is
almost a periodic cycle, in that $f(e_i)$ is so close to $e_{i+1}$
that there is a well-defined ``shortest'' path $\ell_i$ from the first
to the second, indices being read modulo $n$. Choose for each $i$ a
path $m_i$ from $*$ to $e_i$. Set then
$B_i\coloneqq\{m_i\#(\ell_{i+1}\#m_{i+1}^{-1})\lift f{e_i}\}$, the
portrait of bisets encoding $(e_1,\dots,e_n)$.

Theorem~\ref{mainthm:extension} is proven via portraits of bisets. There
is a natural \emph{forgetful intertwiner of bisets}
\begin{equation}\label{eq:forgetinter}
  \subscript{\wtG}{\wtB}_{\wtG} \coloneqq B(\wtf\colon(S^2,\wtA,\wtord)\selfmap )\to B(f\colon(S^2,A,\ord)\selfmap)\eqqcolon \subscript G B_G
\end{equation}
given by $b\mapsto1\otimes b\otimes 1$ where
$B\cong G\otimes_{\wtG}\wtB\otimes_\wtG G$. Every portrait in $\wtB$,
for example its minimal one, induces a portrait in $B$ via the
forgetful map. Let us denote by $\Mod(S^2,A)$ the pure mapping class
group of $S^2\setminus A$. A class $m\in \Mod(S^2,\wtA)$ is called
\emph{knitting} if the image of $m$ is trivial in
$\Mod(S^2,A\sqcup \{a\})$ for every $a\in\wtA\setminus A$. We prove:

\begin{mainthm}[= Theorem~\ref{thm:portrait}]\label{main:thm:A}
  Let $\subscript G B_G$ be an orbisphere biset with portrait
  $f_*\colon A\selfmap$, and let $\wtG\twoheadrightarrow G$ and
  $f_*\colon\wtA\selfmap$ and $\deg\colon\wtA\to\N$ be compatible
  extensions.

  There is then a bijection between, on the one hand, conjugacy
  classes of portraits of bisets $(G_a,B_a)_{a\in\wtA}$ in $B$
  parameterized by $f_*$ and $\deg$ and, on the other hand,
  $\wtG$-$\wtG$-bisets projecting to $B$ under
  $\wtG\twoheadrightarrow G$ considered up to composition with the
  biset of a knitting element. This bijection maps every minimal
  portrait of bisets of $\wtB$ to $(G_a,B_a)_{a\in\wtA}$.
\end{mainthm}

\subsection{Geometric maps}\label{ss:geometric}
A homeomorphism $f\colon (S^2,A)\selfmap$ is \emph{geometric} if $f$
is either of finite order ($f^n=\one$ for some $n>0$) or pseudo-Anosov
(there are two transverse measured foliations preserved by $f$ such
that one foliation is expanded by $f$ while another is contracted). In
both cases, $f$ preserves a geometric structure on $S^2$, and every
surface homeomorphism decomposes, up to isotopy, into geometric
pieces, see~\cite{thurston:surfaces}.

Consider now a non-invertible sphere map $f\colon (S^2,A)\selfmap$,
and let $A^\infty\subseteq A$ denote the forward orbit of the periodic
critical points of $f$. The map $f$ is \emph{B\"ottcher expanding} if
there exists a metric on $S^2\setminus A^\infty$ that is expanded by
$f$, and such that $f$ is locally conjugate to
$z\mapsto z^{\deg_a(f)}$ at every $a\in A^\infty$. The map $f$ is
\emph{geometric} if $f$ is either
\begin{itemize}
\item[\Exp] B\"ottcher expanding; or
\item[\Tor] a quotient of a torus endomorphism
  $z\mapsto M z+q\colon\R^2/\Z^2\selfmap$ by the involution
  $z\mapsto-z$, for a $2\times2$ matrix $M$ whose eigenvalues are
  different from $\pm 1$.
\end{itemize}

The two cases are not mutually exclusive. A map $f\in\Tor$ is
expanding if and only if the absolute values of the eigenvalues of $M$
are greater than $1$. Note also that if $f$ is non-invertible and
covered by a torus endomorphism then either $f\in\Exp$ or the minimal
orbisphere of $f$ satisfies $\#P_f=4$ and $\ord_f\equiv2$.

In that last case, we show that $\subscript G B_G$ is a crossed
product of an Abelian biset with an order-$2$ group.  Let us fix a
$(2,2,2,2)$-orbisphere $(S^2,A, \ord)$ and let us set
$G\coloneqq \pi_1(S^2,A, \ord)\cong\Z^2\rtimes\{\pm1\}$. We may
identify $A$ with the set of all order-$2$ conjugacy classes of
$G$. By Euler characteristic, every branched covering
$f\colon (S^2,A, \ord) \selfmap$ is a self-covering. Therefore, the
biset of $f$ is right principal.

We denote by $\Mat_2^+(\Z)$ the set of $2\times2$ integer matrices $M$
with $\det(M)>0$. For a matrix $M\in \Mat^+_2(\Z)$ and a vector
$v\in \Z^2$ there is an injective endomorphism
$M^v\colon \Z^2\rtimes\{\pm1\}\selfmap$ given by the following
``crossed product'' structure:
\begin{equation}\label{eq:Mv}
  M^{v}(n,1)=(M n,1) \text{ and }M^{v}(n,-1)=(M n+v,-1).
\end{equation}
Furthermore, $\Z^2\rtimes\{\pm1\}$ has exactly $4$ conjugacy classes
of order $2$, which we denote by $A$. Then $M^v$ induces a map
$(M^v)_*\colon A \selfmap$ on these conjugacy classes. We write
$B_{M^v}=B_M\rtimes\{\pm1\}$ the crossed product decomposition of the
biset of $M^v$.

\begin{mainprop}[= Propositions~\ref{prop:EndOfG} and~\ref{prop:EndOfG2}]
  The biset of $M^v$ from~\eqref{eq:Mv} is an orbisphere biset, and
  conversely every $(2,2,2,2)$-orbisphere biset $B$ is of the form
  $B=B_{M^v}$ for some $M\in\Mat^+_2(\Z)$ and some $v\in\Z^2$. Two
  bisets $B_{M^v}$ and $B_{N^w}$ are isomorphic if and only if
  $M=\pm N$ and $v\equiv w\pmod{2\Z^2}$.  The biset $B_{M^{v}}$ is
  geometric if and only if both eigenvalues of $M$ are different from
  $\pm 1$.
\end{mainprop}

A distinguished property of a geometric map is \emph{rigidity}: two
geometric maps are Thurston equivalent (namely, conjugate up to
isotopy) if and only if they are topologically conjugate.

An orbisphere biset $\subscript G B_{G}$ is \emph{geometric} if it is
the biset of a geometric map, and \Tor\ and \Exp\ bisets are defined
similarly. If $B$ is a geometric biset, then by rigidity there is a
map $f_B\colon (S^2,A,\ord)\selfmap$, unique up to conjugacy, with
$B(f_B)\cong B$.

If $\subscript G B_G$ is geometric and
$\subscript\wtG{\wtB}_{\wtG}\to\subscript G
B_G$ is a forgetful intertwiner as in~\eqref{eq:forgetinter}, then
elements of $\wtA\setminus A$ (which \emph{a priori} do not
belong to any sphere) can be interpreted dynamically as extra marked
points on $S^2\setminus A$. More precisely, if $\wtB$ is
itself geometric, and $B$ is the biset of the geometric map
$f\colon(S^2,A)\selfmap$, then there is an embedding of $\wtA$
in $S^2$ as an $f$-invariant set, unique unless $G$ is cyclic, in such
a manner that $\wtB$ is isomorphic to
$B(f\colon(S^2,\wtA)\selfmap)$.

Since furthermore geometric maps have only finitely many periodic
points of given period, we obtain a good understanding of conjugacy
and centralizers of geometric bisets:
\begin{mainthm}[= Theorem~\ref{thm:conjcentr}]\label{thm:RedConjCentrProb}
  Let $\wtG\to G$ be a forgetful morphism of groups and let
  \[\subscript\wtG\wtB_\wtG\to\subscript G B_G\;\text{ and }\; \subscript\wtG{\wtB'}_\wtG\to\subscript G{B'}_G
  \]
  be two forgetful biset morphisms as in~\eqref{eq:forgetinter}.
  Suppose furthermore that $\wtB$ is geometric of degree
  $>1$. Denote by $(G_a,B_a)_{a\in \wtA}$ and
  $(G'_a,B'_a)_{a\in \wtA}$ the portraits of bisets induced by
  $\wtB$ and $\wtB'$ in $B$ and $B'$ respectively.

  Then $\wtB,\wtB'$ are conjugate under
  $\Mod(\wtG)$ if and only if there exists $\phi\in\Mod(G)$
  such that $B^\phi\cong B'$ and the portraits
  $(G_a^\phi,B_a^\phi)_{a\in \wtA}$ and
  $(G'_a,B'_a)_{a\in \wtA}$ are conjugate.

  Furthermore, the centralizer of the portrait
  $(G_a,B_a)_{a\in\wtA}$ is trivial, and the centralizer
  $Z(\wtB)$ of $\wtB$ is isomorphic, via the forgetful
  map $\Mod(\wtG)\to\Mod(G)$, to
  \[\big\{\phi\in Z(B)
    \,\big|\,(G_a^{\phi},B_a^{\phi})_{a\in \wtA}\sim(G_a,B_a)_{a\in
    \wtA}\big\}
  \]
  and is a finite-index subgroup of $Z(B)$.
\end{mainthm}

Let us call an orbisphere map $f\colon(S^2,A,\ord)\selfmap$
\emph{weakly geometric} if its minimal quotient on $(S^2,P_f,\ord_f)$
is geometric; an orbisphere biset $B$ is weakly geometric if its
minimal quotient orbisphere biset is geometric. In
Theorem~\ref{thm:tuning} we characterize weakly geometric maps as
those decomposing as a tuning by homeomorphisms: starting from a
geometric map, some points are blown up to disks which are mapped to
each other by homeomorphisms.

\subsection{Algorithms}
An essential virtue of the portraits of bisets introduced above is
that they are readily usable in algorithms. Previous articles in the
series already highlighted the algorithmic aspects of bisets; let us
recall the following.

From the definition of orbisphere bisets
in~\cite{bartholdi-dudko:bc2}*{Definition~\ref{bc2:dfn:SphBis}}, it is
clearly decidable whether a given biset $\subscript H B_G$ is an
orbisphere biset; the groups $G$ and $H$ may be algorithmically
identified with orbisphere groups $\pi_1(S^2,A,\ord,*)$ and
$\pi_1(S^2,C,\ord,\dagger)$ respectively, and the induced map
$f_*\colon C\to A$ is computable. In particular, if $\subscript G B_G$
is a $G$-$G$-biset, then the dynamical map $f_*\colon A\selfmap$ is
computable, the minimal orbisphere quotient
$\overline G\coloneqq\pi_1(S^2,P_f,\ord_f,*)$ is computable, and the
induced $\overline G$-$\overline G$-biset $\overline B$ is
computable. It is also easy
(see~\cite{bartholdi-dudko:bc4}*{\S\ref{bc4:ss:algo}}) to determine
from an orbisphere biset whether it is \Tor\ (and then to determine an
affine map $M z+q$ covering it) or \Exp.

We shall show that recognizing conjugacy of portraits is decidable, and
give efficient (see below) algorithms proving it, as follows:
\begin{algo}[= Algorithms~\ref{algo:conjportraittor} and~\ref{algo:conjportraitexp}]\label{algo:conjportrait}
  \textsc{Given} a minimal geometric orbisphere biset $\subscript G B_G$, an extension $f_*\colon\wtA\to\wtA$ of the dynamics of $B$ on its peripheral classes, and two portraits of bisets $(G_a,B_a)_{a\in\wtA}$ and $(G'_a,B'_a)_{a\in\wtA}$ with dynamics $f_*$,\\
  \textsc{Decide} whether $(G_a,B_a)_{a\in\wtA}$ and
  $(G'_a,B'_a)_{a\in\wtA}$ are conjugate, and \textsc{Compute}
  the centralizer of $(G_a,B_a)_{a\in\wtA}$, which is a finite
  abelian group.
\end{algo}

\begin{algo}[= Algorithms~\ref{algo:listportraittor} and~\ref{algo:listportraitexp}]\label{algo:listportrait}
  \textsc{Given} a minimal geometric orbisphere biset $\subscript G B_G$ and an extension $f_*\colon\wtA\to\wtA$ of the dynamics of $B$ on its peripheral classes,\\
  \textsc{Produce a list} of representatives of all conjugacy classes of portraits of bisets $(G_a,B_a)_{a\in\wtA}$ in $B$ with dynamics $f_*$.
\end{algo}

Thurston equivalence to a map $f\colon(S^2,A)\selfmap$ reduces to the
conjugacy problem in the mapping class biset $M(f)$. Let $X$ be a
basis of $M(f)$ and let $N$ be a finite generating set of
$\Mod(S^2,A)$; so $M(f)=\bigcup_{n\ge0} N^n X$. We call an algorithm
with input in $M(f)\times M(f)$ \emph{efficient} if for $f,g\in N^n X$
the running time of the algorithm is bounded by a polynomial in $n$.

We deduce that conjugacy and centralizer problems are decidable for
geometric maps, as long as they are decidable on their minimal
orbisphere quotients:
\begin{maincor}[= Algorithm~\ref{algo:red}]\label{cor:conjZpb}
  There is an efficient algorithm with oracle that, given two
  orbisphere maps $f,g$ by their bisets and such that $f$ is
  geometric, decides whether $f,g$ are conjugate, and computes the
  centralizer of $f$.

  The oracle must answer, given two geometric orbisphere maps $f,g$ on
  their minimal orbisphere $(S^2,P_f,\ord_f)$ respectively
  $(S^2,P_g,\ord_g)$, whether they are conjugate and what the
  centralizer of $f$ is.
\end{maincor}

Algorithms for the oracle itself will be described in details in the last article
of the series~\cite{bartholdi-dudko:bc5}. Furthermore, we have the
following oracle-free result, proven in~\S\ref{ss:decidability}:
\begin{maincor}\label{cor:G}
  There is an efficient algorithm that decides whether a rational map
  is equivalent to a given twist of itself, when only extra marked
  points are twisted.
\end{maincor}

\subsection{Historical remarks: Thurston equivalence and its complexity}
The conjugacy problem is known to be solvable in mapping class
groups~\cite{hemion:homeos}. The state of the art is based on the
Nielsen-Thurston classification: decompose maps along their canonical
multicurve; then a complete conjugacy invariant of the map is given by
the combinatorics of the decomposition, the conjugacy classes of
return maps, and rotation parameters along the multicurve. For general
surfaces, the cost of computing the decomposition is at most
exponential time in $n$,
see~\cites{koberda-mangahas:detectNT,tao:linearbound}, and so is the
cost of comparing the pseudo-Anosov return
maps~\cite{masur-minsky:gcc2}. Margalit, Yurttas, Strenner recently
announced polynomial-time algorithms for all the above. At all rates,
for punctured spheres the cost of computing the decomposition is
polynomial~\cite{calvez:NT}.

Kevin Pilgrim developed, in~\cite{pilgrim:combinations}, a theory of
decompositions of Thurston maps extending the Nielsen-Thurston
decomposition of homeomorphisms. It is a fundamental ingredient in
understanding Thurston equivalence of maps, without any claims on
complexity.

In~\cite{bartholdi-n:thurston}, a general strategy is developed, along
with computations of the mapping class biset for the three degree-$2$
polynomials with three finite post-critical points and period
respectively $1,2,3$. The Douady-Hubbard ``twisted rabbit problem'' is
solved in this manner (it asks to determine for $n\in\Z$ the conjugacy
class of $r\circ t^n$ with $r(z)\approx z^2-0.12256+0.74486i$ the
``rabbit'' polynomial and $t(z)$ the Dehn twist about the rabbit's
ears). Russell Lodge computed in~\cite{lodge:pullback} the solution to
an array of other ``twisted rabbit'' problems, by finding explicitly
the mapping class biset structure.

If $M(f)$ is a contracting biset, then its conjugacy problem may be
solved in quasi-linear time. In terms of a twist parameter such as the
$n$ above, the solution has $\mathcal O(\log n)$ complexity. However,
this bound is not uniform, in that it requires e.g.\ the computation
of the nucleus of $M(f)$, which cannot \emph{a priori} be
bounded. Nekrashevych showed
in~\cite{nekrashevych:combinatorialmodels}*{Theorem~7.2} that the
mapping class biset of a hyperbolic polynomial is
contracting. Conversely, if $M(f)$ contains an obstructed map then it
is not contracting.

In the case of polynomials, however, even more is possible: bisets of
polynomials admit a particularly nice form by placing the basepoint
close to infinity (this description goes hand-in-hand with Poirier's
``supporting rays''~\cite{poirier:portraits}). As a consequence, the
``spider algorithm'' of Hubbard and Schleicher~\cite{hubbard-s:spider}
can be implemented directly at the level of bisets and yields an
efficient algorithm, also in practice since it does not require the
computation of $M(f)$'s nucleus. This algorithm will be described
in~\cite{bartholdi-dudko:bc5}.

Selinger and Yampolsky showed
in~\cite{selinger-yampolsky:geometrization} that the canonical
decomposition is computable. In this manner, they solve the conjugacy
problem for maps whose canonical decomposition has only rational maps
with hyperbolic orbifold: a complete conjugacy invariant of the map is
given by the combinatorics of the decomposition, the conjugacy classes
of its return maps, and rotation parameters along the canonical
obstruction.

We showed in~\cite{bartholdi-dudko:bc2} that the conjugacy problem is
decidable in general: a complete conjugacy invariant of a Thurston map
is given by the combinatorics of the decomposition, together with
conjugacy classes of return maps, rotation parameters along the
canonical obstruction, together with the induced action of the
centralizer groups of return maps.

We finally mention a different path towards understanding Thurston
maps, in the case of maps with four post-critical points: the ``nearly
Euclidean Thurston maps''
from~\cite{cannon-floyd-parry-pilgrim:net}. There, the restriction on
the size of the post-critical set implies that the maps may be
efficiently encoded via linear algebra; and as a consequence,
conjugacy of NET maps is efficiently decidable. We are not aware of
any direct connections between their work and ours.

\subsection{Notations}
Throughout the text, some letters keep the same meaning and are not
always repeated in statements. The symbols $A,C,D,E$ denote finite
subsets of the topological sphere $S^2$. There is a sphere map
$\wtf\colon(S^2,C\sqcup E)\to(S^2,A\sqcup D)$, which restricts to a
sphere map $f\colon(S^2,C)\to(S^2,A)$. Implicit in the definition, we
have $f(C)\cup\{\text{critical values of }f\}\subseteq A$. We write
$\wtC=C\sqcup E$ and $\wtA=A\sqcup D$. If there are,
furthermore, orbisphere structures on the involved spheres, we denote
them by $(S^2,A,\ord)$ etc., with the same symbol `$\ord$'. We also
abbreviate $X\coloneqq(S^2,A,\ord)$ and $Y\coloneqq(S^2,C,\ord)$.

For sphere maps $f_0,f_1\colon(S^2,C)\to(S^2,A)$, we mean by
$f_0\approx_C f_1$ that $f_0$ and $f_1$ are isotopic, namely there is
a path $(f_t)_{t\in[0,1]}$ of sphere maps $f_t\colon(S^2,C)\to(S^2,A)$
connecting them.

For $\gamma$ a path and $f$ a (branched) covering, we denote by
$\gamma\lift f x$ the unique $f$-lift of $\gamma$ that starts at the
preimage $x$ of $\gamma(0)$.

We denote by $f\restrict Z$ the restriction of a function $f$ to a
subset $Z$ of its domain.  Finally, for a set $Z$ we denote by
$Z\perm$ the group of all permutations of $Z$.

\section{Forgetful maps}\label{ss:forgetful}
We recall a minimal amount of information
from~\cite{bartholdi-dudko:bc2}: a \emph{marked orbisphere} is
$(S^2,A,\ord)$ for a finite subset $A\subset S^2$ and a map
$\ord\colon A\to\{2,3,\dots,\infty\}$, extended to $S^2$ by
$\ord(S^2\setminus A)\equiv1$. For a choice of basepoint
$*\in S^2\setminus A$, its \emph{orbispace fundamental group} is
generated by ``lollipop'' loops $(\gamma_a)_{a\in A}$ based at $*$
that each encircle once counterclockwise a single point of $A$, and
with $A=\{a_1,\dots,a_n\}$ has presentation
\begin{equation}\label{eq:orbispheregp}
  G=\pi_1(S^2,A,\ord,*)=\langle \gamma_{a_1},\dots,\gamma_{a_n}\mid \gamma_{a_1}^{\ord(a_1)},\dots,\gamma_{a_n}^{\ord(a_n)}, \gamma_{a_1}\cdots\gamma_{a_n}\rangle.
\end{equation}
Abstractly, i.e.\ without reference to a sphere, an orbisphere group
is a group $G$ as in~\eqref{eq:orbispheregp} together with the
conjugacy classes $\Gamma_1,\dots,\Gamma_n$ of
$\gamma_{a_1},\dots,\gamma_{a_n}$ respectively.

An \emph{orbisphere map} $f\colon(S^2,C,\ord_C)\to(S^2,A,\ord_A)$
between orbispheres is an orientation-preserving branched covering
between the underlying spheres, with
$f(C)\cup\{\text{critical values of }f\}\subseteq A$, and with
$\ord_C(p)\deg_p(f)\mid\ord_A(f(p))$ for all $p\in S^2$.

To avoid special cases, we \emph{make, throughout this article except
  in~\S\ref{ss:CyclBis}, the assumption}
   \begin{equation}
   \label{eq:StandAssump}
  \#A\ge3 \hspace{0.8cm} \text{ and } \hspace{0.8cm}\#C\ge3.
  \end{equation} For $\#A=2$, many things go wrong: one must
require $\ord$ to be constant; the fundamental group has a non-trivial
centre; and the degree-$d$ self-covering of $(S^2,A,\ord)$ has an
extra symmetry of order $d-1$. All our statements can be modified to
take into account this special case, see~\S\ref{ss:CyclBis}.

Fix basepoints $*\in S^2\setminus A$ and $\dagger\in S^2\setminus
C$. The \emph{orbisphere biset} of an orbisphere map
$f\colon(S^2,C,\ord_C)\to(S^2,A,\ord_A)$ is the
$\pi_1(S^2,C,\ord_C,\dagger)$-$\pi_1(S^2,A,\ord_A,*)$-biset
\[B(f)=\{\beta\colon[0,1]\to S^2\setminus
  C\mid\beta(0)=\dagger,\,f(\beta(1))=*\}\,/\,{\approx_{C,\ord_C}},\]
with `$\approx_{C,\ord}$' denoting homotopy in the orbispace
$(S^2,C,\ord)$.  An orbisphere biset $\subscript H B_G$ can also be
defined purely algebraically, see
\cite{bartholdi-dudko:bc2}*{\S\ref{bc2:ss:orbispheres} and
  Definition~\ref{bc2:dfn:SphBis}}. By
\cite{bartholdi-dudko:bc2}*{Theorem~\ref{bc2:thm:dehn-nielsen-baer++}},
there is an orbisphere map $f\colon(S^2,C,\ord_C)\to(S^2,A,\ord_A)$,
unique up to isotopy, such that $B$ is isomorphic to $B(f)$. We denote
by $B_*\colon C\to A$ the induced map on the peripheral conjugacy
classes of $H$ and $G$ -- they are identified with the associated
punctures.

Let $(S^2,A,\ord)$ and $(S^2,A\sqcup D,\wtord)$ be
orbispheres, and suppose $\ord(a)\mid\wtord(a)$ for all
$a\in A$. We then have a natural \emph{forgetful} homomorphism
$\Forget_D\colon\pi_1(S^2,A\sqcup
D,\wtord,*)\to\pi_1(S^2,A,\ord,*)$ given by
$\gamma_a\mapsto\gamma_a$ for $a\in A$ and $\gamma_d\mapsto1$ for
$d\in D$. We write the forgetful map
$(S^2,A\sqcup D,\wtord)\dashrightarrow(S^2,A,\ord)$ with a
dashed arrow, because even though $\Forget_D$ is a genuine group
homomorphism, the corresponding map between orbispheres is only
densely defined. Note, however, that its inverse is a genuine
orbisphere map.  

Consider forgetful maps
$(S^2,C\sqcup E,\wtord)\dashrightarrow(S^2,C,\ord)$ and
$(S^2,A\sqcup D,\wtord)\dashrightarrow(S^2,A,\ord)$; in the
sequel we shall keep the notations
\[X\coloneqq(S^2,A,\ord),\quad Y\coloneqq(S^2,C,\ord),\qquad
  \wtA\coloneqq A\sqcup D,\quad\wtC\coloneqq C\sqcup E.
\]
Let $\wtf\colon(S^2,\wtC,\wtord)\to(S^2,\wtA,\wtord)$
be an orbispace map, with such that $\wtf$ restricts to an orbisphere
map $f\colon(S^2,C,\ord)\to(S^2,A,\ord)$. We thus have
\begin{equation}\label{eq:forgetCD}
  \begin{tikzcd}
    (S^2,\wtC,\wtord)\ar[r,"\wtf"]\ar[d,dashed,"\Forget_E"] & (S^2,\wtA,\wtord)\ar[d,dashed,"\Forget_D"]\\
    \makebox[0pt][r]{$X={}$}(S^2,C,\ord)\ar[r,"f"] & (S^2,A,\ord)\makebox[0pt][l]{${}=Y.$}
  \end{tikzcd}
\end{equation}
(In particular, $\wtf(C)\subseteq A$ and $A$ contains all
critical values of $\wtf$.) We are concerned, in this section,
with the relationship between $B(\wtf)$ and $B(f)$; we shall
show that $B(\wtf)$ may be encoded by $B(f)$ and a
\emph{portrait of bisets}, and that this encoding is unique up to a
certain equivalence. The algebraic counterpart of~\eqref{eq:forgetCD} is
\begin{equation}\label{eq:forgetCD2}
  \begin{tikzcd}[column sep=0mm]
    \makebox[0pt][r]{$\wtH\coloneqq{}$}\pi_1(S^2,\wtC,\wtord,\dagger)\ar[d,"\Forget_E"] & \looparrowright B(\wtf)\looparrowleft\ar[d,"\Forget_{E,D}"] &\pi_1(S^2,\wtA,\wtord,*)\ar[d,"\Forget_D"]\makebox[0pt][l]{${}\eqqcolon\wtG$}\\
    \makebox[0pt][r]{$H\coloneqq{}$}\pi_1(S^2,C,\ord,\dagger) &\looparrowright B(f) \looparrowleft &\pi_1(S^2,A,\ord,*)\makebox[0pt][l]{${}\eqqcolon G,$}
  \end{tikzcd}
\end{equation}
where we denote by
\begin{equation}\label{eq:forgetbi}
  \Forget_{E,D}\colon B(\wtf)\to B(f)\cong H\otimes_\wtH B(\wtf)\otimes_\wtG G
\end{equation}
the natural map given by $b\mapsto1\otimes b\otimes1$.

\subsection{Braid groups and knitting equivalence}
\label{ss:BrGr KnEq}
We recall that there is no difference between $\Mod(S^2,A)$ and
$\Mod(X)$,
see~\cite{bartholdi-dudko:bc2}*{\S\ref{bc2:ss:orbispheres}}. We shall
define a subgroup $\Mod(X|D)$ intermediate between
$\Mod(X)\cong\Mod(S^2,A)$ and $\Mod(S^2,\wtA)$ which will be
useful to relate $B(\wtf)$ and $B(f)$.

\begin{defn}[Pure braid group]
\label{defn:PureBraidGroup}
  Let $D$ be a finite set on a finitely punctured sphere
  $S^2\setminus A$. The \emph{pure braid group} $\Braid(S^2\setminus A,D)$
  is the set of continuous motions
  $m\colon[0,1]\times D\to S^2\setminus A$ considered up to isotopy so
  that
  \begin{itemize}
  \item $m(t,-)\colon D\hookrightarrow S^2\setminus A$ is an inclusion
    for every $t\in [0,1]$;
  \item $m(0,-)=m(1,-)=\one\restrict D$.
  \end{itemize}
  The product in $\Braid(S^2\setminus A,D)$ is concatenation of
  motions, and the inverse is time reversal.
\end{defn}
\noindent Note that in the special case $D=\{*\}$ we have
$\Braid(S^2\setminus A,\{*\} )=\pi_1(S^2\setminus A, *)$.

\begin{thm}[Birman]\label{thm:BirmExistUniq}
  For every $m\in \Braid(S^2\setminus A,D)$ there is a unique mapping
  class $\Push(m)\in\ker(\Mod(S^2,A\sqcup D)\to\Mod(S^2,A))$ such that
  $\Push(m)$ is isotopic rel $A$ to the identity via an isotopy moving
  $D$ along $m^{-1}$. The map
  \begin{equation}\label{eq:BirnMap}
    \Push\colon \Braid(S^2\setminus A,D)\to\ker(\Mod(S^2,A\sqcup D)\to\Mod(S^2,A))
  \end{equation}
  is an isomorphism (it would be merely an epimorphism if $\#A\le2$).\qed
\end{thm}

\noindent From now on we identify $\Braid(S^2\setminus A,D)$ with its
image under~\eqref{eq:BirnMap}.

\begin{defn}[Knitting group]
  Let $X= (S^2,A,\ord)$ be an orbisphere and let $D$ be a finite
  subset of $S^2\setminus A$. The \emph{knitting braid group}
  $\knBraid(X,D)$ is the kernel of the forgetful morphism
  \[\Antipush_D\colon\Braid(S^2\setminus A, D)\to\prod_{d\in D}\pi_1(X,d);\]
  it is the set of $D$-strand braids in $S^2\setminus A$ all of whose
  strands are homotopically trivial in $X$.
\end{defn}

In case $X=(S^2,A)$, knitting elements are the
``$(\#D-1)$-decomposable braids''
from~\cite{levinson:decomposablebraids}.

\begin{lem}
\label{lem:knBraid is normal}
The knitting group $\knBraid(X,D)$ is a normal subgroup of
${\Mod(S^2,A\sqcup D)}$.
\end{lem}
\begin{proof}
  We show that for every $m\in \knBraid(X,D)$ and
  $h\in \Mod(S^2,A\sqcup D)$ we have $h^{-1}m h\in
  \knBraid(X,D)$. Indeed, $h$ restricts to an orbisphere map
  $h\colon X\selfmap$ fixing $D$ pointwise. Thus $m(d,-)$ is a trivial
  loop in $\pi_1(X,d)$ if and only if $h(m(d,-))$ is a trivial loop in
  $\pi_1(X,d)$ for every $d\in D$.
\end{proof}
\noindent Define
\begin{equation}\label{eq:modxd}
  \Mod(X|D)\coloneqq\Mod(S^2,A\sqcup D)/\knBraid(X,D).
\end{equation}
With $G=\pi_1(X,*)$, we write $\Mod(G|D)=\Mod(X|D)$ and
$\Mod(G)=\Mod(X)$. We interpret elements of $\Mod(G)$ as outer
automorphisms of $G$, as mapping classes and as biprincipal
bisets. We also introduce the notation
\[\pi_1(X,D)\coloneqq\prod_{d\in D}\pi_1(X,d)\cong G^D.\]
Note the following four exact sequences:
\begin{equation}\label{eq:4exact}
  \begin{tikzpicture}[description/.style={fill=white,inner sep=2pt},baseline]
  \node (knMod) at (-4,3){$\knBraid(X,D)$};
  \node (ModE) at (0,1.5) {$\Braid(S^2\setminus A,D)$};
  \node (PrMod) at (4,0) {$\pi_1(X,D);$};
  \node (ModAE) at (-4/3,0){\makebox[0pt][r]{$\Mod(\wtG)={}$}$\Mod(S^2,A\sqcup D)$};
  \node (ModAsE) at (0,-1.5){$\Mod(X|D)$};
  \node (ModA)  at (-4,-3){\makebox[0pt][r]{$\Mod(G)={}$}$\Mod(S^2,A)$};
  \path[->>,font=\scriptsize] (ModE) edge node[description]{$\Antipush_D$} (PrMod)
    (ModAE) edge (ModAsE)
    (ModAE) edge node[description]{$\Forget_D$}  (ModA)
    (ModAsE) edge node[description]{$\Forget_D$} (ModA);
  \path[right hook->,font=\scriptsize]
    (knMod) edge (ModE)
    (knMod) edge (ModAE)
    (ModE) edge (ModAE)
    (PrMod) edge node[description]{$\Push$} (ModAsE);
  \end{tikzpicture}
\end{equation}
exactness follows by definition except surjectivity in the top
sequence. Given a sequence of loops
$(\gamma_d)_{d\in D}\in\pi_1(X,D)$, we may isotope them slightly to
obtain $\gamma_d(t)\neq\gamma_{d'}(t)$ for all $t\in[0,1]$ and all
$d\neq d'$. Then $b(d,t)\coloneqq\gamma_d(t)$ is a braid in
$S^2\setminus A$ and defines via `$\Push$' an element of
$\Braid(S^2\setminus A,D)$ mapping to $(\gamma_d)$.

\subsection{\boldmath $\wtf$-impure mapping class groups}
As in~\eqref{eq:forgetCD}, consider an orbisphere map
$\wtf\colon(S^2,C\sqcup E,\wtord)\to(S^2,A\sqcup
D,\wtord)$ that projects to
$f\colon(S^2,C,\ord)\to(S^2,A,\ord)$. We will enlarge the groups
in~\eqref{eq:4exact} to ``$\wtf$-impure mapping class groups'' so that
exact sequences analogous to~\eqref{eq:4exact:*} hold.

Let $\Mod^*(S^2,C\sqcup E)$ be the group of homeomorphisms
$m\colon (S^2,C\sqcup E)\selfmap$ considered up to isotopy rel
$C\sqcup E$ such that $m\restrict C$ is the identity and for every
$e\in E$ we have $\wtf(m(e))=\wtf(e)$; i.e.,~$m$ may permute points in
$\wtf^{-1}(\wtf(e))\setminus C$. There is a natural forgetful morphism
\[\Forget_E\colon \Mod^*(S^2,C\sqcup E)\to\Mod(S^2,C).\]

As in Definition~\ref{defn:PureBraidGroup}, the \emph{braid group
  $\Braid^*(S^2\setminus C,E)$} is the set of continuous motions
$m\colon[0,1]\times E\to S^2\setminus C$ considered up to isotopy so
that
\begin{itemize}
\item $m(t,-)\colon E\hookrightarrow S^2\setminus C$ is an inclusion
  for every $t\in [0,1]$;
\item $m(0,-)=\one\restrict E$;
\item $m(1,e)\in \wtf^{-1}(\wtf(e))$ for every $e\in E$.
\end{itemize}
Every $m\in\Braid^*(S^2\setminus C,E)$ induces a permutation
$\pi_m\colon e\mapsto m(1,e)$ of $E$. The pure braid group consists of
those permutations with $\pi_m=\one$. The product in
$\Braid^*(S^2\setminus C,E)$ is
$m\cdot m'=m\#(m'\circ(\one\times\pi_m))$, with as usual `$\#$'
standing for concatenation of motions. Birman's theorem (a slight
generalization of Theorem~\ref{thm:BirmExistUniq}) still holds: the
group $\Braid^*(S^2\setminus C,E)$ is isomorphic to the kernel of
$\Mod^*(S^2,C\sqcup E)\to\Mod(S^2,C)$ via the push operator.

Let $\pi_1^*(Y,E)$ be the group of motions
$m\colon[0,1]\times E\to Y=(S^2,C,\ord)$, considered up to homotopy,
such that
\begin{itemize}
\item $m(0,-)=\one\restrict E$;
\item $m(1,e)\in \wtf^{-1}(\wtf(e))$  for every $e\in E$;
\end{itemize}
here $m,m'\colon E\hookrightarrow Y$ are \emph{homotopic} if $m(-,e)$
and $m'(-,e)$ are homotopic curves (relative to their endpoints) in
$Y$ for all $e\in E$.  The product in $\pi_1^*(Y,E)$ is again
$m\cdot m'=m\#(m'\circ(\one\times\pi_m))$. We have
\begin{equation}\label{eq:decompose pi_1*}
  \pi_1^*(Y,E)\cong\prod_{e\in E}\pi_1(Y,e)\rtimes\prod_{d\in\wtf(E)}(\wtf^{-1}(d)\cap E)\perm,
\end{equation}
the isomorphism mapping $m$ to its restrictions $m(-,e)$ and its
permutation $\pi_m$.

\begin{lem}
  The following sequence is exact:
  \begin{equation}
    \begin{tikzcd}
      \knBraid(Y,E)\ar[r,hook] & \Braid^*(S^2\setminus C,E)\ar[r,"\Antipush_E",->>] & \pi_1^*(Y,E),
    \end{tikzcd}
  \end{equation}
  where $\Antipush_E$ is the natural forgetful morphism.
\end{lem}
\begin{proof}
  Suppose $\Antipush_E(m)=\Antipush_E(m')$. Then $m^{-1}m'$ is pure so
  $m^{-1}m'\in \knBraid(Y,E)$. The converse is also obvious.
\end{proof}

The same argument as in Lemma~\ref{lem:knBraid is normal} shows that
the knitting group $\knBraid(Y,E)$ is a normal subgroup of
${\Mod^*(S^2,C\sqcup E)}$. We may thus define
\begin{equation}\label{eq:modxd:*}
  \Mod^*(Y|E)\coloneqq\Mod^*(S^2,C\sqcup E)/\knBraid(Y,E).
\end{equation}

\noindent As in~\eqref{eq:4exact} we have the following exact sequences
\begin{equation}\label{eq:4exact:*}
  \begin{tikzpicture}[description/.style={fill=white,inner sep=2pt},baseline]
  \node (knMod) at (-4,3){$\knBraid(Y,E)$};
  \node (ModE) at (0,1.5) {$\Braid^*(S^2\setminus C,E)$};
  \node (PrMod) at (4,0) {$\pi_1^*(Y,E).$};
  \node (ModAE) at (-4/3,0){$\Mod^*(S^2,C\sqcup E)$};
  \node (ModAsE) at (0,-1.5){$\Mod^*(Y|E)$};
  \node (ModA)  at (-4,-3){\makebox[0pt][r]{$\Mod(H)={}$}$\Mod(S^2,C)$};
  \path[->>,font=\scriptsize] (ModE) edge node[description]{$\Antipush_E$} (PrMod)
    (ModAE) edge (ModAsE)
    (ModAE) edge node[description]{$\Forget_E$}  (ModA)
    (ModAsE) edge node[description]{$\Forget_E$} (ModA);
  \path[right hook->,font=\scriptsize]
    (knMod) edge (ModE)
    (knMod) edge (ModAE)
    (ModE) edge (ModAE)
    (PrMod) edge node[description]{$\Push$} (ModAsE);
  \end{tikzpicture}
\end{equation}

\subsection{Branched coverings}
Recall that, for orbisphere maps
$f_0,f_1\colon(S^2,C,\ord)\to(S^2,A,\ord)$ we write
$f_0\approx_C f_1$, and call them \emph{isotopic}, if there is a path
$(f_t)_{t\in[0,1]}$ of orbisphere maps
$f_t\colon(S^2,C,\ord)\to(S^2,A,\ord)$. Equivalently,
\begin{lem}\label{lem:factorization}
  $f_0\approx_C f_1$ if and only if $h f_0=f_1$ for a homeomorphism
  $h\colon(S^2,C)\selfmap$ that is isotopic to the
  identity.
\end{lem}
\begin{proof}
  If there exists an isotopy $(h_t)$ witnessing $\one\approx_C h$, then
  $f_t\coloneqq h_t f_0$ witnesses $f_0\approx_C f_1$. Conversely,
  since all critical values of $f_t$ are frozen in $A$, the set
  $f_t^{-1}(y)$ moves homeomorphically for every $y\in S^2$
  (equivalently, no critical points collide). Therefore, we may factor
  $f_t = h_t f_0$, with $h_t(z)$ the trajectory of $z\in f_t^{-1}(y)$;
  this defines an isotopy from $\one$ to $h\coloneqq h_1$.
\end{proof}

Consider orbisphere maps
$\wtf,\wtg\colon(S^2,C\sqcup E,\wtord)\to (S^2,A\sqcup D,\wtord)$
as in~\eqref{eq:forgetCD}. We write $\wtf\approx_{C|E}\wtg$,
and call $\wtf,\wtg$ \emph{knitting-equivalent}, if $\wtf=h\wtg$
for a homeomorphism $h\colon(S^2,C\sqcup E)\selfmap$ in
$\knBraid(Y,E)$; we have
\[\wtf\approx_{C\sqcup E}\wtg\Longrightarrow \wtf\approx_{C|E}\wtg\Longrightarrow \wtf\approx_{C}\wtg.\]
For $m\in\Braid(S^2\setminus A,D)$ we define its \emph{pullback}
$(\wtf)^*m\colon[0,1]\times E\to S^2\setminus C$ by
\[((\wtf)^*m)(-,e):= \begin{cases} m(-,\wtf(e))\lift{\wtf}{e} &\mbox{if } \wtf(e)\in D,\\
    e & \mbox{if } \wtf(e)  \in A.  \end{cases}
\]
This defines a motion of $E$; note that $(\wtf)^*m(1,e)$ need not equal
$e$:
\begin{lem}\label{lem:pushf*}
  If $\Push(m)\in \Braid(S^2\setminus A, D)$, then $\Push((\wtf)^*m)$
  defines an element of $\Braid^*(S^2\setminus C,E)$ and we have the
  following commutative diagram:
  \begin{equation}\label{eq:liftpush}
    \begin{tikzcd}[column sep=3cm]
      (S^2,C\sqcup E)\ar[r,"\Push((\wtf)^*m)"]\ar[d,"\wtf"']\ar[dr,phantom,"\approx_{C\sqcup E}"] & (S^2,C\sqcup E)\ar[d,"\wtf"]\\
    (S^2,A\sqcup D)\ar[r,"\Push(m)"'] & (S^2,A\sqcup D).
  \end{tikzcd}
\end{equation}
\end{lem}
\begin{proof}
  Let us discuss in more detail the operator `$\Push$'. Consider a
  simple arc $\gamma\colon [0,1]\to S^2\setminus A$ and let
  $\mathcal U\subset S^2\setminus A$ be a small disk neighborhood of
  $\gamma$. We can define (in a non-unique way) a homeomorphism
  $\Push(\gamma)\colon (S^2,A)\selfmap$ that maps $\gamma(0)$ to
  $\gamma(1)$ and is identity away from $\mathcal U$. Let
  $\mathcal U_1,\mathcal U_2,\dots,\mathcal U_d$ be the preimages of
  $\mathcal U$ under $\wtf$, where $d=\deg(\wtf)$. Each $\mathcal U_i$
  contains a preimage $\gamma_i$ of $\gamma$. Let
  $\Push(\gamma_i)\colon \mathcal U_i\selfmap$ be the lift of
  $\Push(\gamma)\restrict{\mathcal U}$ under
  $\wtf\colon \mathcal U_i\to\mathcal U$, extended by the identity on
  $S^2\setminus\mathcal U_i$. Then
  \begin{equation}\label{eq:list push gamma}
    \Push(\gamma_1)\Push(\gamma_2)\cdots\Push(\gamma_d)\cdot \wtf=\wtf\cdot \Push(\gamma).
  \end{equation}

  For $m\in\Braid(S^2\setminus A,D)$, we can define $\Push(m)$ as a
  composition of pushes along finitely many simple arcs
  $\beta_i$. Using~\eqref{eq:list push gamma} we lift all
  $\Push(\beta_i)$ through $\wtf$; considering the equation rel
  $C\sqcup E$ we obtain~\eqref{eq:liftpush}.
\end{proof}

We note that $\Push(m)$ does not necessarily lift to a `push' if $m$
moves a critical value. Indeed if $\gamma$ is a simple loop then
$\Push(\gamma)$ is the quotient of two Dehn twists about the boundary
curves of an annulus surrounding $\gamma$,
see~\cite{farb-margalit:mcg}*{\S4.2.2}; however the lift of $\gamma$
will not be a union of simple closed curves if $\gamma$ contains a
critical value; an annulus around $\gamma$ will not lift to an
annulus, but rather to a more complicated surface $F$; and the quotient
of Dehn twists about boundary components of $F$ will not be a quotient
of Dehn twists about boundary components of annuli.

\begin{prop}\label{prop:knittinginert}
  Let
  $\wtf\colon(S^2,C\sqcup E,\wtord)\to (S^2,A\sqcup
  D,\wtord)$ be an orbisphere map as
  in~\eqref{eq:forgetCD}. Then every element in $\knBraid(X,D)$ lifts
  through $\wtf$ to an element of $\knBraid(Y,E)$.

  If
  $\wtg\colon(S^2,C\sqcup E,\wtord)\to(S^2,A\sqcup
  D,\wtord)$ is another orbisphere map, then
  $\wtf\approx_{C|E} \wtg$ if and only if $\wtf\approx_{C} h \wtg k$
  for some $h\in\knBraid(Y,E),k\in\knBraid(X,D)$.
\end{prop}
\begin{proof}
  Consider $h\in \knBraid(X,D)$. By Theorem~\ref{thm:BirmExistUniq} we
  may write $h=\Push(b)$. Since $b(-,d)$ is homotopically trivial in
  $X$ for every $d\in D$, the curve $b(-,\wtf(e))\lift\wtf e$ ends at
  $e$ for all $e\in E$ with $\wtf(e)\in D$, because this curve is in
  $Y$. Therefore, $(\wtf)^*b(1,-)=\one\restrict E$, and
  $\Push((\wtf)^*b)\in \Braid(S^2\setminus C,E)$. Since the lifts
  $b(-,\wtf(e))\lift\wtf e$ are homotopically trivial in $Y$ for all
  $e\in E$, we have $\Push((\wtf)^*b)\in\knBraid(Y,E)$ and
  Lemma~\ref{lem:pushf*} concludes the first claim.

  The second claim is a direct consequence of the first.
\end{proof}

As a consequence, we may detail a little bit more the map
$\Antipush_D$ in~\eqref{eq:4exact}. Choose for every $d\in D$ a path
$\ell_d$ in $S^2\setminus A$ from $*$ to $d$. This path defines an
isomorphism $\pi_1(X,d)\to\pi_1(X,*)=G$ by
$\gamma\mapsto\ell_d\#\gamma\#\ell_d^{-1}$. We thus have a map
\begin{equation}\label{eq:antipush}
  \Antipush_D\colon\Braid(S^2\setminus A,D)\twoheadrightarrow G^D,\qquad m\mapsto(\ell_d\#(\Push^{-1}(m)\restrict d)\#\ell_d^{-1})_{d\in D},
\end{equation}
and $\ker(\Antipush_D)=\knBraid(X,D)$.

\subsection{Mapping class bisets}
\label{ss:MapClassBis}
We introduce some notation parallel to that in~\eqref{eq:4exact}
and~\eqref{eq:4exact:*} for mapping class bisets.  Let
$\Forget_E\colon\wtH\to H$ and $\Forget_D\colon\wtG\to G$ be forgetful
morphisms of orbisphere groups as in~\eqref{eq:forgetCD2}, and let
$\wtB$ be an orbisphere biset. Let
$B\coloneqq H\otimes_\wtH\wtB\otimes_\wtG G$ be the induced
$H$-$G$-biset. We have forgetful morphisms of groups
$\Forget_D\colon \Mod(\wtG)\to\Mod(G)$ and
$\Forget_E\colon \Mod(\wtH)\to\Mod(H)$. Corresponding mapping class
bisets are written respectively, with $\wtf$ and $f$ the orbisphere
maps associated with $\wtB$ and $B$,
\begin{align*}
  M(\wtB)=M(\wtf)&\coloneq \{n\otimes\wtB\otimes m\mid
                   n\in\Mod(\wtH),m\in\Mod(\wtG)\}/{\cong}\\
                   &=\{n  \wtf m\mid
                   n\in\Mod(S^2,\wtC),m\in\Mod(S^2,\wtA)\}/{\approx_{\wtC}},\\
  M(B)=M( f)&\coloneq\{n\otimes B\otimes m\mid  n\in\Mod(H),m\in\Mod(G)\}/{\cong}\\
  &=\{n f m\mid  n\in\Mod(S^2,C),m\in\Mod(S^2,A)\}/{\approx_C}
\end{align*}
together with the natural forgetful intertwiner
\begin{equation}\label{eq:dfn:ForgetE,D}
  \Forget_{E,D} \colon  M(\wtB) \to M( B),\hspace{0.5cm} \wtB'\to\Forget_{E,D}(B')=H\otimes_{\wtH}\wtB'\otimes_{\wtG}G.
\end{equation}

We may also define the following mapping class biset, sometimes larger
than $M(\wtB)$: assume first that $\wtord$ is constant on $E$,
possibly $\infty$, and set
\begin{equation}\label{eq:def M*}
  M^*(\wtB)=M^*(\wtf)\coloneqq\left\{\wtB'\text{ an $\wtH$-$\wtG$-orbisphere biset}\left|\begin{array}{c} \Forget_{E,D}(\wtB')\in M(B)\\ (\wtB')_*=\wtB_*\end{array}\right.\right\}/{\cong},
\end{equation}
where $(\wtB')_*\colon \wtC\to \wtA$ denotes the induced map on marked conjugacy classes.
It is an $\Mod^*(S^2,\wtC)$-$\Mod(S^2,\wtA)$-biset; note indeed that
we have $n \otimes \wtB'\in M^*(\wtB)$ for $\wtB'\in M^*(\wtB)$ and
$n\in\Mod^*(S^2,\wtC)$, because
$\Forget_{E,D}(n\otimes\wtB)=\Forget_E(n)\otimes B$. Again there is a
natural forgetful intertwiner
\begin{equation}\label{eq:dfn:ForgetE,D:*}
  \Forget_{E,D} \colon  \subscript{\Mod^*(S^2,\wtC)} M^*(\wtf)_{\Mod(S^2,\wtA)} \to\subscript{\Mod(S^2,C)}M(f)_{\Mod(S^2,A)}.
\end{equation}
We note that the left action of $\Mod(S^2,\wtC)$ on $M(\wtf)$
does not necessarily extend to an action of $\Mod^*(S^2,\wtC)$ on
$M(\wtf)$, because the result of the action is in general in $M^*(f)$
and not in $M(f)$, see Example~\ref{ex:M*neqM}.

In case $\wtord$ is not constant on $E$, we should be careful,
because permutation of points in $E$ does not leave $\wtH$ invariant;
rather, the image of $\wtH$ under such a permutation gives an
orbisphere group isomorphic to $\wtH$. However, $M^*(\wtB)$ and
$\Mod^*(S^2,\wtC)$ do not depend on the orbisphere structure, so the
definition may be applied with $\wtH$ and $\wtG$ replaced by orbisphere groups
with larger orders.

Let us call the set of extra marked points $E$ \emph{saturated} if
\[\wtf^{-1}(\wtf(E))\subseteq C\sqcup E.\]

\begin{lem}\label{lem:M* is left free}
  \begin{enumerate}
  \item The mapping class biset $M^*(\wtf)$ is left-free.

  \item Suppose that $E$ is saturated and that
    $g_0,g_1\colon (S^2,C\sqcup E,\wtord)\to(S^2,A\sqcup D,\wtord)$
    are orbisphere maps coinciding on $C\sqcup E$ and such that
    $\Forget_{E,D}(g_0)$ and $\Forget_{E,D}(g_1)$ are isotopic through
    maps $(S^2,C)\to(S^2,A)$. Then $g_0,g_1\in M^*(\wtf)$, and there
    is an $m\in \Braid^*(S^2\setminus C,E)$ such that $m g_0=g_1$
    holds in $M^*(\wtf)$.

  \item If $E$ is saturated, then
    \begin{equation}\label{eq:defn:M*}
      M^*(\wtf)= \{m \wtf n\mid m\in \Mod^*(S^2,C\sqcup E), n\in \Mod(S^2,A\sqcup D)\}/{\approx_\wtC}.
    \end{equation}
  \end{enumerate}
\end{lem}
\begin{proof}
  The proof of the first claim follows the lines
  of~\cite{bartholdi-dudko:bc2}*{Proposition 6.4}: suppose that
  $g,m g$ are isotopic through maps $(S^2,\wtC)\to(S^2,\wtA)$ for some
  $m\in \Mod^*(S^2,C\sqcup E)$ and $g\in M^*(\wtf)$. By
  Lemma~\ref{lem:factorization}, we may assume $g= m g$ as maps; then
  the homeomorphism $m$ is a deck transformation of the covering
  induced by $g$, so $m$ has finite order because $\deg
  (g)<\infty$. Recall that $\#C\ge 3$ by our standing
  assumption~\eqref{eq:StandAssump}. Since $m$ fixes at least $3$
  points in $C$ and $m$ has finite order, we deduce that $m$ is the
  identity. This shows that $M^*(\wtf) $ is left-free.

  For the second claim, let
  $(g_t\colon (S^2,C)\to (S^2,A))_{t\in[0,1]}$ be an isotopy between
  $g_0$ and $g_1$. By Lemma~\ref{lem:factorization}, we may write
  $g_t$ at $m_t g_0$ for $m_t\colon (S^2,C)\selfmap$. Then $m_1$
  preserves $E$ because $E$ is saturated, so
  $m_1\in \Braid^*(S^2\setminus C,E)$ as required.

  The third claim directly follows from the second.
\end{proof}

\begin{defn}[Extensions of bisets, see~\cite{bartholdi-dudko:bc2}*{Definition~\ref{bc2:defn:BisExt}}]\label{defn:ses}
 Let $\subscript{G_1}B_{G_2}$ be a $G_1$-$G_2$-biset and let $N_1,N_2$ be
  normal subgroups of $G_1$ and $G_2$ respectively, so that for
  $i=1,2$ we have short exact sequences
  \begin{equation}\label{eq:ExtGr12}
    \begin{tikzcd}
      1\ar[r] & N_i\ar[r] & G_i\ar[r,"\pi"] & Q_i\ar[r] & 1.
    \end{tikzcd}
  \end{equation}
  If the quotient $Q_1$-$Q_2$-biset $N_1\backslash B/N_2$, consisting
  of connected components of $\subscript{N_1}B_{N_2}$, is left-free, then the
  sequence
  \begin{equation}\label{eq:BisetExt}
    \begin{tikzcd}
      \subscript{N_1}B_{N_2}\ar[r,hook] & \subscript{G_1}B_{G_2}\ar[r,->>,"\pi"] & \subscript{Q_1}(N_1\backslash B/N_2)_{Q_2}
    \end{tikzcd}
  \end{equation}
  is called an \emph{extension of left-free bisets}.
\end{defn}

\begin{defn}[Inert biset morphism]\label{defn:inert}
  Let $\wtH\twoheadrightarrow H$ and $\wtG\twoheadrightarrow G$ be
  surjective group homomorphisms, and let $\wtB$ be a left-free
  $\wtH$-$\wtG$-biset. Recall that the tensor product
  $B\coloneqq H\otimes_{\wtH}\wtB\otimes_{\wtG}G$ is isomorphic to the
  double quotient $\ker(\wtH\to H)\backslash \wtB/\ker(\wtG\to G)$
  with natural $H$-$G$-actions.  The natural map
  $\Forget\colon \wtB\to B$ is called \emph{inert} if $B$ is a
  left-free biset and the natural map
  $\{\cdot\}\otimes_\wtG \wtB \to\{\cdot\}\otimes_G B$ is a
  bijection. In particular, $B$ has the same number of left orbits as
  $\wtB$. In other words, assuming that the groups $\wtH$ and $H$ have
  similarly-written generators and so do $\wtG$ and $G$, the wreath
  recursions of $\wtB$ and $B$ are identical.

  Yet said differently, in the extension
  $\subscript{\ker}{\wtB}_{\ker}\hookrightarrow\subscript{\wtH}{\wtB}_{\wtG}\twoheadrightarrow\subscript H B_G$ the kernel is
  a disjoint union of left-principal bisets. If
  $\wtG=\wtH$ and $G=H$ so that the bisets can be
  iterated, then $\wtB\to B$ is inert precisely when we have a
  factorization
  $\wtG\to G\to\IMG_{\wtG}(\wtB)$, the latter
  group being the quotient of $G$ by the kernel of the
  right action on the rooted tree
  $\{\cdot\}\otimes_H\bigsqcup_{n\ge0}B^{\otimes n}$,
  see~\cite{bartholdi-dudko:bc2}*{\S\ref{bc2:ss:inessential}}.
\end{defn}

\noindent Define
\begin{equation}\label{eq:defn:M*|ED}
M^*(B|E,D)\coloneqq \knBraid(Y,E)\backslash M^*(\wtB)/\knBraid(X,D);
\end{equation}
this is naturally a $\Mod^*(Y|E)$-$\Mod(X|D)$-biset, and
Proposition~\ref{prop:knittinginert} implies in particular that it is
left-free:
\begin{prop}\label{prop:inertMCB}
  The natural forgetful maps
  \begin{align*}
    \subscript{\Mod(\wtH)}{M(\wtB)}_{\Mod(\wtG)}
    &\to\subscript{\Mod(H|E)}{M(B|E,D)}_{\Mod(G|D)}\\
    \intertext{and}
    \subscript{\Mod^*(S^2,C\sqcup E)}{M^*(\wtB)}_{\Mod(S^2,A\sqcup D)}
    &\to\subscript{\Mod^*(Y|E)}{M^*(B|E,D)}_{\Mod(X|D)}
  \end{align*}
  are inert.\qed
\end{prop}

Let $E\perm^*$ denote the group of all permutations
$t\colon E\selfmap$ such that $\wtf(t(e))=\wtf(e)$. We denote by
$H^E\rtimes E\perm^*$ the semidirect product where $E\perm^*$ acts on
$H^E$ by permuting coordinates; compare with~\eqref{eq:decompose
  pi_1*}. We have $\pi_1^*(Y,E)\cong H^E\rtimes E\perm^*$.

We denote by
$\subscript{\Braid^*(Y,E)}{M^*(\wtB)}_{\Braid(X,D)}$ and
$\subscript{\pi^*_1(Y,E)}{M^*(B|E,D)}_{\pi_1(X,D)}$ the restrictions
of $M^*(\wtB)$ and $M^*(B|E,D)$ to braid and fundamental
groups.

\begin{thm}\label{thm:4Ext of MCB}
  If $E$ is saturated, then the following sequences are extensions of
  bisets:
  \begin{equation}\label{eq:bis4exact}
    \begin{tikzpicture}[description/.style={fill=white,inner sep=2pt},baseline]
      \node (knMod) at (-4,3){$\subscript{\knBraid(Y,E)}{M^*(\wtB)}_{\knBraid(X,D)}$};
      \node (ModE) at (0,1.5) {$\subscript{\Braid^*(Y,E)}{M^*(\wtB)}_{\Braid(X,D)}$};
      \node (PrMod) at (4,0) {$\subscript{\pi^*_1(Y,E)}{M^*(B|E,D)}_{\pi_1(X,D)}$};
     \node (Portr) at (4,-2.5) {$\subscript{H^E\rtimes E\perm^*}{\Big\{\begin{array}{c}(B'\in M(B),(B'_c)_{c\in\wtC})\\ (B'_c)_{c\in\wtC}\text{ is a portrait in }B'\end{array}\Big\}}_{G^D}$.};
      \node (ModAE) at (-4/3,0){$M^*(\wtB)$};
      \node (ModAsE) at (0,-1.5){$M^*(B|E,D)$};
      \node (ModA)  at (-4,-3){$M(B)$};
      \path[->>,font=\scriptsize] (ModE) edge node[description]{$\Antipush_{E,D}$} (PrMod)
      (ModAE) edge (ModAsE)
      (ModAE) edge node[description]{$\Forget_{E,D}$}  (ModA)
      (ModAsE) edge node[description]{$\Forget_{E,D}$} (ModA);
      \path[right hook->,font=\scriptsize]
      (knMod) edge (ModE)
      (knMod) edge (ModAE)
      (ModE) edge (ModAE)
      (PrMod) edge (ModAsE);
      \path (PrMod) edge[->] node[description,pos=0.4]{$\mathcal P$} node[right,pos=0.6]{$\cong$} (Portr);
    \end{tikzpicture}
  \end{equation}
  (The isomorphism on the right is the topic of
  Theorem~\ref{thm:portrait}, and will be proven there.)
\end{thm}
\begin{proof}
  By Lemma~\ref{lem:M* is left free}(2), $\Forget_{E,D}^{-1}(B)$ is a
  connected subbiset of
  $\subscript{\Braid^*(Y,E)}{M^*(\wtB)}_{\Braid(X,D)}$; thus
  the central-to-left sequence is an extension of bisets. Exactness of
  other sequences follows from Proposition~\ref{prop:inertMCB}.
\end{proof}

Note that, if $E$ were not saturated or if we replaced
$M^*(\wtB)$ by $M(\wtB)$ in~\eqref{eq:bis4exact}, then
we wouldn't have exact sequences of bisets anymore, because the fibres
of $M^*(\wtB)\twoheadrightarrow M(B)$ wouldn't have to be
connected; see Example~\ref{ex:notexact}. The failure of transitivity
is described by Lemma~\ref{lem:twisting bar E}. There are similar
exact sequences in case $f_*\colon E\to D$ is a bijection, see
Theorem~\ref{thm:extension}.

\subsection{Portraits of bisets}
First, a \emph{portrait of groups} amounts to a choice of
representative in each peripheral conjugacy class of an orbisphere
group:
\begin{defn}[Portraits of groups]
  Let $G$ be an orbisphere group with marked conjugacy classes
  $(\Gamma_a)_{a\in A}$ and let $\wtA$ be a finite set
  containing $A$. A \emph{portrait of groups}
  $(G_a)_{a\in\wtA}$ in $G$ is a collection of cyclic
  subgroups $G_a\le G$ so that
  \[G_a=\begin{cases} \langle g \rangle &\mbox{for some } g\in \Gamma_a, \mbox{ if } a\in A, \\
      \langle1 \rangle & \mbox{otherwise.}\end{cases}
  \]
  If $\wtA=A$, then $(G_a)_{a\in A}$ is called a
  \emph{minimal} portrait of groups.
\end{defn}

\begin{defn}[Portraits of bisets]\label{dfn:PrtrOfBst}
  Let $H,G$ be orbisphere groups with peripheral classes indexed by
  $C,A$ respectively, let $\subscript H B_G$ be an orbisphere biset,
  and let $f_*\colon C\to A$ be its portrait. We also have a
  ``degree'' map $\deg\colon C\to\N$.  A \emph{portrait of bisets}
  relative to these data consists of
  \begin{itemize}
  \item portraits of groups $(H_c)_{c\in\wtC}$ in $H$ and
    $(G_a)_{a\in\wtA}$ in $G$;
  \item extensions $f_*\colon\wtC\to\wtA$ of
    $f_*$ and $\deg\colon\wtC\to\N$ of $\deg$;
  \item a collection $(B_c)_{c\in\wtC}$ of subbisets of $B$
    such that every $B_c$ is an $H_c$-$G_{f_*(c)}$-biset that
    is right-transitive and left-free of degree $\deg(c)$, and such
    that if $f_*(c)=f_*(c')$ and $H B_c=H B_{c'}$
    qua subsets of $B$ then $c=c'$.
  \end{itemize}
  If $\wtC=C$ and $\wtA=A$, then $(B_c)_{c\in C}$ is
  called a \emph{minimal} portrait of bisets.
\end{defn}
Note in particular that if $c\in\wtC\setminus C$ then $H_c$ is
trivial and the subbiset $B_c$ consists of $\deg(c)$ elements permuted
by $G_{f_*(c)}$. If moreover $f_*(c)\notin A$, then $\deg(c)=1$. For
simplicity the portrait of bisets is sometimes simply written
$(B_c)_{c\in\wtC}$, the other data
$f_*,\deg,(H_c)_{c\in\wtC}$, $(G_a)_{a\in\wtA}$ being
implicit.

Here is a brief motivation for the definition. Consider an expanding
Thurston map $f$ and its associated biset $B$. Recall
from~\cite{nekrashevych:ssg} that bounded sequences in $B$ represent
points in the Julia set of $f$; in particular constant sequences
represent fixed points of $f$ and vice versa. On the other hand, fixed
Fatou components of $f$ are represented by local subbisets of $B$ with
same cyclic group acting on the left and on the right. All of these
are instances of portraits of bisets in $B$. Furthermore,
(pre)periodic Julia or Fatou points are represented by portraits of
bisets with (pre)periodic map $f_*$. E.g., closures of Fatou
components intersect if and only if they admit portraits that
intersect; and similarly for inclusion in the closure of a Fatou
component of a (pre)periodic point in the Julia set.

Let $\Forget_D\colon\wtG\to G$ be a marked forgetful morphism
of orbisphere groups as in~\eqref{eq:forgetCD2}. For all $a\in\wtA$
choose a small disk neighbourhood $\mathcal U_a\ni a$ that avoids all
other points in $\wtA$, so that
$\pi_1(\mathcal U_a\setminus\wtA)\cong\Z$. We call a curve
$\gamma$ \emph{close to $a$} if $\gamma\subset\mathcal U_a$.

\begin{lem}\label{lmm:connect:ell_a}
  Every minimal portrait of groups
  $(\wtG_a)_{a\in\wtA}$ in $\wtG$ is uniquely
  characterized by a family of paths $(\ell_a)_{a\in\wtA}$
  with
  \begin{equation}\label{eq:lmm:connect:ell_a1}
    \ell_a\colon [0,1]\to S^2\text{ rel }\wtA,\qquad
    \ell_a(t)\notin\wtA \text{ for } t<1, \ell_a(0)=* \text{ and }\ell_a(1)=a,
  \end{equation}
  so that for any sufficiently small $\epsilon>0$
  \begin{equation}\label{eq:lmm:connect:ell_a11}
    \wtG_a= \left\{\ell_a\restrict{[0,1-\epsilon]}\#\alpha_a \#(\ell_{a}\restrict{[0,1-\epsilon]})^{-1}\mid\alpha_a \text{ is close to }a\right\}\text{ rel } \wtA.
  \end{equation}
  Conversely, every collection of curves~\eqref{eq:lmm:connect:ell_a1}
  defines by~\eqref{eq:lmm:connect:ell_a11} a minimal portrait of
  groups for every sufficiently small $\epsilon>0$.
\end{lem}
\begin{proof}
  This follows immediately from the definition of ``lollipop''
  generators, see~\eqref{eq:orbispheregp}.
\end{proof}
It follows that every $\wtG_a$ is self-normalizing in
$\wtG$: conjugating $\wtG_a$ by an element
$g\not\in\wtG_a$ amounts to precomposing the path $\ell_a$
with $g$, resulting in a different path.

\begin{lem}\label{lem:CritPortrUnique}
  Let $\subscript H B_G$ be an orbisphere biset. Then every pair of
  minimal portraits of groups $(H_c)_{c\in C}$ in $H$ and
  $(G_a)_{a\in A}$ in $G$ can be uniquely completed to a minimal
  portrait of bisets $(B_c)_{c\in C}$ in $B$.
\end{lem}
As a consequence, the minimal portrait of bisets is unique up to
conjugacy.  We note that the lemma is also true in case
$\subscript H B_G$ is a cyclic biset, namely if $H,G\cong\Z$; in this
case all $B_c$ are equal to $B$.
\begin{proof}
  Since by Assumption~\eqref{eq:StandAssump} the sets $A$ and $C$
  contain at least $3$ elements, $H$ and $G$ are non-cyclic, and in
  particular have trivial centre.

  Let $f_*\colon C\to A$ be the portrait of $B$. Choose generators
  $g_a\in G_a$ and $h_c\in H_c$.  Recall
  from~\cite{bartholdi-dudko:bc1}*{\S2.6}
  or~\cite{bartholdi-dudko:bc2}*{Definition~\ref{bc2:defn:lifts}} that
  there are $b_c\in B$ and $k_c \in H$ such that
  $k_c^{-1}h_c k_c b_c = b_c g_{f(c)}^{\deg(c)}$ for all $c\in C$. Set
  then $B_c\coloneqq H_c k_c b_c G_{f(c)}$, and note that these
  subbisets satisfy Definition~\ref{dfn:PrtrOfBst}.

  Suppose next that $(B'_c)_{c\in C}$ is another portrait of bisets,
  and choose elements $b'_c\in B'_c$. Then
  by~\cite{bartholdi-dudko:bc2}*{Definition~\ref{bc2:dfn:SphBis}\eqref{bc2:cond:3:dfn:SphBis}}
  the conjugacy class $\Delta_c$ appears exactly once among the lifts
  of $\Gamma_{f(c)}$, so $H b'_c\cap B_c\neq\emptyset$, so we may
  choose $k_c\in H$ with $k_c b'_c\in B_c$. Then
  $k_c B'_c = k_c b'_c G_{f(c)} = B_c G_{f(c)} = B_c$. We have
  $k_c H_c= H_c k_c$, so $k_c\in H_c$ because $H_c$ is
  self-normalizing in $H$, and therefore $B_c=B'_c$.
\end{proof}

Consider next a forgetful morphism
$\Forget_{E,D}\colon\subscript{\wtH}{\wtB}_{\wtG} \to\subscript H B_G$ of orbisphere bisets, and let
$(\wtB_c)_{c\in\wtC}$ be the minimal portrait of
bisets given by Lemma~\ref{lem:CritPortrUnique}.
\begin{lem}
\label{lmm:connect:ell_a:2}
  Let $(m_c)_{c\in\wtC}$ and $(\ell_a)_{a\in\wtA}$ be
  the paths (see Lemma~\ref{lmm:connect:ell_a}) associated with the respective portraits of bisets
  $(\wtH_c)_{c\in\wtC}$ and
  $(\wtG_a)_{a\in\wtA}$.  The portrait
  $(\Forget_{E,D}(\wtB_c))_{c\in\wtC}$ admits the
  following description: for every sufficiently small $\epsilon>0$,
  \begin{equation}\label{eq:lmm:connect:ell_a2}
    B_c = \left\{m_c\restrict{[0,1-\epsilon]}\#\beta_c \#(\ell_{f_*(c)}\restrict{[0,1-\epsilon]})^{-1} \lift{f}{\beta_c(1)}\mid\beta_c \text{ is close to }c\right\}\text{ rel } C.
  \end{equation}
\end{lem}
\begin{proof}
  Since $B_c$ is obtained from $\wtB_c$ by the intertwiner
  $\Forget_{E,D}$ (see~\eqref{eq:forgetbi}), it suffices to consider
  the case $E=D=\emptyset$; and in that case, by
  Lemma~\ref{lem:CritPortrUnique} it suffices to note that $B_c$ is
  indeed an $H_c$-$G_{f_*(c)}$-biset.
\end{proof}

Let $\subscript H B_G$ be an orbisphere biset; let
$f_*\colon \wtC\to\wtA$ be an abstract portrait extending $B_*$; let
$\deg\colon \wtC\to\N$ be an extension of $\deg_B\colon C\to\N$; and
let $(B_c)_{c\in \wtC}$ be a portrait of bisets in $B$ with portraits
of groups $(H_c)_{c\in \wtC}$ in $H$ and $(G_a)_{a\in \wtA}$ in $G$.
A \emph{congruence of portraits} is defined by a choice of
$(h_c)_{c\in \wtC}$ in $H$ and $(g_a)_{a\in \wtA}$ in $G$, and
modifies the portrait of bisets by replacing
\[H_c\rightsquigarrow h_c^{-1} H_c h_c,\quad B_c\rightsquigarrow
  h_c^{-1} B_c g_{f_*(c)},\quad G_a\rightsquigarrow g_a^{-1} G_a g_a.
\]
By Lemma~\ref{lem:CritPortrUnique}, any two minimal portraits of
bisets are congruent.

\subsection{Main result}
Consider an orbisphere map $f\colon (S^2,C,\ord)\to (S^2,A,\ord)$. We
call it \emph{compatible} with
$\Forget_E\colon (S^2,\wtC,\wtord)\dashrightarrow (S^2,C,\ord)$ and
$\Forget_D\colon (S^2,\wtA,\wtord)\dashrightarrow (S^2,A,\ord)$ and
$f_*\colon \wtC\to\wtA$ and $\deg\colon \wtC\to\N$ if $\deg(e)=1$ for
all $e\in E$ with $f_*(e)\in D$, and
$\wtord(c)\deg(c)\mid \wtord(f_*(c))$ for all $c\in \wtC$, and
$\{\deg(f@f^{-1}(a))\}=\{\deg(f_*^{-1}(a)\}$ for all $a\in A$.
Equivalently, there is an orbisphere map
$\wtf\colon(S^2,\wtC,\wtord)\to(S^2,\wtA,\wtord)$ that can be isotoped
within maps $(S^2,C,\ord)\to(S^2,A,\ord)$ to a map
making~\eqref{eq:forgetCD} commute and such that
$\deg\colon \wtC\to \N$ and $f_*\colon \wtC\to \N$ are induced by
$\wtf$.

Compatibility of an orbisphere biset $\subscript H B_G$ with
$\Forget_E\colon\wtH\to H$, $\Forget_D\colon\wtG\to G$,
$f_*\colon \wtC\to\wtA$, and $\deg\colon \wtC\to\N$ is defined
similarly, and is equivalent to the existence of a biset $\wtB$
making~\eqref{eq:forgetCD2} commute.

We are now ready to relate the mapping class bisets $M^*$ to portraits
of bisets:
\begin{thm}\label{thm:portrait}
  Let $\Forget_E\colon\wtH\to H$ and
  $\Forget_D\colon\wtG\to G$ be forgetful morphisms of
  orbisphere groups as in~\eqref{eq:forgetCD2}, and let
  $\subscript H B_G$ be an orbisphere biset compatible with $\Forget_E$, $\Forget_D$, $f_*\colon \wtC\to\wtA$, and $\deg\colon \wtC\to\N$.

  Then for every portrait of bisets $(B_c)_{c\in \wtC}$ in $B$
  parameterized by $f_*$ and $\deg$ there exists an orbisphere biset
  $\subscript{\wtH}{\wtB}_{\wtG}$ mapping to $B$ under $\Forget_{E,D}$
  with a minimal portrait mapping to $(B_c)_{c\in \wtC}$.

  Furthermore, $\wtB$ is unique up to pre- and
  post-composition with bisets of knitting mapping classes and we have
  a congruence of bisets (see~\eqref{eq:antipush})
  \begin{equation}\label{eq:thm:portrait:congr}
    \begin{tikzpicture}[description/.style={fill=white,inner sep=2pt},baseline]
      \node(A) at (0,0){$\subscript{\pi_1^*(Y,E)}{M^*(B|E,D)}_{\pi_1(X,D)}$};
      \node(B) at(0,-2){$\subscript{H^E\rtimes E\perm^*}{\big\{(B'\in M(B),(B'_c)_{c\in\wtC})\mid B'\in M(B),(B'_c)_{c\in\wtC}\text{ a portrait in }B'\big\}}_{G^D}$};
      \draw (A) edge[->]node[description]{$\mathcal P$} (B);
    \end{tikzpicture}
  \end{equation}
  given by
  $\mathcal P(\wtB')=(\Forget_{E,D}(\wtB'),\text{induced portrait of }\wtB')$.
 \end{thm}
 \noindent The $H^E$-$G^D$-action on $\{(B',(B'_c)_{c\in\wtC})\}$ is
 given by
\[(h_c)_{c\in E}(B',(B'_c)_{c\in \wtC})(g_d)_{d\in D}=(B',(h_c
  B'_c g_{f_*(c)})_{c\in \wtC}),
\]
with the understanding that $h_c=1$ if $c\not\in E$ and $g_{f_*(c)}=1$
if $f_*(c)\not\in D$, and the action of $E\perm^*$ is given by
permutation of the portrait of bisets.

In the dynamical situation (i.e.~when $H=G$ and
$\wtH=\wtG$), Theorem~\ref{thm:portrait} proves
Theorem~\ref{main:thm:A}.

\begin{proof}[Proof of Theorem~\ref{thm:portrait}]
  Clearly,~\eqref{eq:thm:portrait:congr} is an intertwiner: firstly,
  the actions of $E\perm^*$ are compatible; we may ignore them in the
  sequel. Secondly, let $(B'_c)_{c\in \wtC}$ be
  the induced portrait of bisets of $\wtB'$.  For $e\in E$ we
  may write
  \[B'_e=\{m_e\#p^{-1}\mid p\colon[0,1]\to Y,p^{-1}(0)=e,f\circ p=\ell_{f_*(e)}\},
  \]
  see~\eqref{eq:lmm:connect:ell_a2}. Consider $\psi \in \Mod(Y|E)$;
  the action of $\psi$ on $B'_e$ replaces $m_e$ by $\psi\circ m_e$;
  and if $\Antipush_E(\psi)=(h_e)_{e\in E}$ then $\psi\circ m_e=h_e m_e$
  by the very definition of `$\Push$' and~\eqref{eq:antipush}. The same
  argument applies to the right action.

  Let us now show that $\mathcal P$ is a congruence. Since
  $\mathcal P$ is an intertwiner between left-free bisets with
  isomorphic acting groups, it is sufficient to show that $\mathcal P$
  preserves left orbits.

  We proceed by adding new points to $D$ and $E$. If $E=\emptyset$,
  then the forgetful maps $M^*(B|E,D)\to M(B)$ and
  \[\big\{(B'\in M(B),(B'_c)_{c\in\wtC})\mid B'\in M(B),(B'_c)_{c\in\wtC}\text{ a portrait in }B'\big\}\to M(B)\]
  are bijections and the claim follows. Therefore, it is sufficient to
  assume that $D=\emptyset$.

  By~\cite{bartholdi-dudko:bc2}*{Theorem~\ref{bc2:thm:dehn-nielsen-baer++}},
  the biset $B$ may be written as $B(f)$ for a branched covering
  $f\colon(S^2,C,\ord)\to(S^2,A,\ord)$, unique up to isotopy rel
  $C$. We lift $f$ to a branched covering
  $f^+\colon(S^2,f^{-1}(A))\to(S^2, A)$. Let $\wtB^+$ be its biset and
  let $(\wtB_c^+)_{c\in f^{-1}(A)}$ be its minimal portrait of bisets,
  which is unique by Lemma~\ref{lem:CritPortrUnique}.

  Let us show that for every portrait of bisets
  $(B_c)_{c\in\wtC}$ in $B$ there is an orbisphere biset
  $\wtB$ whose minimal portrait of bisets maps to
  $(B_c)$. This $\wtB$ will be of the form
  $B(m)\otimes \Forget_{f^{-1}(A)\setminus\wtC,\emptyset}(\wtB^+)$ for a homeomorphism
  $m\colon (S^2,f^{-1}(A))\selfmap$.

  First, consider the images of all $\wtB_c^+$ in
  $\{\cdot\}\otimes_\wtH\wtB^+\cong\{\cdot\}\otimes_H B$, and compare
  them to the images of all $\wtB_c$. The condition that as $c'$
  ranges over $f_*^{-1}(f_*(c))$ the $B_{c'}$ lie in different
  $H$-orbits of $B$ lets us select which preimages of $A$ correspond
  to the marked points in $\wtC$, and thus produces a well-defined map
  $\wtC\to f^{-1}(A)$. Let $m'$ be an isotopy of $(S^2,C)$ which maps
  $\wtC$ to $f^{-1}(A)$ in the specified manner; then
  $m' f^+\colon(S^2,\wtC,\ord)\to(S^2,A,\ord )$ has biset $\wtB^0$ and
  portrait of bisets $(\wtB^0_c)_{c\in\wtC}$; and
  $\Forget_{E,\emptyset}(\wtB^0_c)\subseteq H B_c$ for all $c\in\wtC$,
  so we may write $h_c\Forget_{E,\emptyset}(\wtB^0_c)=B_c$ for some
  $h_c\in H$. (We recall that $B_c$ consists of $d(c)$ elements
  permuted by $G_{f_*(c)}$, where $d(c)$ is the local degree of $f$ at
  $c$. We have $h_c\Forget_{E,\emptyset}(\wtB^0_c)=B_c$ if and only if
  $h_c\Forget_{E,\emptyset}(\wtB^0_c)\cap B_c\neq\emptyset$.)  We set
  $\wtB=\prod_{c\in E}\Push(h_c)\wtB^0$, and note that the minimal
  portrait of bisets of $\wtB$ maps to $(B_c)_{c\in\wtC}$ under
  $\Forget_{E,\emptyset}$.

  The only choice involved is that of a mapping class that yields
  $\Push(h_c)$ when restricted to $C\cup\{c\}$, namely that of
  knitting mapping classes.
\end{proof}

\subsection{Fibre bisets}
Consider an orbisphere map
$\wtf\colon (S^2, C\sqcup E,\wtord)\to (S^2, A\sqcup D ,\wtord)$ as above and define
the saturation of $E$ as
\[\widebar E\coloneqq \bigsqcup_{e\in E}\wtf^{-1}\big(\wtf(e)\big)\setminus C.\]

\begin{lem}\label{lem:twisting bar E}
  Let
  $\wtf\colon (S^2,C\sqcup E,\wtord)\to (S^2,A\sqcup
  D,\wtord)$ be an orbisphere map and let
  $m\colon (S^2,C\sqcup \widebar E)\selfmap$ be a homeomorphism such
  that $m\restrict C=\one$ and for every $e\in \widebar E$ we have
  $\wtf(m(e))=\wtf(e)$. If the isotopy class of $m$ is not in
  $\Mod^*(S^2,C\sqcup E)$, namely if $m$ moves at least one point in
  $E$ to $\widebar E\setminus E$, then
  $m \wtf\not\approx_{C\sqcup E}\wtf$.

  For every $\wtg\in M^*(\wtf)$ there is a homeomorphism
  $m\colon (S^2,C\sqcup \widebar E)\selfmap$ as above
  (i.e.~$m\restrict C=\one$ and $\wtf(m(e))=\wtf(e)$ for
  $e\in \widebar E$) such that $\wtg \approx_{C\sqcup E} m\wtf$.
\end{lem}

Note that $\Mod^*(S^2,C\sqcup \widebar E)$ does \emph{not} act on
$M^*(\wtf)$: there are orbisphere maps
$\wtf_1,\wtf_2\colon (S^2, C\sqcup E,\wtord)\to (S^2, A\sqcup
D,\wtord)$ such that $\wtf_1\approx_{C\sqcup E} \wtf_2$ but
$m \wtf_1\not\approx_{C\sqcup E} m \wtf_2$ for a homeomorphism
$m\colon (S^2,C\sqcup \widebar E)\selfmap$ as above.

\begin{proof}
  Suppose $m \wtf\approx_{C\sqcup E} \wtf$. Since $\widebar E$ is
  saturated, by Lemma~\ref{lem:M* is left free}(2) there is an
  $n\colon (S^2,C\sqcup \widebar E)\selfmap$ with
  $n\restrict {C\sqcup E}=\one$ and $\wtf(n(e))=\wtf(e)$ for
  $e\in \widebar E\setminus E$ such that
  $n m \wtf\approx_{C\sqcup \widebar E} \wtf$. This contradicts
  Lemma~\ref{lem:M* is left free}(1): the biset
  \begin{equation}\label{eq:prf:lem:twisting bar E}
    M^*(\wtf \colon (S^2, C\sqcup \widebar E)\to (S^2, A\sqcup D))
  \end{equation}
  is left-free while $n m\neq\one$.

  The second claim follows from Lemma~\ref{lem:M* is left free}(2)
  applied to the biset~\eqref{eq:prf:lem:twisting bar E}.
\end{proof}

We are interested in the fibre biset
$\subscript{H^E}{ \Forget_{E,D}^{-1}(B')}_{G^D}$ under the forgetful
map $\Forget_{E,D}\colon M^*(B|E,D)\to M(B)$. For every
$a\in \wtA$ define $E_a\coloneqq E\cap f_*^{-1}(a)$, where
$f_*\colon \wtC\to\wtA$ is the portrait.

\begin{prop}\label{prop:FibBiset}
  We have
  \begin{equation}\label{eq:prop:FibBiset}
    \subscript{H^E}{ \Forget_{E,D}^{-1}(B')}_{G^D}\cong \prod_{a\in A} \subscript{H^{E_a}} {M^*( B'| E_a,\emptyset )}_1\times \prod_{d\in D}\subscript{H^{E_d}} {M^*( B'| E_d,\{d\})}_{G^{\{d\}}}.
  \end{equation}

  Suppose that $B'$ is the biset of $g\colon (S^2,C,\ord)\to (S^2,A,\ord)$; then
  \[\subscript{H^{E_a}} {M^*( B'| E_a,\emptyset )}_1\cong H^{E_a}\times\{\iota\colon E_a\hookrightarrow  g^{-1}(a)\setminus C\}.\]

  The biset
  $\subscript{H^{E_d}} {M^*( B'| E_d,\{d\})}_{G^{\{d\}}}$ is
  congruent to the biset
  \[\{(b_e)_{e\in E_d}\in B'^{E_d}\mid H b_e\neq H b_{e'}\text{ if }e\neq e'\}\]
  endowed with the actions
 \[(h_e)_{e\in E_d}\cdot (b_e)_{e\in E_d}\cdot g= (h_e b_e g)_{e\in E_d}.\]
\end{prop}
\begin{proof}
  By Theorem~\ref{thm:portrait}, the biset
  $\subscript{H^E}{ \Forget_{E,D}^{-1}(B')}_{G^D}$ is isomorphic to
  the biset of portraits $(B_c)_{c\in \wtC}$ in $B'$. Recall
  that $(B_c)_{c\in C}$ is fixed (by Lemma~\ref{lem:CritPortrUnique})
  while $(B_e)_{e\in E}$ varies; this
  shows~\eqref{eq:prop:FibBiset}.

  The second claim follows from Lemma~\ref{lem:twisting bar E}.

  For $d\in D$ and $e\in E_d$ we have $B_e=\{b_e\}$; i.e.~the choice
  of $(B_e)_{e\in E_d}$ is equivalent to the choice of
  $(b_e)_{e\in E_d}\in B'^{E_d}$ subject to the condition stated above.
\end{proof}

Note that $\subscript{H^{E_a}} {M^*( B'| E_a,\emptyset )}_1$ will not
be transitive, as soon as there are at least two maps
$\iota\colon E_a\to g^{-1}(a)\setminus C$.

\noindent Let us define
\[M(B|E,D)\coloneqq \Mod(Y|E)\otimes M(\wtB)\otimes\Mod(X|D).\]
\begin{thm}\label{thm:extension}
  Suppose $f_*(E)\subset D$ and, moreover, that $f_*\colon E\to D$ is
  a bijection. We then have a congruence
  \begin{equation}\label{eq:1:thm:extension}
    \subscript{H^E}{\big(\Forget_{E,D}^{-1}(B')\big)}_{G^D} \to (\subscript{H}{B'}_G)^E
  \end{equation}
  mapping the portrait $(\wtB_c)_{e\in \wtC}$ in $B'$ to
  $(b_e)_{e\in E}$ where $\wtB_e=\{b_e\}$. The group $G^D$ is
  identified with $G^E$ via the bijection $f_*\colon E\to D$.

  Moreover, $M^*(\wtB)=M(\wtB)$, $M^*(B|E,D)=M(B|E,D)$, and exact
  sequences similar to~\eqref{eq:bis4exact} hold:
  \begin{equation}\label{eq:bis4exact:v2}
    \begin{tikzpicture}[description/.style={fill=white,inner sep=2pt},baseline]
      \node (knMod) at (-4,3){$\subscript{\knBraid(Y,E)}{M(\wtB)}_{\knBraid(X,D)}$};
      \node (ModE) at (0,1.5) {$\subscript{\Braid(Y,E)}{M(\wtB)}_{\Braid(X,D)}$};
      \node (PrMod) at (4,0) {$\subscript{\pi_1(Y,E)}{M(B|E,D)}_{\pi_1(X,D)}$};
     \node (Portr) at (4,-2.5) {$\subscript{H^E}{\Big\{\begin{array}{c}(B'\in M(B),(B'_c)_{c\in\wtC})\\ (B'_c)_{c\in\wtC}\text{ is a portrait in }B'\end{array}\Big\}}_{G^D}$.};
      \node (ModAE) at (-4/3,0){$M(\wtB)$};
      \node (ModAsE) at (0,-1.5){$M(B|E,D)$};
      \node (ModA)  at (-4,-3){$M(B)$};
      \path[->>,font=\scriptsize] (ModE) edge node[description]{$\Antipush_{E,D}$} (PrMod)
      (ModAE) edge (ModAsE)
      (ModAE) edge node[description]{$\Forget_{E,D}$}  (ModA)
      (ModAsE) edge node[description]{$\Forget_{E,D}$} (ModA);
      \path[right hook->,font=\scriptsize]
      (knMod) edge (ModE)
      (knMod) edge (ModAE)
      (ModE) edge (ModAE)
      (PrMod) edge (ModAsE);
      \path (PrMod) edge[->] node[description,pos=0.4]{$\mathcal P$} node[right,pos=0.6]{$\cong$} (Portr);
    \end{tikzpicture}
  \end{equation}

  The bottom sequence in~\eqref{eq:bis4exact:v2} can be written
  (using~\eqref{eq:1:thm:extension}) as
  \begin{equation}\label{eq:2:thm:extension}
    \begin{tikzcd}
      \bigsqcup_{B'\in B}\left(B'\right)^E\ar[hook,r] & M(B|E,D)\ar[->>,r] & M(B).
    \end{tikzcd}
  \end{equation}
\end{thm}
\begin{proof}
  The first claim follows from Proposition~\ref{prop:FibBiset}. Since
  $(B')^E$ is a transitive biset, we obtain $M(B|E,D)=M^*(B|E,D)$
  and the exact sequences~\eqref{eq:2:thm:extension} and~\eqref{eq:bis4exact:v2} hold because the fibres are connected.
\end{proof}

\begin{cor}\label{cor:thm:extension}
  Let $f\colon(S^2,A,\ord)\selfmap $ be an orbisphere map and let
  $D=\bigsqcup_{i\in I} D_i$ be a finite union of periodic cycles
  $D_i$ of $f$. Then~\eqref{eq:2:thm:extension} takes the form
  \begin{equation}\label{eq:cor:thm:extension}
    \bigsqcup_{g\in M(f)}\bigsqcup_{i\in I} B(g)^{\otimes(\#D_i)}\hookrightarrow  M(f|D,D) \twoheadrightarrow M(f).\qed
  \end{equation}
\end{cor}

\subsection{Centralizers of portraits}\label{ss:diff centr}
Let us now consider the dynamical situation: $H=G$ and $\wtG=\wtH$; we
abbreviate $\Forget_{D}=\Forget_{D,D}$ and $M(B|D)=M(B|D,D)$.

Given an orbisphere biset $\wtB$ we denote by
$Z(\wtB)\le\Mod(\wtG)$ the centralizer of
$\wtB$ in $M(\wtB)$ and by
$Z(\wtB|D)\le\Mod(G|D)$ the centralizer of the image of
$\wtB$ in $M(B|D)$, see Theorem~\ref{thm:portrait}. We have a
natural forgetful map
\begin{equation}\label{eq:Z to Z_D}
  Z(\wtB) \to Z(\wtB|D)
\end{equation}
which is, in general, neither injective nor surjective. However, we
will show in Corollary~\ref{cor:conj probl reduction} that~\eqref{eq:Z
  to Z_D} is an isomorphism if $\wtB$ is geometric
non-invertible.

Recall from~\eqref{eq:4exact} the short exact sequence
$\pi_1(X,D)\hookrightarrow \Mod(G|D) \twoheadrightarrow \Mod(G)$. We
have the corresponding sequence for centralizers:
\begin{equation}
  \begin{tikzcd}
    1\ar[r] & Z(\wtB|D) \cap \pi_1(X,D) \ar[r] & Z(\wtB|D) \ar[r] & Z(B).
  \end{tikzcd}
\end{equation}
If $\wtB$ is geometric non-invertible, then
$Z(\wtB|D) \cap \pi_1(X,D)$ is trivial, so
$Z(\wtB|D)\cong Z(\wtB)$ is naturally a subgroup of
$ Z(B)$, and in Theorem~\ref{thm:conjcentr} we will show that it has
finite index.

\begin{defn}
  The \emph{relative centralizer} $Z_D((B_a)_{a\in \wtA})$ of
  a portrait of bisets $(B_a)_{a\in \wtA}$ is the set of
  $(g_d)\in G^D$ such that
  \[B_d = g_d^{-1} B_d g_{f_*(d)}\text{ for all }d\in D,
  \]
  with the understanding that $g_{f_*(d)}=1$ if $f_*(d)\not\in D$.
\end{defn}
We remark that we could also have defined the ``full'' normalizer,
consisting of all $(g_a)\in G^{\wtA}$ with $G_a^{g_a}=G_a$ and
$B_a=g_a^{-1}B_a g_{f_*(a)}$ for all $a\in\wtA$, and its
subgroup the ``full'' centralizer, in which $g_a$ centralizes $G_a$
and $g_a^{-1}\cdot{-}\cdot g_{f_*(a)}$ is the identity on $B_a$; but
we will make no use of these notions. The ``full'' normalizer is the direct
product of $\prod_{a\in A}G_a$ and the relative centralizer.

We also note that, if $(g_d)_{d\in D}$ belongs to the relative
centralizer of $(B_a)_{a\in\wtA}$ and $f^n(d)\in A$ for some
$n\in\N$, then $g_d=1$.

\noindent Applying Theorem~\ref{thm:portrait} to the dynamical
situation we obtain:
\begin{prop}\pushQED{\qed}
  Let
  $\subscript{\wtG}{\wtB}_{\wtG}\twoheadrightarrow
  \subscript G B_G$ be a forgetful morphism and let
  $(B_a)_{a\in\wtA}$ be the induced portrait of bisets in
  $B$. Then any choice of isomorphisms $\pi_1(X,d)\cong G$ gives an isomorphism
  \[Z(\wtB|D) \cap \pi_1(X,D)\stackrel\cong\longrightarrow Z_D((B_a)_{a\in \wtA}).\qedhere\]
\end{prop}

\section{$G$-spaces}
We start by general considerations.  Let $Y$ be a right $H$-space, and
let $X$ be a right $G$-space. For every map $f\colon Y\to X$ there
exists a natural $H$-$G$-biset $M(f)$, defined by
\begin{equation}\label{eq:B(ftilde)}
  M(f)\coloneqq\{h f g\coloneqq f({-}\cdot h)\cdot g\mid h\in H,g\in G\},
\end{equation}
namely the set of maps $Y\to X$ obtained by precomposing with the
$H$-action and post-composing with the $G$-action. Note that $M(f)$ is
right-free if the action of $G$ is free on $X$. We have a natural
$G$-equivariant map $Y\otimes_H M(f)\to X$ given by evaluation:
$y\otimes b\mapsto b(y)$.  Define
\[H_f\coloneqq \{h\in H\mid \exists g\in G\text{ with $h f =f g$ in }M_f\}
\]
the stabilizer in $H$ of $f G\in M(f)/G$. Then $f$ descends to a
continuous map $\overline f\colon Y/H_f\to X/G$.
\begin{lem}\label{lem:MCB vs B}
  Suppose that $X$ and $Y$ are simply connected and that $G,H$ act
  freely with discrete orbits. Then $M(f)$ is isomorphic to the biset
  of the correspondence
  $Y/H \leftarrow Y/H_f\overset{\overline f}\to X/G$ as defined
  in~\cite{bartholdi-dudko:bc1}.
\end{lem}
\begin{proof}
  Let us define the following subbiset of $M(f)$:
  \begin{equation}\label{eq:M'}
    M'(f)\coloneqq\{f({-}\cdot h)\cdot g\mid h\in H_f,g\in G\}.
  \end{equation}
  Since $Y,X$ are simply connected, the biset of
  $\overline f\colon Y/H_f\to X/G$ is isomorphic to $M'(f)$. The
  isomorphism is explicit: choose basepoints $\dagger\in Y$ and
  $*\in X$ so that $\pi_1(Y/H,\dagger H)\cong H$ and
  $\pi_1(X/G,*G)\cong G$. Given $b\in B(\overline f)$, represent it as
  a path $\overline\beta\colon[0,1]\to X/G$ with
  $\overline\beta(0)=\overline f(\dagger H_f)$ and
  $\overline\beta(1)=*G$, and lift it to a path
  $\beta\colon[0,1]\to X$. We have $\beta(0)=f(\dagger)h$ and
  $\beta(1)=*g$ for some $h\in H_f,g\in G$, and we map
  $b\in B(\overline f)$ to $h^{-1}\cdot f\cdot g\in M'(f)$. This map
  is a bijection because both $B(\overline f)$ and $M'(f)$ are
  right-free. We finally have
  \[M(f)\cong H\otimes _{H_f} M'(f)\cong B(Y/H \leftarrow
    Y/H_f\rightarrow X/G).\qedhere\]
\end{proof}
In case the actions of $G,H$ are discrete but not free, there still is
a surjective morphism
$B(X/G\leftrightarrow Y/H)\twoheadrightarrow M(f)$, when $X/G$ and
$Y/H$ are treated as orbispaces.

\begin{exple}[Modular correspondence]
  The mapping class biset $M(f)$ from~\S\ref{ss:MapClassBis} is
  isomorphic to the biset of the associated correspondence on Moduli
  space,
  see~\cite{bartholdi-dudko:bc2}*{Proposition~\ref{bc2:prop:modular
      correspondence}}. Indeed, $M(f)$ is identified with
  \begin{equation}\label{eq:exm:mod corr}
    \{\sigma_n \circ \sigma_f\circ \sigma_m\mid n\in\Mod(S^2,C),m\in\Mod(S^2,A) \},
  \end{equation}
  where $\sigma_m\colon \mathscr{T}_A \selfmap$,
  $\sigma_n\colon \mathscr{T}_C \selfmap$, and
  $\sigma_f\colon \mathscr{T}_A \to \mathscr{T}_C$ are the pulled-back
  actions between Teichm\"uller spaces of $m,n,f$ respectively. By
  Lemma~\ref{lem:MCB vs B}, the biset~\eqref{eq:exm:mod corr} is
  isomorphic to the biset of the modular correspondence
  \[\mathscr M_C
    \overset{\overline{\sigma_f}}\longleftarrow
    \mathscr W_f \overset i\longrightarrow  \mathscr M_A.\]        
\end{exple}

\subsection{Universal covers}
Let us now generalize Lemma~\ref{lem:MCB vs B} by dropping the
requirement that $X,Y$ be simply connected. Choose basepoints
$\dagger,*$, write $Q=\pi_1(Y,\dagger)$ and $P=\pi_1(X,*)$ and
$\wtH=\pi_1(Y/H,\dagger H)$ and $\wtG=\pi_1(X/G,*G)$;
so we have exact sequences
\begin{equation}\label{eq:gp-ses}
  1\to Q\to\wtH\overset\pi\to H\to 1,\qquad 1\to P\to\wtG\overset\pi\to G\to 1.
\end{equation}

The \emph{universal cover} of $X$ may be defined as
\[\widetilde X\coloneqq\{\beta\colon[0,1]\to X\mid\beta(1)=*\}/{\approx};
\]
it has a natural basepoint $\widetilde *$ given by the constant path
at $*$, and admits a right $P$-action by right-concatenation of loops
at $*$. The projection $\widetilde X\to X$ is a covering, and is given
by $\beta\mapsto\beta(0)$. We denote by $\widetilde X^\vee$ the left
$P$-set structure on $\widetilde X$, with action
$g\cdot\beta^\vee\coloneqq(\beta\cdot g^{-1})^\vee$.  We may naturally
identify $\widetilde X^\vee$ with
$\{\beta^{-1}\mid\beta\in\widetilde X\}$ and its natural left
$P$-action.

Let us consider the universal covers $\widetilde X,\widetilde Y$ of
$X,Y$ respectively and a lift $\wtf$ of $f$. We have the
following situation, with the acting groups represented on the right,
and omitted in the left column:
\[\begin{tikzcd}[column sep=large,row sep=tiny]
    \widetilde Y\ar[rr,"\wtf"]\ar[dd] & & \widetilde X\makebox[0mm][l]{${}\looparrowleft\wtG$}\ar[dd]\\
    &\phantom{Z}\\
    \makebox[0mm][r]{$\widetilde Y/Q\cong{}$}Y\ar[rr,"f"]\ar[dd]\ar[dr] & & X\makebox[0mm][l]{${}\looparrowleft G$}\ar[dd]\\
    & Y/H_f\ar[dl]\ar[dr,"\overline f"]\\
    \makebox[0mm][r]{$\widetilde Y/\wtH\cong{}$}Y/H & & X/G
  \end{tikzcd}
\]
Let $\wtH_f$ be the full preimage of $H_f$ in $\wtH$. Note that
$\wtH_f$ is the stabilizer of $\wtf\wtG$ in $M(\wtf)/\wtG$: we have
$\widetilde h\wtf \in \wtf \wtG$ if and only if
$\pi(\widetilde h)f \in f G$ by the unique lifting property. By
Lemma~\ref{lem:MCB vs B}, we have
\[M(\wtf)\cong B(\widetilde Y/\wtH\leftarrow\widetilde Y/\wtH_f\to\widetilde X/\wtG)\cong B(Y/H\leftarrow Y/H_f\to X/G),
\]
and we shall see that $M(\wtf)$ is an extension of bisets.  We
have a natural map $\pi\colon M(\wtf)\to M(f)$, given by
$h\wtf g\mapsto\pi(h)f\pi(g)$ for all
$h\in\wtH,g\in\wtG$.
\begin{lem}
  There is a short exact sequence of bisets
  \begin{equation}\label{eq:biset-ses}
    \begin{tikzcd}\subscript Q{M(\wtf)}_P\ar[r,hookrightarrow] & M(\wtf)\ar[r,twoheadrightarrow,"\pi"] & M(f),\end{tikzcd}
  \end{equation}
  in which every fibre $\pi^{-1}(h f g)$ is isomorphic to a twisted
  form $B(h f g)$ of the biset of $f$.
\end{lem}
\begin{proof}
  This follows immediately from Lemma~\ref{lem:MCB vs B} applied to
  $\widetilde Y,\widetilde X$ with actions of $Q,P$ respectively.
\end{proof}

Let us now assume that the short exact sequences~\eqref{eq:gp-ses} are
split, so $\pi_1(X/G)\cong P\rtimes G$ and
$\pi_1(X/H)\cong Q\rtimes H$. We shall see that the
sequence~\eqref{eq:biset-ses} is split, and that some additional
structure on $M(f)$ and $B(f)$ allow the extension
$M(f)\cong B(Y/H\leftrightarrow X/G)$ to be reconstructed as a crossed
product.

The splitting of the map $\pi\colon\pi_1(X/G,*G)\to G$ means that
there is a family $\{\alpha_g\}_{g\in G}$ of curves in $X$ such that
$\alpha_g$ connects $*g$ to $*$ and
$\alpha_{g_1g_2}\approx(\alpha_{g_1}g_2)\#\alpha_{g_2}$. Similarly
there is a family $\{\beta_h\}_{h\in H}$ of curves in $Y$ such that
$\beta_h$ connects $\dagger h$ to $\dagger$ and
$\beta_{h_1h_2}\approx(\beta_{h_1}h_2)\#\beta_{h_2}$.

For every $h\in H_f$ there is a unique element of $G$, written
$h^f\in G$, such that $h\cdot f= f\cdot h^f$ in $M'(f)$. For every
$h\in H_f$ and every $b\in B(f)$, represented as a path
$b\colon[0,1]\to X$ with $b(0)=f(\dagger)$ and $b(1)=*$, define
\[b^h\coloneqq(f\circ\beta_h^{-1})\#(b\cdot h^f)\#\alpha_{h^f}.
\]
We clearly have $(q\cdot b\cdot p)^h=q^h\cdot b^h\cdot p^{h^f}$, so
$H_f$ acts on $B(f)$ by congruences. We convert that right action to a
left action by ${}^h b\coloneqq b^{h^{-1}}$.

For every $c\in M'(f)$ and every $p\in P$, write $c=f g$ and define
${}^c p\coloneqq {}^g p=g p g^{-1}$. We clearly have
${}^{c g}p={}^c({}^g p)$.
\begin{lem}
\label{lem:cross prod str}
  The biset of $\overline f\colon Y/H_f\to X/G$ is isomorphic to the
  crossed product $B(f)\rtimes M'(f)$, which is $B(f)\times M'(f)$ as
  a set, with actions of $\wtH_f\cong Q\rtimes H_f$ and
  $\wtG\cong P\rtimes G$ given by
  \[(q,h)\cdot(b,c)\cdot(p,g)=(q\cdot{}^h(b\cdot{}^c p),h\cdot c\cdot g).\]
  As a consequence,
  \[B(Y/H \leftarrow Y/H_f\overset{\overline f}\to X/G) \cong \wtH\otimes _{\wtH_f} B(f)\rtimes M'(f).\]
\end{lem}
\begin{proof}
  This is almost a tautology. The short exact
  sequence~\eqref{eq:biset-ses} splits, with section
  $h\cdot f\cdot g\mapsto(1,h)\cdot\wtf\cdot(1,g)$, and the actions of
  $\wtG,\wtH$ on $\widetilde X,\widetilde Y$ can be identified with
  concatenation of lifts of the paths $\alpha_g,\beta_h$.
\end{proof}

\subsection{Self-similarity of \boldmath $G$-spaces}
We return in more detail the situation in which $Y,X$ are universal
covers; we rename them to $\widetilde Y,\widetilde X$ so as to keep
$Y,X$ for the original space.

Consider two pointed spaces $(Y,\dagger)$ and $(X,*)$ with
$H\coloneqq\pi_1(Y,\dagger)$ and $G\coloneqq\pi_1(X,*)$, and let
$f\colon Y\to X$, be a continuous map. Recall that its biset is defined by
\begin{equation}\label{eq:B(f)}
  \subscript H{B(f)}_G\coloneqq\{\beta\colon[0,1]\to X\mid\beta(0)=f(\dagger),\beta(1)=*\}/{\approx},
\end{equation}
with the natural actions by left- and right-concatenation. We thus
naturally have $B(f)\subseteq\widetilde X$, with corresponding right
actions, and left action given by composing loops via $f$.

The map $f\colon Y\to X$ naturally lifts to a map
$\widetilde Y\to Y\to X$, and every choice of $\beta\in B(f)$ defines
uniquely, by the lifting property of coverings, a lift
$\wtf_\beta\colon\widetilde Y\to\widetilde X$ with the
property that $\widetilde\dagger\mapsto\beta$. Furthermore, we have
the natural identities
$\wtf_\beta({-}\cdot h)\cdot g=\wtf_{h\beta g}$, so that
the biset $B(f)$ as defined in~\eqref{eq:B(f)} is canonically
isomorphic to every biset $M(\wtf_\beta)$ as defined
in~\eqref{eq:B(ftilde)}, when an arbitrary $\beta\in B(f)$ is chosen,
and to $B(\wtf)$, when an arbitrary lift
$\wtf\colon\widetilde Y\to\widetilde X$ of $f$ is chosen.

If $f\colon Y\to X$ is a covering, then we may assume
$\widetilde Y=\widetilde X$; choosing $\wtf=\one$ gives the
simple description $B(\wtf)=\subscript H G_G$. Recall that the
biset of a correspondence $(f,i)\colon Y\leftarrow Z\to X$ is defined
as $B(f,i)=B(i)^\vee\otimes B(f)$. In the case of a covering
correspondence, in which $f$ is a covering, we therefore arrive at
$B(f,i)=B(i)^\vee\otimes_{\pi_1(Z)}G$. We shall give now a more
explicit description of this biset using covering spaces, just as we
had in~\cite{bartholdi-dudko:bc1}*{\eqref{bc1:eq:Bfi:lmm:FibrCorr}}
\[B(f,i)=\{(\delta\colon[0,1]\to Y,p\in Z)\mid\delta(0)=\dagger,\delta(1)=i(p),f(p)=*\}/{\approx}\]
and the special case, when $i\colon Z\to Y$ is injective, of
\[B(f,i)=\{\delta\colon[0,1]\to Y\mid\delta(0)=\dagger,f(i^{-1}(\delta(1)))=*\}/{\approx}.\]
Still assuming that $f\colon Z\to X$ is a covering map, define
\begin{equation}\label{eq:ZH}
  \widetilde Z_H\coloneqq\{(\delta,p)\in\widetilde Y\times Z\mid i(p)=\delta(0)\},
\end{equation}
the fibred product of $\widetilde Y$ with $Z$ above $Y$, see
Diagram~\eqref{prop:Gspaces:diag} left. (Note that $\widetilde Z_H$ is
\emph{not} the universal cover of $Z$.) In case $i\colon Z\to Y$ is
injective, we may write
\[\widetilde Z_H=\{\delta\in\widetilde Y\mid\delta(0)\in i(Z)\},
\]
so $\widetilde Z_H$ is the full preimage of $i(Z)$ under the
covering map $\widetilde Y\to Y$.
\begin{prop}\label{prop:Gspaces}
  If $(f,i)$ is a covering correspondence, the following map defines
  an isomorphism of $H$-spaces:
  \[\Phi\colon\subscript H{B(f,i)}_G \otimes \widetilde X^\vee\to(\widetilde Z_H)^\vee\text{ given by }\Phi((\delta,p)\otimes\alpha^\vee)=((\delta\#(i\circ\alpha^{-1}\lift f p))^{-1},p).\]
  For every $b=(\delta,p)\in B(f,i)$ the map
  \[\wtf^{-1}_b\colon \widetilde X^\vee\to(\widetilde Z_H)^\vee\text{ given by }\alpha^\vee\mapsto\Phi(b\otimes\alpha^\vee)
  \]
  is the unique lift of the inverse of the correspondence
  $Y\leftarrow Z\to X$ mapping $\widetilde *$ to $b$; we have the
  equivariance properties
  \begin{equation}\label{eq:propGspaces}
    \wtf^{-1}_{h b g}=h\cdot\wtf^{-1}_b(g\cdot{-}).
  \end{equation}
\end{prop}
The inverses and contragredients may seem unnatural in the statements
above; but are necessary for the actions to be on the right sides, and
are justified by the fact that we construct a lift of the
\emph{inverse} of the correspondence, rather than the correspondence
itself:
\begin{equation}\label{prop:Gspaces:diag}
  \begin{tikzcd}[column sep=huge,row sep=tiny]
    & (\widetilde Z_H)^\vee\arrow[dl,"\widetilde i"{above}]\arrow[dd]\\
    \widetilde Y^\vee\arrow[dd] & & \widetilde X^\vee\arrow[dd]\arrow[ul,"\wtf_b^{-1}"{above}]\\[5ex]
    & Z\arrow[dl,"i"{above}]\arrow[dr,"f"]\\
    Y & & X
  \end{tikzcd}
\end{equation}

\begin{proof}
  It is obvious that $\Phi$ is $H$-equivariant, and it is surjective:
  given $(\delta,p)\in\widetilde Z_H$, choose a path
  $\alpha\in\widetilde X$ with $f(p)=\alpha(0)$, and write
  $\delta=(\alpha^{-1}\lift f p)^{-1}\#\alpha^{-1}\lift f p\#\delta$,
  expressing in this manner
  $(\delta,p)^\vee=\Phi(((\alpha^{-1}\lift f
  p\#\delta)^{-1},\alpha^{-1}\lift f p(1))\otimes\alpha^\vee)$.

  If
  $\Phi((\delta,p)\otimes\alpha^\vee)=\Phi((\delta',p')\otimes(\alpha')^\vee)$,
  then $\alpha$ and $\alpha'$ start at the same point, so
  $\alpha=\alpha' g$ for some $g\in G$, and we have then
  $(\delta,p)=(\delta',p')g^{-1}$ so
  $(\delta,p)\otimes\alpha^\vee=(\delta',p')\otimes(\alpha')^\vee$
  and $\Phi$ is injective.

  It is easy to see that $\wtf^{-1}_b$ is a lift of
  $f^{-1}$. Conversely, every lift $\wtf^{-1}$ of $f^{-1}$
  maps $\widetilde *$ to an element $b\in B(f,i)$ because
  $\wtf^{-1}(\widetilde *^\vee)$ ends at an $f$-preimage of
  $*$ by construction; and then
  $\wtf^{-1}=\wtf^{-1}_b$ by unicity of
  lifts. Finally, equivariance follows from
  $\wtf^{-1}_b(g \alpha^\vee)=\Phi(b\otimes
  g\alpha^\vee)=\Phi(b g\otimes\alpha^\vee)=\wtf^{-1}_{b
    g}(\alpha^\vee)$.
\end{proof}

\subsection{Planar discontinuous groups}\label{ss:PlanDiscGroups}
A \emph{planar discontinuous group} $\widetilde X\looparrowleft G$ is
a group acting properly discontinuously on a plane, which will be
denoted by $\widetilde X$: for every bounded set $V$ the set
$\{g\in G\mid V g\cap V\neq\emptyset\}$ is finite.

Let $X\coloneqq(S^2,A,\ord)$ be an orbifold with non-negative Euler
characteristic, consider $*\in S^2\setminus A$ a base-point, and write
$G\coloneqq\pi_1(X,*)$. Then the universal cover $\widetilde X_G$ of
$X$ is a plane endowed with a properly discontinuous action of $G$. We
denote by $\pi\colon\widetilde X\to X$ the covering map.

By the classification of surfaces, there are only two cases to
consider: $X=\C$ with $G$ a lattice in the affine group
$\{z\mapsto a z+b\mid a,b\in\C,|a|=1\}$, and $X=\mathbb H$ the upper
half plane, with $G$ a lattice in $\operatorname{PSL}_2(\R)$.

Consider another orbisphere $Y\coloneqq(S^2,C,\ord)$ and a branched
covering $f\colon Y\to X$, and fix basepoints $\dagger\in Y$ and
$*\in X$ with corresponding fundamental groups $H=\pi_1(Y,\dagger)$
and $G=\pi_1(X,*)$. Let $\subscript H{B(f)}_G$ be the biset of $f$. As
usual, we view $f$ as a correspondence $Y\leftarrow Z\to X$ with
$Z=(S^2,f^{-1}(A),\ord)$ and $i$ a homeomorphism $S^2\to S^2$ mapping
a subset of $f^{-1}(A)$ to $C$. The fibred product $\widetilde Z_H$
constructed in~\eqref{eq:ZH} is naturally a subset of the plane
$\widetilde Y$, with orbispace points and punctures at all $H$-orbits
of $f^{-1}(A)\setminus C$.  We need the following classical result.
\begin{thm}[Baer~\cite{zieschang-vogt-coldewey:spdg}*{Theorem~5.14.1}]\label{thm:Baer}
  Every orientation-preserving homeomorphism of a plane commuting with
  a properly discontinuous group action is isotopic to the identity
  along an isotopy commuting with the action.\qed
\end{thm}

\noindent Let us
reprove~\cite{bartholdi-dudko:bc2}*{Theorem~\ref{bc2:thm:dehn-nielsen-baer++}}
using our language of group actions:
\begin{cor}\label{cor:OrbiIsComplInvar}
  Let two orbisphere maps $f,g\colon (S^2,C,\ord)\to (S^2,A,\ord)$
  have isomorphic orbisphere bisets. Then $f\approx_C g$.

  In other words, the orbisphere biset of
  $f\colon(S^2,C,\ord)\to(S^2,A,\ord)$ is a complete invariant of $f$
  up to isotopy rel $C$.
\end{cor}
\begin{proof}
  Let us write $X=(S^2,A,\ord)$ and $Y=(S^2,C,\ord)$ and
  $G=\pi_1(X,*)$ and $Y=\pi_1(Y,\dagger)$. We may represent $f,g$
  respectively by covering pairs $(f,i)$ and $(g,i)$, with coverings
  $f,g\colon(S^2,P,\ord)\to(S^2,A,\ord)$ and
  $i\colon(S^2,P,\ord)\to(S^2,C,\ord)$. Let us furthermore write
  $Z=(S^2,P,\ord)$ and $\widetilde Z_H$ its fibred product with
  $\widetilde Y$. Identifying $B(f,i)$ and $B(g,i)$, choose
  $b\in B(f,i)=B(g,i)$, and let
  \[\wtf^{-1}_b,\wtg^{-1}_b\colon \widetilde X^\vee\to(\widetilde Z_H)^\vee
  \]
  be the corresponding lifts as in Proposition~\ref{prop:Gspaces}.

  Since $i$ is injective, the map $(\widetilde Z_H)^\vee\to Z$ is a
  covering, so $\wtf^{-1}_b$ and $\wtg^{-1}_b$ are
  coverings. We may therefore consider their quotient
  $\wtf_b\circ \wtg^{-1}_b$, which is a well-defined
  map $(\widetilde X)^\vee\selfmap$, and is normalized to preserve the
  base point $\widetilde *$. By~\eqref{eq:propGspaces} it is a
  homeomorphism commuting with the action of $G$.

  By Theorem~\ref{thm:Baer} there is an isotopy
  $(\widetilde h_{b,t}^{-1})_{t\in[0,1]}$ of maps
  satisfying~\eqref{eq:propGspaces} between $\one$ and
  $\wtf_b\circ \wtg^{-1}_b$.

  Since the set of fixed points of $G$ is discrete,
  $\widetilde h_{b,t}^{-1}$ preserves all fixed points of $G$ and
  therefore projects to an isotopy $(h_t)_{t\in[0,1]}$ in $X$. We have
  $h_0=\one$ and $h_1\circ g=f$, so the maps $f$ and $g$ are isotopic
  rel $B$.
\end{proof}

\section{Geometric maps}
Let $M$ be a matrix with integer entries and with $\det(M)>1$. We call
$M$ \emph{exceptional} if one of the eigenvalues of $M$ lies in
$(-1,1)$; so the dynamical system $M\colon\R^2\selfmap$ has one
attracting and one repelling direction.

For $r\in\R^2$, denote by $\langle \Z^2,-z+r\rangle$ the group of
affine transformations of $\R^2$ generated by translations
$z\mapsto z+v$ with $v\in \Z^2$ and the involution $z\mapsto
-z+r$. The quotient $\R^2/\langle\Z^2,-z+r\rangle$ is a topological
sphere, with cone singularities of angle $\pi$ at the four images of
$\tfrac12(r+\Z^2)$.

We call a branched covering $f\colon S^2\selfmap$ a \emph{geometric
  exceptional} map if $f\colon S^2\selfmap$ is conjugate to a quotient
of an exceptional linear map $M\colon\R^2\selfmap$ under the action of
$\langle \Z^2,-z+r\rangle$ for some $r\in\R^2$; in particular,
$(\one-M)r\in\Z^2$.  A Thurston map $f\colon(S^2,A,\ord)\selfmap$ is
called \emph{geometric} if the underlying branched covering
$f\colon S^2\selfmap$ is either B\"{o}ttcher expanding
(see~\cite{bartholdi-dudko:bc4}*{Definition~\ref{bc4:defn:metricallyexpanding}};
there exists a metric on $S^2$ that is expanded everywhere except at
critical periodic cycles), a geometric exceptional map, or a
pseudo-Anosov homeomorphism.  We refer to the first two types as
\emph{non-invertible geometric maps}.

We may consider more generally affine maps $z\mapsto M z+q$ on $\R^2$,
and then their quotient by the group $\langle \Z^2,-z\rangle$; the map
$z\mapsto M z$ on $\R^2/\langle\Z^2,-z+r\rangle$ is converted to that
form by setting $q=(M-\one)r$. Conversely, if $1$ is not an eigenvalue
of $M$, then we can always convert an affine map into a linear one, at
the cost of replacing the reflection $-z$ by $-z+r$ in the acting
group.

\begin{lem}\label{lem:IrratOfDynam}
  Let $M\colon\R^2\selfmap$ be exceptional. Then for every bounded set
  $D\subset\R^2$ containing $(0,0)$ there is an $n>0$ such that for
  every $m\ge n$ we have $M^{-m}(D)\cap\Z^2=\{(0,0)\}$. Moreover,
  $n=n(D)$ can be taken with
  $n(D)\preccurlyeq \log \operatorname{diam}(D)$.
\end{lem}
\begin{proof}
  Let $\lambda_1, \lambda_2$ be the eigenvalues of $M$, and let $e_1$
  and $e_2$ be unit-normed eigenvectors associated with $\lambda_1$
  and $\lambda_2$.  It is sufficient to assume that $D$ is a
  parallelogram centered at $(0,0)$ with sides parallel to $e_1$ and
  $e_2$:
  \begin{equation}\label{eq:prf:lem:IrratOfDynam}
    D =\{v\in \R^2\mid v= t_1e_1+t_2v_2\text{ with }|t_1|\le x\text{ and }|t_2|\le y\}.
  \end{equation}
  In particular, the area of $D$ is comparable to $x y$. Then
  $M^{-n}(D)$ is again a parallelogram centered at $(0,0)$ with sides
  parallel to $e_1$ and $e_2$.

  We claim that there is $\delta>0$ such that if $D$ is a
  parallelogram of the form~\eqref{eq:prf:lem:IrratOfDynam} with
  $\operatorname{area}(D)<\delta$, then $D\cap \Z^2=\{(0,0)\}$. This
  will prove the lemma because
  $\operatorname{area}(M^{-n}(D))=\operatorname{area}(D)/(\det M)^n$
  and $\log\operatorname{diam}(D)\asymp\log\operatorname{area}(D)$.

  Without loss of generality we assume $x>y$, so $D$ is close to
  $\R e_1$. Let $\mu_1$ be the slope of $\R e_1 $. Since $M$ is
  exceptional, the numbers $\lambda_1,\lambda_2,\mu_1$ are quadratic
  irrational, so they are not well approximated by rational numbers:
  there is a positive constant $C$ such that
  $|\mu_1-\frac p q |> \frac{C}{q^2}$ for all $\frac p q\in \Q$.

  Suppose that $D$ contains a non-zero integer point $w=(p,q)$; so
  $x\ge|q|$. Also, $w$ is close to $\R e_1$, and in particular
  $q\neq 0$ if $\delta$ is sufficiently small; we may suppose
  $q>0$. It also follows that $\mu_1$ is close to $\frac p q$.  The
  distance from $w$ to $\R e_1$ is
  \[d(w,\R e_1)\le y\asymp \frac{\operatorname{area}(D)}{x}.
  \]
  On the other hand,
  \[d(w,\R e_1)=|w| \sin \sphericalangle(w,e_1) \succcurlyeq q\;\Big| \mu_1- \frac p q\Big|.\]
  Combining, we get
  \[\Big|\mu_1 - \frac p q\Big|\preccurlyeq\frac{d(w,\R e_1)}{q}\preccurlyeq\frac{\operatorname{area}(D)}{q\cdot x}\preccurlyeq\frac{\delta}{q^2},
  \]
  and the claim follows for $\delta\prec\!\!\prec C$.
\end{proof}

\begin{cor}\label{lem:IrratOfDynam2}
  Let $f\colon(S^2,A)\selfmap$ be a geometric exceptional map, let
  $p\in S^2\setminus A$ be a periodic point with period $n$, and let
  $\gamma\in\pi_1(S^2\setminus A,p)$ be a loop starting and ending at
  $p$. Let $|\gamma|$ be the length of $\gamma$ with respect to the
  Euclidean metric of the minimal orbifold structure
  $(P_f,\ord_f)$. If $\gamma$ is trivial rel $(P_f,\ord_f)$ and
  $m\succ \log |\gamma|$, then the lift $\gamma \lift{f^{n m}}{p}$ is
  a trivial loop rel $A$.
\end{cor}
\begin{proof}
  By passing to an iterate, we may assume that $p$ is a fixed point of
  $f$.  Since $f$ is geometric exceptional, we have a branched
  covering map $\pi\colon\R^2\to S^2$ under which $f$ lifts to an
  exceptional linear map $M$, and we may assume $\pi(0,0)=p$. By
  assumption, $\gamma\lift {\pi}{(0,0)}$ is a trivial loop in
  $\R^2$. By Lemma~\ref{lem:IrratOfDynam}, the sets
  $M^{-m}(\gamma\lift {\pi}{(0,0)})$ and $\pi^{-1}(A)$ are disjoint
  for $m\succ \log |\gamma|$; hence $\gamma \lift{f^{n m}}{p}$ is
  a trivial loop rel $A$.
\end{proof}

Geometric exceptional maps all admit a minimal orbifold modeled on
$\R^2/\langle \Z^2,-z\rangle$, which has cone singularities of angle
$\pi$ at the images of $\{0,\frac12\}\times\{0,\frac12\}$. We consider
this class in a little more detail:
\subsection{\boldmath $(2,2,2,2)$-maps}
A branched covering $f\colon(S^2,P_f,\ord_f)\selfmap$ is \emph{of type
  $(2,2,2,2)$} if $|P_f|=4$ and $\ord_f(x)=2$ for every $x\in P_f$. In
this case, $f$ is isotopic to a quotient of an affine map
$z\mapsto M z+q$ under the action of $\langle\Z^2,-z\rangle$, see
Proposition~\ref{prop:EndOfG2}(B).

\begin{lem}\label{lem:CharactTorusMaps}
  Let $M$ be a matrix with integer entries and $\det(M)>1$. Denote by
  $\lambda_1$ and $\lambda_2$ the eigenvalues of $M$, ordered as
  $|\lambda_2|\ge |\lambda_1|>0$. Then the following are all mutually
  exclusive possibilities:
  \begin{itemize}
  \item if $M$ is exceptional, that is $0<|\lambda_1|<1<|\lambda_2|$
    and $\lambda_1,\lambda_2\in \R$, then $M\colon\R^2\selfmap$ is
    expanding in the direction of the eigenvector corresponding to
    $\lambda_2$ and $M\colon\R^2\selfmap$ is contracting in the
    direction of the eigenvector corresponding to $\lambda_1$;
  \item if $\lambda_1\in \{\pm 1\}$, then $M^2\colon\R^2\selfmap$
    preserves the rational line $\{z\in \R^2\mid M z=z\}$;
  \item the map $M\colon\R^2\selfmap$ is expanding in all other cases;
    that is $\lambda_1=\overline {\lambda_2}\notin\R$, or
    $\lambda_1,\lambda_2\in \R$ but $|\lambda_1|,|\lambda_2| > 1$.
  \end{itemize}
\end{lem}
\begin{proof}
  If $M$'s eigenvalues are non-real, then
  $\lambda_1=\overline{\lambda_2}$ and $|\lambda_1|=|\lambda_2|$, so
  $M$ is expanding. If $\lambda_1$ and $\lambda_2$ are real, then they
  are of the same sign and their product is greater than $1$. The
  lemma follows.
\end{proof}

\noindent The following lemma follows from~\cite{selinger-yampolsky:geometrization}*{Main Theorem II}:
\begin{lem}\label{lem:Tor:geom=lf}
  If $f\colon(S^2,P_f,\ord_f)\selfmap$ is doubly covered by a torus
  endomorphism $z\mapsto M z+q$, then $f$ is geometric if and only if
  $f$ is Levy-free.
\end{lem}
\begin{proof}
  We consider the exclusive cases of
  Lemma~\ref{lem:CharactTorusMaps}. In the first case, $f$ is
  geometric and $z\mapsto M z+q$ preserves transverse irrational
  laminations (given by the eigenvectors of $M$), so $z\mapsto M z+q$
  admits no Levy cycle and \emph{a fortiori} neither does $f$.

  In the second case, $f$ is not geometric and admits as Levy cycle
  the projection of the $1$-eigenspace of $M$.

  In the third case, $f$ is expanding, and admits no Levy cycle
  by~\cite{bartholdi-dudko:bc4}*{Theorem~\ref{bc4:thm:main}}.
\end{proof}

We shall mainly study $(2,2,2,2)$-maps algebraically: we write the
torus as $\R^2/\Z^2$ and the $(2,2,2,2)$-orbisphere as
$\R^2/\langle\Z^2,-z\rangle$. Its fundamental group $G$ is identified
with $\langle\Z^2,-z\rangle\cong\Z^2\rtimes\{\pm1\}$. The orbifold
points are the images on the orbisphere of
$\{0,\tfrac12\}\times\{0,\tfrac12\}$.  We start with some basic
properties of $G$:
\begin{prop}\label{prop:EndOfG}
  \begin{itemize}
  \item[(A)] Every injective endomorphism of $G$ is of the form
    \begin{equation}\label{eq:InjEndOfK}
      M^{v}\colon (n,1)\mapsto(M n,1)\text{ and }(n,-1)\mapsto(M n+v,-1)
    \end{equation}
    for some $v\in\Z^2$ and some non-degenerate matrix $M$ with
    integer entries. We have
    \[ N^w\circ M^v= (N M)^{w+N v}.\]There are precisely $4$ order-$2$ conjugacy
    classes in $G$, each of the form
    \[(a,-1)^G=\{(a+2w,-1)\mid w\in\Z^2\}\text{ for some }a\in\{0,1\}\times\{0,1\}\subset\Z^2.
    \]
    The set of order-$2$ conjugacy classes of $G$ is preserved by $M^{v}$.
  \item[(B)] The automorphism group of $G$ is $\{M^v\mid\det M=\pm1\}$
    and is naturally identified with
    $\Z^2\rtimes\operatorname{GL}_2(\Z)$. The inner automorphisms of
    $G$ are identified with $2\Z^2\rtimes\{\pm\one\}$, and the outer
    automorphism group of $G$ is identified with
    $(\Z/2\Z)^2\rtimes\operatorname{PGL}_2(\Z)$. The index-$2$
    subgroup of positively oriented outer automorphisms is
    $\Out^+(G)=(\Z/2\Z)^2\rtimes\operatorname{PSL}_2(\Z)$.
  \item[(C)] The modular group $\Mod(G)$ is free of rank $2$, and we have
    \begin{align*}
      \Mod(G)&=\{M^v\mid\det(M)=1,M\equiv\one\pmod2,v\in2\Z^2\}/(\pm\one)^{2\Z^2}\\
      &\cong\{M^0\mid M\equiv\one\pmod2\}/\{\pm\one\}.
    \end{align*}
  \item[(D)] Two bisets $B_{M^v}$ and $B_{N^w}$ are isomorphic if and
    only if $M=\pm N$ and $(M^v)_*=(N^w)_*$ as maps on order-$2$
    conjugacy classes, if and only if $M=\pm N$ and $v\equiv w\pmod{2\Z^2}$.
  \end{itemize}
\end{prop}
We remark that~\eqref{eq:InjEndOfK} also follows from
Lemma~\ref{lem:cross prod str}.
\begin{proof}
  (A). It is easy to check that $M^v$ defines an injective
  endomorphism. Conversely, let $M'\colon G\to G$ be an injective
  endomorphism. Then $M'(0,-1)=(v,-1)$ for some $v\in\Z^2$ because all
  $(w,-1)$ have order $2$. On the other hand,
  $M'\restrict{\Z^2\times\{1\}}= M\restrict{\Z^2}$ for a
  non-degenerate matrix $M$ with integer entries because $M'$ is
  injective. It easily follows that $M'=M^v$ as given
  by~\eqref{eq:InjEndOfK}.

  The claims on composition and order-$2$ conjugacy classes of $G$
  follow from direct computation.

  (B). Follows directly from~(A) and the identification of $G$ with
  $\{\pm\one\}^{\Z^2}$.

  (C). We use~(B); the modular group of $G$ is the subgroup of
  $\Out^+(G)=(\Z/2)^2\rtimes\operatorname{PSL}_2(\Z)$ that fixes the
  order-$2$ conjugacy classes. The action of $(\Z/2\Z)^2$ on these
  classes is simply transitive, so the set of order-$2$ classes may be
  put in bijection with $(\Z/2)^2$ under the correspondence
  $(a+2\Z^2,-1)\leftrightarrow a$; then the action of
  $\operatorname{PSL}_2(\Z)$ on order-$2$ conjugacy classes is
  identified with the natural linear action (noting that $-\one$ acts
  trivially mod $2$). It follows that $\Mod(G)$ is the congruence
  subgroup
  $\{M\in\operatorname{PSL}_2(\Z)\mid M\equiv\one\pmod{\pm2}\}$, and
  it is classical that it is a free group of rank $2$.

  (D). The bisets $B_{M^v}$ and $B_{N^w}$ are isomorphic if and only
  the maps $M^v,N^w$ are conjugate by an inner automorphism; so the
  claimed description follows from~(B).
\end{proof}

\noindent We turn to $(2,2,2,2)$-maps, and their description in terms
of the above; we use throughout $G=\Z^2\rtimes\{\pm1\}$:
\begin{prop}\label{prop:EndOfG2}
  Let $f$ be a $(2,2,2,2)$-map with biset
  $\subscript G B_G=B(f,P_f,\ord_f)$. Then
  \begin{itemize}
  \item[(A)] The biset $B$ is right principal, and for any choice of
    $b_0\in B$ there exists an injective endomorphism $M^v$ of
    $G$ satisfying $g b_0= b_0 M^v(g)$ for all $g\in G$.
  \item[(B)] The map $f$ is Thurston equivalent to a quotient
    of $z\mapsto M z+\tfrac12v\colon\R^2\selfmap$ under the action of
    $G$.
  \item[(C)] The map $f$ is Levy obstructed if and only if $M$ has an
    eigenvalue in $\{\pm 1\}$.
  \item[(D)] If $f$ is not Levy obstructed then for every $b_0\in B$,
    writing $M^v$ as in~(A) and~\eqref{eq:InjEndOfK}, the map $f$ is
    Thurston equivalent to the quotient of
    $z\mapsto M z\colon\Z^2\selfmap$ by the action of
    $\langle \Z^2, z\mapsto -z+r \rangle \looparrowright \R^2$ for a
    vector $r\in\R^2$ satisfying $M r=r+v$.  The $G$-equivariant map
    of Proposition~\ref{prop:ShadViaUnCover} takes the form
    \begin{equation}\label{eq:GEquivMap}
      \Phi\colon\subscript G B_G\otimes  \subscript{\langle \Z^2,-z+r \rangle}{\R^2}\to\subscript{\langle \Z^2,-z+r \rangle}{\R^2},\qquad b_0\otimes z\mapsto M^{-1}z.
    \end{equation}
  \end{itemize}
\end{prop}
\begin{proof}
  (A). Since $f$ is a self-covering of orbifolds, $\subscript G B_G$
  is right principal. The claim then follows from
  Proposition~\ref{prop:EndOfG}(A).

  (B). We claim that the quotient map, denoted by $\bar f$, has a
  biset isomorphic to $\subscript G B_G$. Indeed, for $g\in G$ it
  suffices to verify that
  \[(M z+\tfrac{1}{2}v) \circ g(z)= M^v(g)\circ  (M z+\tfrac{1}{2}v).
  \]
  If $g=(t,1)$ then both parts are $M z+M t+\frac{1}{2}v$, and if
  $g=(0,-1)$ then both parts are $-(M z+\frac{1}{2}v).$ Therefore, $f$
  and $\bar f$ have isomorphic orbisphere bisets because marked
  conjugacy classes are preserved automatically by Proposition~\ref{prop:EndOfG}~(A).  By
  Corollary~\ref{cor:OrbiIsComplInvar} the maps $f$ and $\bar f$ are
  isotopic.

  (C). Suppose that $M$ has an eigenvalue in $\{1,-1\}$; let $w$ be
  the eigenvector of $M$ associated with this eigenvalue. Consider the
  foliation $F_w$ of $\R^2$ parallel to $w$. Then $F_w$ is invariant
  under $z\mapsto M z+\frac{1}{2}v$ as well as under the action of
  $\langle \Z^2, z\mapsto -z \rangle$. Therefore, the quotient map has
  invariant foliation $\overline F_w= F_w/G$. There are two
  leaves in $\overline F_w$ connecting points in pairs in the
  post-critical set of the quotient map; any other leave is a Levy
  cycle.

  (D). Since $\det(M-\one)\neq0$, there is a unique $r$ such that
  $M r=r+v$.  It is easy to see (as in~(C)) that the quotient map
  in~(E) has a biset isomorphic to $\subscript G B_G$. By
  Corollary~\ref{cor:OrbiIsComplInvar} the quotient map is conjugate
  to $f$, and~\eqref{eq:GEquivMap} is immediate.
\end{proof}

\begin{cor}\label{cor:bis:2222}
  Let $f$ be a $(2,2,2,2)$-map. Then its biset
  $B(f\colon(S^2,P_f,\ord_f)\selfmap)$ is of the form $B_{M^v}$ for an
  endomorphism $M^v$ of $G\coloneqq\Z^2\rtimes\{\pm1\}$, namely it is
  $G$ qua right $G$-set, with left action given by $g\cdot h=M^v(g)h$.\qed
\end{cor}

Let us also recall how the biset of a $(2,2,2,2)$-map is converted
to the form $B_{M^v}$. First, the fundamental group $G$ is identified
with $\Z^2\rtimes\{\pm1\}$. The group $G$ has a unique subgroup $H$ of
index $2$ that is isomorphic to $\Z^2$, so $H$ is easy to find. The
restriction of the biset to $H$ yields a $2\times2$-integer matrix
$M$; and the translation part $v$ is found by tracking the peripheral
conjugacy classes. Note that this procedure applies as well to
non-invertible maps as to homeomorphisms; and that the map is
orientation-preserving precisely if $\det(M)>0$.

\subsection{Homotopy pseudo-orbits and shadowing}\label{ss:shadowing}
Let $f\colon S^2\selfmap$ be a self-map, and let $I$ be an index set
together with a index map also written $f\colon I\selfmap$.  An
\emph{$I$-poly-orbit} $(x_i)_{i\in I}$ is a collection of points in
$S^2$ such that $f(x_i)=x_{f(i)}$. If all points $x_i$ are distinct,
then $(x_i)_{i\in I}$ is an \emph{$I$-orbit}.

Thus a poly-orbit differs from an orbit only in that it allows
repetitions. Every poly-orbit has a unique associated orbit, whose
index set is obtained from $I$ by identifying $i$ and $j$ whenever
$x_i=x_j$. Note that we allow $I=\N$, $I=\Z$ and $I=\Z/n\Z$ as index
sets, with $f(i)=i+1$, as well as $I=\{0,\dots,m,\dots,m+n\}$ with
$f(i)=i+1$ except $f(m+n)=m$, an orbit with period $n$ and preperiod
$m$.

We shall consider a homotopical weakening of the notion of orbits. Our
treatment differs from~\cite{ishii-smillie:shadowing} in a subtle manner
(see below); recall that $\beta\approx_{A,\ord}\gamma$ means that the
curves $\beta,\gamma$ are homotopic in the orbispace $(S^2,A,\ord)$:
\begin{defn}[Homotopy pseudo-orbits]
\label{defn:HomPsOrbit}
  Let $f\colon (S^2,A,\ord)\selfmap$ be an orbisphere self-map, and
  let $I$ be a finite index set together with an index map also written
  $f\colon I\selfmap$.

  An \emph{$I$-homotopy pseudo-orbit} is a collection of paths
  \[(\beta_i)_{i\in I} \text{ with } \beta_i\colon[0,1]\to S^2\setminus A \text{ satisfying }\beta_{f(i)}(0)=f(\beta_i(1)).
  \]
  Two homotopy pseudo-orbits $(\beta_i)_{i\in I}$ and
  $(\beta'_i)_{i\in I}$ are \emph{homotopic}, written
  $(\beta_i)_{i\in I}\approx_{A,\ord}(\beta'_i)_{i\in I}$, if
  $\beta_i\approx_{A,\ord}\beta'_i$ for all $i\in I$. In particular,
  $\beta_i(0)=\beta'_i(0)$ and $\beta_i(1)=\beta'_i(1)$.

  Two homotopy pseudo-orbits $(\beta_i)_{i\in I}$ and
  $(\gamma_i)_{i\in I}$ are \emph{conjugate}, written
  $(\beta_i)_{i\in I}\sim(\gamma_i)_{i\in I}$, if there is a
  collection of paths $(\ell_i)_{i\in I}$ with
  \[\ell_i(0)=\beta_i(0),\quad \ell_i(1)=\gamma_i(0)\text{ and
    }\beta_i\#\ell_{f(i)}\lift{f}{\beta_i(1)}\approx_{A,\ord}\ell_i\#\gamma_i.
  \]
  Poly-orbits are special cases of homotopy pseudo-orbits, in which
  the paths $\beta_i$ are all constant.
\end{defn}

\begin{rem}
  A homotopy pseudo orbit can also be defined for an infinite $I$ with
  the assumption that $(\ell_i)_{i\in I}$ and $(\beta_i)_{i\in I}$ in
  Definition~\ref{defn:HomPsOrbit} have uniformly bounded length. Then
  Theorem~\ref{thm:ShadowEquivalence} still holds.
\end{rem}

In a homotopy pseudo-orbit, the curves $(\beta_i)_{i\in I}$ encode
homotopically the difference between $x_i\coloneqq\beta_i(0)$ and a
choice of preimage of $x_{f(i)}$. Note that Ishii-Smillie define
in~\cite{ishii-smillie:shadowing}*{Definition~6.1} homotopy
pseudo-orbits by paths $\overline{\beta_i}$ connecting $f(x_i)$ to
$x_{f(i)}$; their $\overline{\beta_i}$ may be uniquely lifted to paths
$\beta_i$ from $x_i$ to an $f$-preimage of $x_{f(i)}$ as in
Definition~\ref{defn:HomPsOrbit}, but our definition does not reduce
to theirs, since our $\beta_i$ are defined rel $(A,\ord)$, not rel
$(f^{-1}(A),f^*(\ord))$.

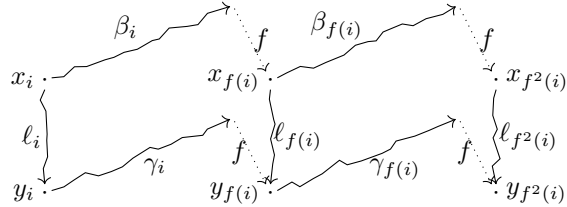
\begin{center}\begin{tikzpicture}[decoration={random steps,segment length=5pt,amplitude=2pt},inner sep=1pt,->]
  \begin{scope}[yshift=1.5cm]
    \node (x_i) at (0,0) [label=left:$x_i$] {$\cdot$};
    \node (x_j') at (2.5,1) {};
    \node (x_j) at (3,0) [label=left:$x_{f(i)}$] {$\cdot$};
    \node (x_k') at (5.5,1) {};
    \node (x_k) at (6,0) [label=right:$x_{f^2(i)}$] {$\cdot$};
    \draw[decorate] (x_i) -- node[above left] {$\beta_i$} (x_j');
    \draw[decorate] (x_j) -- node[above left] {$\beta_{f(i)}$} (x_k');
    \draw[dotted] (x_j') -- node [right] {$f$} (x_j);
    \draw[dotted] (x_k') -- node [right] {$f$} (x_k);
  \end{scope}
  \node (y_i) at (0,0) [label=left:$y_i$] {$\cdot$};
  \node (y_j') at (2.5,1) {};
  \node (y_j) at (3,0) [label=left:$y_{f(i)}$] {$\cdot$};
  \node (y_k') at (5.5,1) {};
  \node (y_k) at (6,0) [label=right:$y_{f^2(i)}$] {$\cdot$};
  \draw[decorate] (y_i) -- node[below right] {$\gamma_i$} (y_j');
  \draw[decorate] (y_j) -- node[below right] {$\gamma_{f(i)}$} (y_k');
  \draw[dotted] (y_j') -- node [left] {$f$} (y_j);
  \draw[dotted] (y_k') -- node [left] {$f$} (y_k);
  \draw[decorate] (x_i) -- node [left] {$\ell_i$} (y_i);
  \draw[decorate] (x_j) -- node [right] {$\ell_{f(i)}$} (y_j);
  \draw[decorate] (x_k) -- node [right] {$\ell_{f^2(i)}$} (y_k);
\end{tikzpicture}\end{center}

Choose a length metric on $S^2$, and define the distance between homotopy
pseudo-orbits by
\begin{equation}
  d\big((\beta_i)_{i\in I},(\gamma_i)_{i\in I}\big)\coloneqq\max_{i\in I,t\in[0,1]} d(\gamma_i(t),\beta_i(t)).
\end{equation}

The following result states that the set of conjugacy classes of
homotopy pseudo-orbits is discrete:
\begin{lem}[Discreteness]\label{lmm:ShadowPsOrb}
  Let $(\beta_i)_{i\in I}$ be a homotopy pseudo-orbit of
  $f\colon(S^2,A,\ord)\selfmap$. Then there is an $\varepsilon> 0$
  such every homotopy pseudo-orbit at distance less than $\varepsilon$
  from $(\beta_i)_{i\in I}$ is conjugate to it.
\end{lem}
\begin{proof}
  Set $\varepsilon=\min_{a\neq b\in A}d(a,b)$. Consider a homotopy
  pseudo-orbit $(\gamma_i)_{i\in I}$ at distance $\delta<\varepsilon$
  from $(\beta_i)_{i\in I}$. Connect $\beta_i(0)$ to $\gamma_i(0)$ by
  a path $\ell_i$ of length at most $\delta$. Since $S\setminus A$ is
  locally contractible space, the curve $\ell_i$ is unique up to
  homotopy and
  $\beta_i\#\ell_{f(i)}\lift{f}{\beta_i(1)}\approx_{A,\ord}\ell_i\#\gamma_i$.
\end{proof}

We recall that $A^{\infty}$ denotes the union of all periodic cycles
containing critical points.
\begin{defn}[Shadowing]
\label{defn:Shadowing}
  A homotopy pseudo-orbit $(\beta_i)_{\in I}$ \emph{shadows rel
    $(A,\ord)$} a poly-orbit $(p_i)_{i\in I}$ in
  $S\setminus A^{\infty}$ if they are conjugate; namely if there are
  curves $\ell_i$ connecting $\beta_i(0)$ to $p_i$ that lie in
  $S^2\setminus A$ except possibly their endpoints and such that for
  every $i\in I$ we have
  \[\ell_i\approx_{(A,\ord)} \beta_i\#\ell_{f(i)}\lift{f}{\beta_i(1)}.\qedhere\]
\end{defn}

\begin{lem}
  The homotopy pseudo-orbit $(\beta_i)_{\in I}$ shadows the poly-orbit
  $(p_i)_{i\in I}$ if and only if every neighbourhood
  $(\mathcal U_i)_{i\in I}$ of $(p_i)_{i\in I}$ contains a homotopy
  pseudo-orbit $(\beta'_i)_{i\in I}$ conjugate to
  $(\beta_i)_{i\in I}$; namely, $\beta'_i\subset\mathcal U_i\ni p_i$
  for all $i\in I$.
\end{lem}
\begin{proof}
  If $(\beta_i)$ can be conjugated into a small enough neighbourhood
  of $(p_i)_{i\in I}$, then it is conjugate to $(p_i)_{i\in I}$ by
  Lemma~\ref{lmm:ShadowPsOrb}. The converse is obvious.
\end{proof}

\begin{prop}\label{prop:ShadAndExtraPts}
  Suppose $f\colon(S^2,A,\ord)\selfmap$ is a geometric non-invertible
  map. Then a periodic pseudo-orbit $(\beta_i)_{i\in I}$ shadows
  $(p_i)_{i\in I}$ rel $(A,\ord)$ if and only if $(\beta_i)_{i\in I}$
  shadows $(p_i)_{i\in I}$ rel $(P_f,\ord_f)$.
\end{prop}
\begin{proof}
  Clearly if $(\beta_i)_{i\in I}$ shadows an orbit $(p_i)_{i\in I}$
  rel $(A,\ord)$ then $(\beta_i)_{i\in I}$ shadows $(p_i)_{i\in I}$
  rel $(P_f,\ord_f)$. Conversely, suppose $(\beta_i)_{i\in I}$ shadows
  an orbit $(p_i)_{i\in I}$ rel $(P_f,\ord_f)$, so there are paths
  $\ell_i$ connecting $\beta_i(0)$ to $p_i$ with
  $\ell_i\approx_{P_f,\ord_f} \beta_i\#
  \ell_{f(i)}\lift{f}{\beta_i(1)}$. Thus
  $\ell_i^{-1}\#\beta_i\#\ell_{f(i)}\lift{f}{\beta_i(1)}$ is a loop
  which is trivial rel $(P_f,\ord_f)$, but may not be trivial rel $A$.

  Consider the following pullback iteration. Set
  $(\beta^0_i)_{i\in I}\coloneqq(\beta_i)_{i\in I}$ and
  $(\ell^{0}_i)_{i\in I}\coloneqq(\ell^{0}_i)_{i\in I}$. Define
  \[(\beta^n_i)_{i\in I}\coloneqq\big(\beta^{n-1}_{f(i)}\lift{f}{\beta^{n-1}_i(1)}\big)_{i\in I}
    \text{ and }(\ell^{n}_i)_{i\in I}\coloneqq\big(\ell^{n-1}_{f(i)}\lift{f}{\beta_i(1)}\big)_{i\in I}.
  \]
  Clearly the $(\beta^n_i)_{i\in I}$ are all conjugate.  Observe that
  $(\ell_i^n)^{-1}\#\beta_i\#\ell^n_{f(i)}\lift{f}{\beta_i(1)}$ is a
  loop passing through $p_i$ and
  $(\ell_i^n)^{-1}\#\beta_i\#\ell^n_{f(i)}\lift{f}{\beta_i(1)}$ is a
  degree-$1$ preimage of
  $(\ell_i^{n-1})^{-1}\#\beta_i\#\ell^{n-1}_{f(i)}\lift{f}{\beta_i(1)}$.

  If $f$ is expanding, then the diameter of
  $(\ell_i^n)^{-1}\#\beta_i\#\ell^n_{f(i)}\lift{f}{\beta_i(1)}$
  tends to $0$ exponentially fast, hence
  $(\ell_i^n)^{-1}\#\beta_i\#\ell^n_{f(i)}\lift{f}{\beta_i(1)}$ is
  trivial rel $(A,\ord)$ for all sufficiently large $n$, and the claim
  follows.

  If $f$ is exceptional, then
  $(\ell_i^n)^{-1}\#\beta_i\#\ell^n_{f(i)}\lift{f}{\beta_i(1)}$ is
  trivial rel $(A,\ord)$ for all sufficiently large $n$ by
  Corollary~\ref{lem:IrratOfDynam2}.
\end{proof}

At one extreme, homotopy pseudo-orbits include poly-orbits,
represented as constant paths. At the other extreme, homotopy
pseudo-orbits may be assumed to consist of paths all starting at the
basepoint $*$. As such, these paths represent elements of the biset
$B(f)$ of $f$, see~\ref{def:B(f)}.

\subsection{Symbolic orbits}
We shall be interested in marking periodic orbits of regular
points. These are conveniently encoded in the following simplification
of portraits of bisets (in which the subbisets are singletons and
therefore represented simply as elements):
\begin{defn}[Symbolic orbits]\label{def:symbolic}
  Let $I$ be a finite index set with self-map $f\colon I\selfmap$, and let
  $B$ be a $G$-biset. An \emph{$I$-symbolic orbit} is a sequence
  $(b_i)_{i\in I}$ of elements of $B$, and two $I$-symbolic orbits
  $(b_i)_{i\in I}$ and $(c_i)_{i\in I}$ are \emph{conjugate} if there
  exists a sequence $(g_i)_{i\in I}$ in $G$ with
  $g_i c_i=b_i g_{f(i)}$ for all $i\in I$.
\end{defn}

\begin{lem}
  Every homotopy pseudo-orbit can be conjugated to a symbolic orbit,
  unique up to conjugacy, in $B(f,A,\ord)$.
\end{lem}
\begin{proof}
  Given $(\beta_i)_{i\in I}$ a homotopy pseudo-orbit, choose paths
  $\ell_i$ from $\beta_i(0)$ to $*$ and define
  $\gamma_i=\ell_i^{-1}\#\beta_i\#\ell_{f(i)}\lift{f}{\beta_i(1)}$. Then
  $\gamma_i\in B(f,A,\ord)$ and
  $(\beta_i)_{i\in I}\sim(\gamma_i)_{i\in I}$. Furthermore another
  choice of paths $(\ell'_i)_{i\in I}$ differs from
  $(\ell_i)_{i\in I}$ by $\ell'_i=\ell_i g_i$ for some $g_i$, so the
  symbolic orbits $(\gamma_i)_{i\in I}$ and
  $((\ell'_i)^{-1}\#\beta_i\#\ell'_{f(i)}\lift{f}{\beta_i(1)})_{i\in I}$
  are conjugate.
\end{proof}

Let $f\colon(S^2,P_f,\ord_f)\selfmap$ be an expanding map, and let
$\subscript G B_G$ be its biset. Recall that the Julia set $\Julia(f)$
of $f$ is the accumulation set of preimages of a generic point $*$,
\[\Julia(f)\coloneqq\overline{\bigcap_{n\ge0}\bigcup_{m\ge n}f^{-m}(*)}.
\]
Every bounded sequence $b_0 b_1\cdots\in B^{\otimes\infty}$ defines an
element of $\Julia(f)$ as follows: set $c_0=b_0$ and
$c_i=c_{i-1}\#b_i\lift{f^i}{c_{i-1}(1)}$ for all $i\ge1$; then
$\lim_{n\to\infty}c_n(1)$ converges to a point
$\rho(b_0 b_1\cdots)\in\Julia(f)$. The following proposition directly
follows from the definition:
\begin{prop}[Expanding case]\label{prop:ShadExpCase}
  Suppose $f\colon(S^2,P_f,\ord_f)\selfmap$ is an expanding map with
  orbisphere biset $\subscript G B_G$, and let $(b_i)_{i\in I}$ be a
  finite symbolic orbit. Let $\Sigma$ be a generating set of
  $\subscript G B_G$ containing all $b_i$ and let
  $\rho\colon\Sigma^{+\infty}\to\Julia(f)$ be the symbolic encoding
  defined above. Then $(b_i)_{i\in I}$ shadows
  $(\rho(b_i b_{f(i)}) b_{f^2(i)}\dots)_{i\in I}$.\qed
\end{prop}
This proposition is useful to solve shadowing problems (namely,
determining when two symbolic orbits shadow the same poly-orbit) using
language of automata.

It also follows from the proposition that every orbit homotopy
pseudo-orbit shadows a unique orbit in $\Julia(f)$.

\begin{prop}[Shadowing and universal covers]\label{prop:ShadViaUnCover}
  Let $f\colon(S^2,A,\ord)\selfmap$ be an orbisphere map with biset
  $\subscript G B_G$, let
  $\pi\colon\subscript G{\widetilde U}\to (S^2, A,\ord)$ be the
  universal covering map of $(S^2, A,\ord)$, and let
  $\Phi\colon\subscript G B_G \otimes \subscript G{\widetilde
    U}\hookrightarrow\subscript G{\widetilde U}$ be the $G$
  equivariant map defined by Proposition~\ref{prop:Gspaces}.

  Then there is a completion $\subscript G{\widetilde U^+}$ of
  $\subscript G{\widetilde U}$ such that $\pi$ and $\Phi$ extend to
  continuous maps
  $\pi\colon\subscript G{\widetilde U^+}\to S^2\setminus A^\infty$ and
  $\Phi\colon\subscript G B_G \otimes \subscript G{\widetilde
    U^+}\to\subscript G{\widetilde U^+}$ with the following property:
  a symbolic orbit $(b_i)_{i\in I}$ shadows an orbit $(x_i)_{i\in I}$
  in $S^2\setminus A^\infty$ if and only if
  \[\Phi(b_i\otimes \widetilde x_{f(i)})=\widetilde x_i \text{ for some }\widetilde x_i\in \widetilde U^+ \text{ with }\pi(\widetilde x_i)=x_i.
  \]
  Furthermore, if $\ord(a)<\infty$ for every
  $a\in A\setminus A^\infty$ then we may take $\widetilde U^+=\widetilde U$.
\end{prop}
\begin{proof}
  To define $\widetilde U^+$, add to $\widetilde U$ all limit points
  of parabolic elements corresponding to small loops around punctures
  $a\in A\setminus A^\infty$ with
  $\ord(a)=\infty$. The extension of $\pi$ and
  $\Phi$ is given by continuity; $\pi$ is a branched covering with
  branch locus $\widetilde U^+\setminus\widetilde U$.

  If $(b_i)_{i\in I}$ shadows $(x_i)_{i\in I}$, then there are
  curves
  \[\widetilde x_i\colon[0,1]\to (S^2,A,\ord)\text{ with }\widetilde x_i(0)=*
    \text{ and }\widetilde x_i(1)=x_i
  \]
  such that
  $\widetilde x_i^{-1}\#b_i\#\widetilde x_{f(i)}\lift{f}{b_i(1)}$ is
  a homotopically trivial loop. This exactly means that
  $\Phi(b_i\otimes\widetilde x_{f(i)})=\widetilde x_i$. Conversely, if
  $\Phi(b_i\otimes\widetilde x_{f(i)})=\widetilde x_i$ for all
  $i\in I$ then $(b_i)_{i\in I}$ shadows
  $(\pi(\widetilde x_i))_{i\in I}$.
\end{proof}

\begin{prop}[Shadowing in the $(2,2,2,2)$ case]\label{prp:ShadIn2222Case}
  Using notations of Proposition~\ref{prop:EndOfG2}, suppose $f$ is a
  $(2,2,2,2)$ geometric non-invertible map and
  \[\Phi\colon\subscript G B_G\otimes\subscript{\langle \Z^2,-z+r \rangle}{\R^2}\to\subscript{\langle \Z^2,-z+r \rangle}{\R^2}\colon(b_0, z)\mapsto M^{-1}z
  \]
  is as in Proposition~\ref{prop:EndOfG2}(D). Then for every symbolic
  orbit $(b_i)_{i\in I}$ of $\subscript G B_G$ there is a unique
  collection $(r_i)_{i\in I}$ of points in $\R^2$ such that
  $r_i= \Phi(b_i \otimes r_{f(i)})$.  The points $r_i$ are solutions
  of linear equations~\eqref{eq:r_iInR2}.
\end{prop}
By Proposition~\ref{prop:ShadViaUnCover}, the image of
$(r_i)_{i\in I}$ under the projection
$\pi\colon\R^2\to\R^2/G\approx (S^2,P_f,\ord_f)$ is the
unique poly-orbit shadowed by $(b_i)_{i\in I}$.

\noindent The proof will use the following easy fact.
\begin{lem}\label{lem:EigenValOfM-I}
  If $M$ is non-invertible geometric, then
  $|\det(M^n+\epsilon I)|\ge 1$ for every $n\ge 1$ and every
  $\epsilon\in\{\pm 1\}$.
\end{lem}
\begin{proof}
  If $\lambda_1,\lambda_2$ are $M$'s eigenvalues, then
  Lemma~\ref{lem:CharactTorusMaps} gives
 \[|\det(M^n+\epsilon\one)| = |(\lambda^n_1+\epsilon )(\lambda^n_2+\epsilon )| \ge 1.\qedhere\]
\end{proof}

\begin{proof}[Proof of Proposition~\ref{prp:ShadIn2222Case}]
  Write $b_i = b_0g_i$. Recall that every $g_i$ acts on $\R^2$ as
  $z\mapsto \epsilon_i z+ t_i$ with $\epsilon_i\in\{\pm 1\}.$ We get
  the following system of linear equations:
  \[M^{-1}(\epsilon_i r_{f(i)}+ t_i) = r_i,\qquad i\in I.
  \]
  By splitting $f\colon I\selfmap$ into grand orbits and eliminating
  variables we arrive at equations of the form
  \begin{equation}\label{eq:r_iInR2}
    (\one+\theta M^n) r_i = t\text{ with }\theta\in\{\pm 1\}
  \end{equation}
  and $n$ the period of an orbit and $t\in\R^2$ some parameters
  depending on $\epsilon_i$ and $t_i$. By
  Lemma~\ref{lem:EigenValOfM-I}, the system~\eqref{eq:r_iInR2} has a
  unique solution.
\end{proof}

Consider an \emph{$I$-symbolic orbit} $(b_i)_{i\in I}$ shadowing a
poly-orbit $(x_i)_{i\in I}$. This means (see
Definition~\ref{defn:Shadowing}) that there are curves
$(\ell_i)_{i\in I}$ connecting $*$ to $x_i$ such that
$\ell_i^{-1}\#b_i\#\ell_{f(i)} \lift{f}{b_i(1)}$ is a trivial loop
rel $(A,\ord)$. The \emph{local group} $G_{x_i} \le \pi_1(X,*)=G$
consists of loops of the form
\begin{equation}
\label{eq:ell-1 alpha ell}
\ell_i[0,1-\varepsilon] \#\alpha\#\ell_i^{-1}[0,1-\varepsilon],\qquad\alpha \text{ is close to }x_i.
\end{equation}

If $x_i\not\in A$, then $G_{x_i}$ is a trivial group; otherwise
$G_{x_i}$ is an abelian group of order $\ord (x_i)$. If
$(A,\ord)=(P_f,\ord_f)$, then $G_{x_i}$ is a finite abelian
group. Clearly, $(g_i b_i h_i)_i$ also shadows $x_i$ for all
$g_i\in G_{x_i}$ and all $h_i\in G_{x_{i+1}}$. Conversely:
\begin{lem}\label{lem:orb shad the same pnt}
Let $f\colon(S^2,A,\ord)\selfmap$ be a geometric non-invertible map.
  Suppose that the symbolic orbits $(b_i)_{i\in I}$ and
  $(c_i)_{i\in I}$ shadow $(x_i)_i$. Let $G_{x_i}$ be the local group
  associated with $(b_i)_i$. Then there are $h_i\in G_{x_{f(i)}}$ such
  that $(c_i)_{i\in I}$ is conjugate to $(b_i h_{i})_{i\in I}$.

  If $I$ consists only of periodic indices, then $(c_i)_{i\in I}$ is
  conjugate to $(g_i b_i)_{i\in I}$ for some $g_i\in G_{x_{i}}$.
\end{lem}
\begin{proof}
  By conjugating $(c_i)_{i\in I}$ we can assume that
  $\ell_i^{-1} c_i \ell_i \lift{f}{b_i(1)}$ is a trivial loop; i.e.\
  local groups associated with $(c_i)_{i\in I}$ coincide with whose
  associated with $(b_i)_{i\in I}$. It is now routine to check that
  $(b_i)_{i\in I}$ and $(c_i)_{i\in I}$ differ by the action of local
  groups. Indeed, $b_i$ is of the form
  \[\ell_i[0,1-\varepsilon] \#\beta_i\#(\ell_{f(i)}
    [0,1-\varepsilon])^{-1} \lift f {\beta_i(1)},\qquad\beta_i\text{ is close to }x_i,
  \]
  and, similarly, $c_i$ is of the form
  \[\ell_i[0,1-\varepsilon] \#\gamma_i\#(\ell_{f(i)} [0,1-\varepsilon])^{-1} \lift f {\gamma_i(1)},\qquad\gamma_i\text{ is close to }x_i.\]
  Let $\alpha_i$ be the image of $\beta^{-1}_i\#\gamma_i$ and define
  $h_i$ to be
  $\ell_{f(i)} [0,1-\varepsilon] \#\alpha_{i}\#
  \ell_{f(i)}^{-1}[0,1-\varepsilon]$; compare with~\eqref{eq:ell-1
    alpha ell}. Then $c_i =b_i h_{i}$ for all $i\in I$.

  If $I$ has only periodic indices, then $\beta^{-1}_i$ and $\gamma_i$
  end at the same point and we set
  $\alpha_i\coloneqq \gamma_i\#\beta^{-1}_i$ and
  $g_i\coloneqq \ell_{i} [0,1-\varepsilon] \#\alpha_{i}\#
  \ell_{i}^{-1}[0,1-\varepsilon]$. We obtain $c_i=g_i b_i$ for all
  $i\in I$.
\end{proof}

\begin{cor}\label{cor:shaw is finite to one} Let $f\colon(S^2,P_f,\ord_f)\selfmap$ be a geometric non-invertible map. For every poly-orbit there are only finitely many symbolic orbits
  that shadow it.  Therefore, there are only
  finitely many conjugacy classes of portraits in $B(f)$.\qed
\end{cor}

\begin{thm}\label{thm:ShadowEquivalence}
  Let $f\colon(S^2,A,\ord)\selfmap$ be a non-invertible geometric
  map. Then the shadowing operation defines a map from conjugacy
  classes of symbolic finite orbits onto poly-orbits in
  $S^2\setminus A^\infty$. If $(A,\ord)=(P_f,\ord_f)$, then the shadowing map is finite-to-one.
\end{thm}
\begin{proof}
  Follows from
  Propositions~\ref{prop:ShadAndExtraPts},~\ref{prop:ShadExpCase},~\ref{prp:ShadIn2222Case}
  and Corollary~\ref{cor:shaw is finite to one}.
\end{proof}

\subsection{From symbolic orbits to portraits of bisets}
The \emph{centralizer} of a symbolic finite orbit $(b_i)_{i\in I}$ is
the set of $(g_i)_{i\in I}\in G^I$ such that $g_i b_i=b_i g_{f(i)}$
for all $i\in I$.
\begin{lem}\label{lem:Zsfi}
  If $\subscript G B_G$ is the biset of a geometric non-invertible map
  $f$ and $(b_i)_{i\in I}$ is a symbolic finite orbit shadowing a
  poly-orbit $(x_i)_{i\in I}$, then its centralizer is contained in
  $\prod_{i\in I}G_{x_i}$, where $G_{x_i}$ are the local groups
  associated with $(b_i)_{i\in I}$.
\end{lem}
\begin{proof}
  By Definition~\ref{defn:Shadowing} there are curves
  $(\ell_i)_{i\in I}$ connecting $*$ to $x_i$ such that
  $\ell_i^{-1}\#b_i\#\ell_i \lift{f}{b_i(1)}$ is a trivial loop rel
  $(A,\ord)$. Suppose that $(g_i)_{i\in I}\in G^I$ centralizes
  $(b_i)_{i\in I}$. Set
  $\wtg_i \coloneqq \ell_i^{-1}\#g_i \#\ell_i$. Then
  $\wtg_{f(i)}\lift{f}{x_i}$ is isotopic to $\wtg_i$
  rel $(A,\ord)$. By the expanding property of $f$, or
  Corollary~\ref{lem:IrratOfDynam2} if $f$ is exceptional,
  $\wtg_i$ is trivial rel $(A,\ord)$. This implies that
  $(g_i)_{i\in I} \in\prod_{i\in I}G_{x_i}$.
\end{proof}

Consider a portrait of bisets $(G_a,B_a)_{a\in \wtA}$ in
$\subscript G B_G$. We may decompose $\wtA=A\sqcup F\sqcup I$
with $f_*^n(F)\subseteq A$ for $n\gg0$ and $f_*(I)\subseteq I$. Then for every
$i\in I$ the group $G_i$ is trivial and $B_i=\{b_i\}$ is a
singleton. We obtain the symbolic orbit $(b_i)_{i\in I}$ which is the
essential part of $(G_a,B_a)_{\wtA}$:
\begin{lem}\label{lem:Zportrait=Zhpo}
  The relative centralizer $Z_D((G_a,B_a)_{a\in \wtA})$
  (see~\S\ref{ss:diff centr}) is isomorphic to the centralizer of
  $(b_i)_{i\in I}$ via the forgetful map
  $(g_{d})_{d\in D}\to (g_i)_{i\in I}$.

  Let $(G_a,B_a)_{a\in \wtA}$ and
  $(G'_a,B'_a)_{a\in \wtA}$ be two portraits of bisets with
  associated symbolic orbits $(b_i)_{i\in I}$ and $(b'_i)_{i\in
    I}$. Assume that $G_a=G'_a$ and $B_a=B'_a$ for all $a\in A$. Then
  $(G_a,B_a)_{a\in \wtA}$ and
  $(G'_a,B'_a)_{a\in \wtA}$ are conjugate if and only if
  \begin{itemize}
  \item[(1)] $(b_i)_{i\in I}$ and $(b'_i)_{i\in I}$ are conjugate; and
  \item[(2)] $G\otimes_{G_d} B_d=G\otimes_{G_d} B'_d$ for every
    $d\in F$.
  \end{itemize}
\end{lem}
This reduces the conjugacy problem of portrait of bisets to the
conjugacy problem of symbolic orbits; indeed Condition~(2) is easily
checkable. Note that every preperiodic point in $\wtA$ imposes a
finite condition on conjugacy and centralizers: for points attracted
to $A$, Condition~(2) imposes a congruence condition modulo the action
on $\{\cdot\}\otimes_G B$ on the conjugator; for preperiodic points in
$I$, conjugacy again amounts (by Condition~(1) and
Definition~\ref{def:symbolic}) to a congruence condition modulo the
action on $\{\cdot\}\otimes_G B$.
\begin{proof}
  If $d\in D$ with $f_*(d)\not\in D$, then from
  $g_d B_d=B_d g_{f_*(d)}=B_d$ follows that $g_d=1$. Induction on
  the escaping time to $A$ gives $g_d=1$ for all $d\in F$.

  Similarly, if $d\in D$ with $f_*(d)\not\in D$, then there is a
  $g_d\in G$ with $g_d B_d=B'_d= B'_d g_{f_*(d)}$ if and only if
  $G\otimes_{G_d} B_d=G\otimes_{G_d} B'_d$.  By induction on the
  escaping time, $(G_a,B_a)_{a\in A\sqcup F}$ and
  $(G'_a,B'_a)_{a\in A\sqcup F}$ are conjugate if and only if
  Condition~(2) holds.
\end{proof}

We are now ready to show that a geometric map, equipped with a
portrait of bisets, yields a sphere map --- possibly with some points
infinitesimally close to each other. A \emph{blowup} of a
two-dimensional sphere is a topological sphere $\widetilde S^2$
equipped with a monotone map $\widetilde S^2\to S^2$, namely a
continuous map under which preimages of connected sets are
connected. We remark that arbitrary countable subsets of $S^2$ may be
blown up into disks.
\begin{prop}\label{prop:Shadow}
  Let $f\colon (S^2,A)\selfmap$ be a non-invertible geometric map, let
  $\wtA$ be a set containing a copy of $A$, let
  $f_*\colon\wtA\selfmap$ be a symbolic map which coincides
  with $f$ on the subset $A\subseteq\wtA$, and let
  $(B_a)_{a\in\wtA}$ be a portrait of bisets in $B(f)$.

  Then there exists a unique map $e\colon\wtA\to S^2$
  extending the identity on $A$, and a blowup
  $b\colon \widetilde S^2\to S^2$, with the following properties. The
  locus of non-injectivity of $e$ is disjoint from $A^\infty$, and $b$
  blows up precisely at the grand orbits of
  \[\{x\mid \exists a\neq a'\in \wtA\colon x=e(a)=e(a')\}\]
  replacing points by disks on which the metric is identically $0$ and
  all points in $\wtA\subset\widetilde S^2$ are disjoint. The
  maps $f_*\colon\wtA\selfmap$ and $f\colon(S^2,A)\selfmap$
  extend to a map $\wtf\colon(\widetilde S^2,\wtA)\selfmap$,
  which is semiconjugate to $f$ via $b$, and whose minimal portrait of
  bisets projects to $(B_a)_{a\in\wtA}$.
\end{prop}
\begin{proof}
  Let us set $G\coloneqq\pi_1(S^2\setminus A,*)$.  We may decompose
  $\wtA=A\sqcup I\sqcup J$ with $f_*(I)=I$ and
  $f_*^n(J)\subseteq A\sqcup I$ for $n\gg0$. On $A$, we naturally
  define $e$ as the identity. If $i\in I$, then $i$ is $f_*$-periodic
  and the bisets $B_i,B_{f_*(i)},\dots$ determine a homotopy
  pseudo-orbit which shadows a unique periodic poly-orbit in $S^2$, by
  Theorem~\ref{thm:ShadowEquivalence}. Thus $e$ is uniquely defined.

  We now blow up the grand orbits of all points in $S^2$ which are the
  image of more than one point in $ A\sqcup I$ under $e$, replacing
  them by a disk on which the metric is identically $0$. We inject
  $ A\sqcup I$ arbitrarily in the blown-up sphere $\widetilde S^2$,
  and now identify $A\sqcup I$ with its image in the blowup.

  Let us next extend $f$ to a self-map $\wtf$ of
  $\widetilde S^2$ so that $(B_a)_{a\in A\sqcup I}$ is the induced
  portrait of bisets. We first do it by arbitrarily mapping the disks
  to each other by homeomorphisms restricting to $f_*$ on
  $ A\sqcup I$. Let $(B'_a)_{a\in A\sqcup I}$ be the projection of the
  minimal portrait of bisets of $\wtf$ via
  $\pi_1(\widetilde S^2,A\sqcup I) \to G$. Consider $i\in I$. By
  Lemma~\ref{lem:orb shad the same pnt} (the periodic case), we can
  assume that $b_i=h_i b'_i$ with $h_i\in G_{e(i)}$. (Note that if
  $e(i)\not\in A$, then $h_i=\one$ and $b_i=b'_i$.)  Let $m'\in \Braid(\widetilde S^2\setminus A,I)$ be a
  preimage of $(h_i)_{i\in I}$ under $\Antipush_{I}$ (see~\eqref{eq:antipush}) and set $m\coloneqq\Push(m') $; it is defined up to pre-composing with a
  knitting element. Since $h_i\in G_{e(i)}$, we can assume that $m$ is
  identity away from the blown up disks. Then $(B_a)_{a\in A\sqcup I}$
  is a portrait of bisets induced by $m\wtf$.

  Consider next $j\in J$ and assume that $f_*(j)$ has already been
  defined. The biset $B_j$, and more precisely already its image in
  $\{\cdot\}\otimes_G B(f)$, see Definition~\ref{dfn:PrtrOfBst}(B),
  determines the correct $\wtf$-preimage of $f_*(j)\in \widetilde S^2$
  that corresponds to $j$, and thus determines $e(j)$ uniquely.
\end{proof}

\subsection{Promotions of geometric maps}
Let $X\coloneqq(S^2,A,\ord)$ be an orbisphere, and consider a subset
$D\subset S^2\setminus A$. Recall from~\eqref{eq:modxd} the quotient
$\Mod(X|D)$ of mapping classes of $(S^2,A\sqcup D)$ by knitting
elements; there, we called two maps $f,g\colon(X,D)\selfmap$
\emph{knitting-equivalent}, written $f\approx_{A|D}g$, if they differ
by a knitting element in $\knBraid(X,D)$. We shall show that
knitting-equivalent rigid maps are conjugate rel $A\sqcup D$:
\begin{thm}[Promotion]
\label{thm:promotion}
  Suppose $f,g\colon(X,D)\selfmap$ are orbisphere maps, and
  assume that either
  \begin{itemize}
  \item $g$ is geometric non-invertible, or
  \item $g^m(D)\subseteq A$ for some $m\ge 0$.
  \end{itemize}
  Then every conjugacy $h\in\Mod(X|D)$ between $f$ and $g$ rel $A|D$
  lifts to a unique conjugacy $\widetilde h\in\Mod(S^2,A\sqcup D)$
  between $f$ and $g$ rel $A\sqcup D$ such that
  $\widetilde h \approx_{A|D}h$.
\end{thm}
\begin{proof}
  We begin by a
  \begin{lem}\label{lmm:TrvOfLift2Cases}
    Under the assumption on $g$, consider that
    $b\in\Braid(S^2\setminus A,D)$. If for every $m\ge 1$ the element
    $b$ is liftable through $g^m$, then $(g^*)^m b$ is trivial for all
    $m\succ \log |g|$, where $|\ |$ is the word metric.
  \end{lem}
  \begin{proof}
    Consider first the case that $g$ is expanding. Then the lengths of
    curves $(g^*)^m(b)(-,a)$ tend to $0$ as $m\to\infty$, so
    $(g^*)^m(b)=1$ for sufficiently large $m$. If $g$ is exceptional,
    then Corollary~\ref{lem:IrratOfDynam2} replaces the expanding
    argument. If finally $g^m(E)\subseteq A$, then $(g^*)^m(b)=1$ for
    the same $m$.
  \end{proof}

  We resume the proof. Let $\widetilde h_0\in \Mod[p](S^2,A\sqcup D)$
  be any preimage of $h$ under the forgetful map
  $\Mod(S^2,A\sqcup D)\to\Mod(X|D)$. Then
  $\widetilde h_0f \widetilde h_0^{-1}\approx_{A\sqcup D} h_1^{-1}g$
  for some $h_1\in \knBraid(X,D)$. Setting
  $\widetilde h_1\coloneqq h_1 \widetilde h_0$ we get
  $\widetilde h_1f \widetilde h_1^{-1}\approx_{A\sqcup D}h_2^{-1}g$,
  where $h_2^{-1}$ is the lift of $h_1^{-1}$ through $g$. Continuing
  this process we eventually get
  $\widetilde h_m f \widetilde h_m^{-1}\approx_{A\sqcup D} g$ because
  a the corresponding lift of $h_1^{-1}$ is trivial by
  Lemma~\ref{lmm:TrvOfLift2Cases}. We have shown the existence of
  $\widetilde h$.

  If $\widetilde h'$ is another promotion of $h$, then $\widetilde h'$
  and $\widetilde h$ differ by a knitting element commuting with $f$.
  By Lemma~\ref{lmm:TrvOfLift2Cases}, that knitting element is
  trivial, hence $\widetilde h'=\widetilde h$.
\end{proof}

\begin{cor}\label{cor:conj probl reduction}
  Let
  $\subscript{\Mod(S^2,A\sqcup D)}{M(f)}_{\Mod(S^2,A\sqcup
    D)}\to\subscript{\Mod(X|D)}{M(f|D)}_{\Mod(X|D)}$ be the inert map
  forgetting the action of knitting elements (see
  Proposition~\ref{prop:inertMCB}).  Suppose that either $f$ is
  geometric non-invertible or $f^{m}(D)\subseteq A$ for some $m\ge
  0$. Then $f,g$ are conjugate in $M(f)$ if and only if their images
  are conjugate in $M(f|D)$. Moreover, the centralizers $Z(f)$ and
  $Z(f|D)$ (see~\S\ref{ss:diff centr}) are naturally isomorphic via
  the projection $\Mod(S^2,A\sqcup D)\to\Mod(X|D)$.  \qed
\end{cor}

\subsection{Automorphisms of bisets}
Recall that the automorphism group of a biset $\subscript H B_G$ is
the set $\Aut(B)$ of maps $\tau\colon B\selfmap$ satisfying
$\tau(h b g)=h\tau(b)g$ for all $h\in H,g\in G$.
\begin{prop}\label{prop:NoGhostAut}
  If $\subscript H B_G$ is a left-free, right-transitive biset and $H$
  is centreless, then $\Aut(B)$ acts freely on $B$, so $B/\Aut(B)$ is
  also a left-free, right-transitive $H$-$G$-biset.
\end{prop}
\begin{proof}
  $\Aut(B)$ acts by permutations on $\{\cdot\}\otimes_H B$, and
  commutes with the right $G$-action, which is transitive, so
  $\Aut(B)/\ker(\text{action})$ acts freely. Consider now
  $\tau\in\Aut(B)$ that acts trivially on $\{\cdot\}\otimes_H B$, and
  consider $b\in B$. We have $\tau(b)=t b$ for some $t\in H$, and for
  all $h\in H$ there exists $g\in G$ with $h b=b g$, because $B$ is
  right-transitive. We have
  \[h t b=h\tau(b)=\tau(h b)=\tau(b g)=\tau(b)g=t b g=t h b,
  \]
  so $t\in Z(H)=1$ and therefore $\tau=\one$. It follows that $\Aut(B)$
  acts freely on $B$.
\end{proof}
\begin{cor}[No biset automorphisms]\label{cor:NoGhostAut}
  For an orbisphere biset $\subscript H B_G$ with $H$ is non-cyclic,
  we have $\Aut(B)=1$.
\end{cor}
\noindent In the dynamical situation $\subscript G B_G$ is a cyclic
biset if and only if $G$ is a cyclic group.
\begin{proof}
  Since $B$ is an orbisphere biset, it is left-free and
  right-transitive; and since it is not cyclic, $H$ is a non-abelian
  orbisphere group and in particular is centreless, so
  Proposition~\ref{prop:NoGhostAut} applies. Now $B$ cannot have a
  proper quotient, because conjugacy classes of elements of $H$ appear
  exactly once in wreath recursions of peripheral elements,
  by~\cite{bartholdi-dudko:bc2}*{Definition~\ref{bc2:dfn:SphBis}\eqref{bc2:cond:3:dfn:SphBis}
    of orbisphere bisets}.
\end{proof}

\begin{cor}\label{cor:UniqBisIsom}
  Suppose $f,g\colon (S^2,C,\ord)\to(S^2,A,\ord)$ are isotopic
  orbisphere maps. Let $B(f)$ and $B(g)$ be the bisets of $f$ and
  $g$ with respect to the same base points $\dagger\in S^2\setminus C$
  and $*\in S^2\setminus A$. Then there is a unique isomorphism
  between $B(f)$ and $B(g)$.
\end{cor}
\begin{proof}
  Consider an isotopy $(f_t)_{t\in[0,1]}$ rel $A$ from $f$ to $g$. It
  induces a continuous motion $B(f_t,A,\ord,*)$ of the bisets of
  $f_t$. Therefore, all $B(f_t,A,\ord,*)$ are isomorphic. This shows
  existence, the uniqueness follows from Corollary~\ref{cor:NoGhostAut}.
\end{proof}

\begin{thm}[Rigidity]\label{thm:rigidity}
  Suppose $f\colon(S^2,A,\ord)\selfmap$ is a geometric map. If
  $h\colon(S^2,A)\selfmap$ commutes with $f$, is the identity in some
  neighbourhood of $A^\infty$, and is isotopic to $\one$ rel $A$, then
  $h=\one$.
\end{thm}
\begin{proof}
  Fix a basepoint $*\in S^2\setminus A$ and the fundamental group
  $G=\pi_1(S^2\setminus A,*)$.  Since $h$ is isotopic to the identity
  rel $A$, there is a path $\ell\colon[0,1]\to S^2\setminus A$ from
  $*$ to $h(*)$ such that
  $\gamma\approx\ell\#(h\circ\gamma)\#\ell^{-1}$ for all
  $\gamma\in G$. Define then
  \[h_*\colon\subscript G{B(f)}_G\selfmap,\qquad
    h_*(\beta)\coloneqq\ell\#(h\circ\beta)\#(\ell^{-1})\lift f{h(\beta(1))}.
  \]
  Since $h$ commutes with $f$, this defines an automorphism of
  $B(f)$. By Corollary~\ref{cor:NoGhostAut}, it is the identity on
  $B(f)$. It also fixes, therefore, all conjugacy classes in
  $B(f)^{\otimes n}$ for all $n\in\N$. By
  Theorem~\ref{thm:ShadowEquivalence}, these are in bijection with
  periodic orbits of $f$. It follows that $h$ is the identity on all
  periodic points of $f$, and therefore on $f$'s Julia set. Since
  furthermore $h$ is the identity near $A^\infty$, and every point of
  $S^2\setminus A$ either belongs to $\Julia(f)$ or gets attracted to
  $A^\infty$, we get $h=\one$ everywhere.
\end{proof}

\subsection{Weakly geometric maps}
We are now ready to show that maps whose restriction to their minimal
orbisphere is geometric are of a particularly simple form.

\begin{defn}[Tunings]
  Let $f\colon(S^2,A)\selfmap$ be a Thurston map.  A \emph{tuning
    multicurve} for a $f$ is an $f$-invariant\footnote{In the sense
    that $f^{-1}(\CC)$ equals $\CC$ rel $A$.}  multicurve $\CC$ such
  that $f^{-1}(\CC)$ does not contain nested components rel
  $f^{-1}(A)$; namely, the adjacency graph of components of
  $S^2\setminus f^{-1}(\CC)$ is a star.

  A \emph{tuning} is an amalgam of orbisphere bisets in which the
  graph of bisets is a star. Equivalently, it is an amalgam of
  Thurston maps along a tuning multicurve.
\end{defn}

Every tuning has a \emph{central map}, corresponding to the centre of
the star, as well as \emph{satellite maps}, corresponding to its
leaves. Furthermore, unless the star has two vertices, its central map
is uniquely determined.

\begin{defn}[Weakly geometric maps]
  A \emph{tuning by homeomorphisms} is a tuning in which all satellite
  maps are homeomorphisms.

  A map is \emph{weakly geometric} if it is a (possibly trivial)
  tuning by homeomorphisms in such a manner that the central map is
  isotopic to a geometric map.
\end{defn}

\begin{thm}\label{thm:tuning}
  Let $f\colon(S^2,A)\selfmap$ be a non-invertible Thurston
  map. Then $f$ is weakly geometric if and only if its restriction
  $f\colon(S^2,P_f,\ord_f)\selfmap $ is isotopic to a geometric map.
\end{thm}
\begin{proof}
  If $f$ is weakly geometric, then its central map $\overline f$ is
  isotopic to a geometric map, and $f$ is isotopic to $\overline f$
  rel $P_f$.

  Conversely, assume that the restriction
  $\overline f\colon(S^2,P_f,\ord_f)\selfmap$ of $f$ is geometric, and
  consider its portrait of bisets $(B_a)_{a\in A}$ induced from the
  minimal portrait of bisets of $f$. By Proposition~\ref{prop:Shadow},
  there exists a blown-up sphere $\widetilde S^2$ and extension
  $\wtf$ of $f$, such that the portraits of bisets of $f$ and
  $\wtf$ are conjugate portraits of bisets in
  $B(\overline f)$. The tuning multicurve we seek is the boundary of
  the blowup disks.

  By Theorem~\ref{thm:portrait}, the bisets $B(f)$ and
  $B(\wtf)$ differ by a knitting element, so we may write
  $\wtf=m f$ for some
  $m\in\knBraid(S^2\setminus A,\wtA\setminus A)$. By
  Proposition~\ref{prop:knittinginert}, the class $m$ is liftable
  arbitrarily often through $\overline f$, and by
  Lemma~\ref{lmm:TrvOfLift2Cases} its lift becomes eventually trivial;
  in other words, we have a relation $n^{-1}\wtf n=m f$ for
  some $m$ with trivial image in $\Braid(S^2\setminus A)$, and
  therefore $n$ is a product of mapping classes in the infinitesimal
  disks. Its restriction to these disks yields the required
  homeomorphisms of the tuning.
\end{proof}

\noindent We now turn to the algebraic aspects of geometric maps.
\begin{defn}
  An orbisphere biset is \emph{geometric} if it is the biset of a
  geometric map. A biset $B$ is \emph{weakly geometric} if its
  projection $\overline B$ to its minimal orbisphere group is
  geometric.
\end{defn}

\begin{cor}
  A Thurston map $f$ is [weakly] geometric if and only if its biset is
  [weakly] geometric.\qed
\end{cor}

The following result was already obtained by Selinger and Yampolsky in
the case of torus maps~\cite{selinger-yampolsky:geometrization}:
\begin{cor}
  A non-invertible Thurston map is isotopic to a geometric map if and
  only if it is Levy-free.
\end{cor}
\begin{proof}
  Let $f$ be a non-invertible Thurston map. If $f$ admits a Levy
  cycle, then either $f$ shrinks no metric under which a curve in the
  Levy cycle is geodesic, or $f$ is doubly covered by a torus
  endomorphism with eigenvalue $\pm1$; in all cases, $f$ is not
  geometric.

  Conversely, if $f$ is Levy-free, then its restriction $\overline f$
  to $(S^2,P_f,\ord_f)$ is still
  Levy-free. By~\cite{bartholdi-dudko:bc4}*{Theorem~A} (for \Exp\
  maps) or Lemma~\ref{lem:Tor:geom=lf} (for \Tor\ maps) the map
  $\overline f$ is geometric. By Proposition~\ref{prop:Shadow}, there is
  a map $e\colon A\to S^2$ whose image is preserved by $\overline
  f$. If $e$ were not injective, there would be a Levy cycle for $f$
  consisting of curves surrounding elements of $A$ with same image
  under $e$. Now if $e$ is injective then the tuning constructed in
  Theorem~\ref{thm:tuning} is trivial, so $f$ is geometric.
\end{proof}

\subsection{Conjugacy and centralizer problem}
We are now ready to show how conjugacy and centralizer problems may be
solved in geometric bisets. We review the algebraic interpretation of
expanding maps:
by~\cite{bartholdi-dudko:bc4}*{Theorem~\ref{bc4:thm:main}}, a map is
expanding if and only if its minimal biset is \emph{contracting}. We
recall briefly that a left-free $G$-$G$-biset $B$ is contracting if,
for every basis $X\subseteq B$, there exists a finite subset
$N\subseteq G$ with the following property: for every $g\in G$ there
is $n\in\N$ such that $X^n g\in N X^n\subseteq B^{\otimes n}$.

The minimal such subset $N$ is called the \emph{nucleus} associated
with $(B,X)$. Note that different bases yield different nuclei, but
that finiteness of $N$ is one basis implies its finiteness for all
other bases, or even for all finite sets $X$ generating $B$ qua left
$G$-set.

The biset $B$ may be given by its \emph{structure map}
$X\times G\to G\times X$, written $(x,g)\mapsto(g@x, x^g)$; it is
defined by the equality $x g=(g@x)x^g$ in $B$. Then $B$ is contracting
with nucleus $N$ if for every $g\in G$ every sufficiently long
iteration of the maps $g\mapsto g@x$ eventually gives elements in $N$.

We may assume, without loss of generality, that $X N\subseteq N X$
holds. For finitely generated groups, there is perhaps more intuitive
formulation of contraction: there exists a proper metric on $G$ and
constants $\lambda<1,C$ such that $\|g@x\|\le\lambda\|g\|+C$ for all
$g\in G,x\in X$.
\begin{lem}\label{lem:cc in NS}
  Let $\subscript G B_G$ be a contracting biset; choose a basis $X$ of
  $B$, and let $N\subseteq G$ be the associated nucleus.

  Then every symbolic finite orbit $(b_i)_{i\in I}$ is conjugate to
  one in which $b_i\in N X$ for all $i\in I$.
\end{lem}
\begin{proof}
  Write every $b_i$ in the form $g_i x_i$ with $g_i\in G$ and
  $x_i\in X$. Conjugate the symbolic orbit by $(g_i)_{i\in I}$; then
  each $b_i$ becomes
  $g_i^{-1} b_i g_{f(i)}=x_i g_{f(i)}=g'_i x'_i\in B$ for some
  $g'_i\in G,x'_i\in X$. Note that each $g'_i$ is a state of some
  $g_i$. Conjugate again by $(g'_i)_{i\in E}$, etc.; after a finite
  number of steps, each $b_i$ will belong to $N X$.
\end{proof}

\begin{lem}\label{lem:finiteperiod}
  Let $f\colon(S^2,A,\ord)\selfmap$ be a geometric map. Then for every
  $n\in\N$ the number of $n$-periodic points of $f$ is finite.
\end{lem}
\begin{proof}
  If $f$ is doubly covered by a torus endomorphism $z\mapsto M z+q$,
  then its period-$n$ points are all solutions to
  $z\in(\one-M)^{-1}q+(\one-M^n)^{-1}\Z^2$; the image of this set
  under $\R^2\to\R^2/\langle\Z^2,-z\rangle$ is finite, of size at most
  $\det(\one-M^n)$.

  If $f$ is expanding, then consider its biset $B$, which is
  orbisphere contracting,
  see~\cite{bartholdi-dudko:bc4}*{Theorem~\ref{bc4:thm:main}}. By
  Lemma~\ref{lem:cc in NS}, there are finitely many conjugacy classes
  of period-$n$ finite symbolic orbits in $B$, and by
  Theorem~\ref{thm:ShadowEquivalence}, every symbolic orbit shadows a
  unique finite poly-orbit.
\end{proof}

\begin{thm}\label{thm:conjcentr}
  Let $\wtG\to G$ be a forgetful morphism of groups and let
  \[\subscript{\wtG}{\wtB}_{\wtG}\to\subscript G B_G\;\text{ and }\; \subscript{\wtG}{\wtB'}_{\wtG}\to\subscript G{B'}_G
  \]
  be two forgetful biset morphisms as in~\eqref{eq:forgetCD2}.
  Suppose furthermore that $\wtB$ is geometric of degree
  $>1$. Denote by $(G_a,B_a)_{a\in \wtA}$ and
  $(G'_a,B'_a)_{a\in \wtA}$ the portraits of bisets induced by
  $\wtB$ and $\wtB'$ in $B$ and $B'$ respectively.

  Then $\wtB,\wtB'$ are conjugate under
  $\Mod(\wtG)$ if and only if there exists $\phi\in\Mod(G)$
  such that $B^\phi\cong B'$ and the portraits
  $(G_a^\phi,B_a^\phi)_{a\in \wtA}$ and
  $(G'_a,B'_a)_{a\in A\sqcup E}$ are conjugate.

  Furthermore, the centralizer of the portrait
  $(G_a,B_a)_{a\in\wtA}$ is trivial, and the centralizer
  $Z(\wtB)$ of $\wtB$ is isomorphic, via the forgetful
  map $\Mod(\wtG)\to\Mod(G)$, to
  \[\big\{\phi\in Z(B)
    \,\big|\,(G_a^{\phi},B_a^{\phi})_{a\in \wtA}\sim(G_a,B_a)_{a\in
    \wtA}\big\}
  \]
  and is a finite-index subgroup of $Z(B)$.
\end{thm}

\begin{proof}
  If $\wtB,\wtB'$ are conjugate, then certainly their
  images and subbisets are conjugate. Conversely, let $\phi\in\Mod(G)$
  be such that $B^\phi\cong B'$ and the portraits
  $(G_a^\phi,B_a^\phi)_{a\in \wtA}$ and
  $(G'_a,B'_a)_{a\in A\sqcup E}$ are conjugate, and let
  $\wtf,\wtf'$ be maps realizing
  $\wtB,\wtB'$ respectively, with $f$ geometric. By
  Theorem~\ref{thm:portrait}, we have $\wtf'=m\wtf$
  for a knitting mapping class $m$. Since $m$ is liftable by
  Proposition~\ref{prop:knittinginert}, we have $m\wtf=f m'$
  for some knitting mapping class $m'$ and therefore
  $m'\wtf'(m')^{-1}=m'\wtf$. Repeating with $m'$, we
  obtain
  $m^{(k)}\cdots m'\wtf'(m^{(k)}\cdots
  m')^{-1}=m^{(k)}\wtf$ for all $k\in\N$, and $m^{(k)}=\one$
  when $k$ is large enough by
  Lemma~\ref{lmm:TrvOfLift2Cases}. Therefore,
  $\wtf,\wtf'$ are conjugate, and so are
  $\wtB,\wtB'$.

  The description of $Z(\wtB)$ follows. The centralizer of the
  portrait $(G_a,B_a)_{a\in\wtA}$ is trivial, because (since
  $\wtB$ is geometric) all points shadowed by bisets $B_a$ are
  unmarked in $G$.  The $Z(B)$-orbit of $(G_a,B_a)_{a\in\wtA}$
  contains finitely many conjugacy classes by
  Lemma~\ref{lem:finiteperiod}, since all of them are induced from a
  geometric biset, and therefore are encoded in periodic
  or pre-periodic points for a map realizing $B$.
\end{proof}

\section{Algorithmic questions}
We give some algorithms that decide whether two portraits of bisets
are conjugate, thereby reducing the conjugacy and centralizer problem
from orbisphere maps to their minimal orbisphere (with only the
post-critical set marked).

All the algorithms we describe make use of the symbolic description of
maps via bisets, and are inherently quite fast and practical. Their
precise performance, and implementation details, be studied in the
last paper~\cite{bartholdi-dudko:bc5} of this series.

There are two implementations of Algorithms~\ref{algo:conjportrait}
and~\ref{algo:listportrait}, one for \Tor\ maps and one for \Exp\
maps. We describe them in separate subsections.

We already gave the following algorithms:
\begin{algo}[\cite{bartholdi-dudko:bc4}*{Algorithms~\ref{bc4:algo:is2cover} and~\ref{bc4:algo:getparam}}]\label{algo:2222}
  \textsc{Given} an orbisphere biset $\subscript G B_G$,\\
  \textsc{Decide} whether $B$ is the biset of a map double covered by
  a torus endomorphism, and if so \textsc{Compute} parameters $M,q$ for a torus endomorphism $z\mapsto M z+q$ covering $B$.
\end{algo}

\begin{algo}[\cite{bartholdi-dudko:bc4}*{Algorithms~\ref{bc4:algo:istor} and~\ref{bc4:algo:isexp}}]\label{algo:geom}
  \textsc{Given} an orbisphere biset $\subscript G B_G$,\\
  \textsc{Decide} whether $B$ is the biset of a \Tor\ or an \Exp\ map, and in particular whether $B$ is geometric.
\end{algo}

\subsection{\Tor\ maps}
In the case of \Tor\ maps, we can decide, without access to an oracle,
whether two such maps are conjugate, and we can compute their
centralizers, as follows. We shall need the following fact:
\begin{thm}[Corollary of~\cite{grunewald:conjugacyarithmetic}]\label{thm:MatrConj}
  There is an algorithm deciding whether two matrices
  $M,N\in \Mat_2^+(\Z)$ are conjugate by an element $X\in \SL_2(\Z)$,
  and produces such an $X$ if it exists.

  There is an algorithm computing, as a finitely generated subgroup of
  $\SL_2(\Z)$, the centralizer of $M\in \Mat_2^+(\Z)$.\qed
\end{thm}

\begin{algo}\label{algo:minimaltor}
  \textsc{Given} $\subscript G B_G$, $\subscript G C_G$ two minimal \Tor\ bisets\\
  \textsc{Decide} whether $B$ and $C$ are conjugate by an element of $\Mod(G)$, and if so \textsc{construct} a conjugator; and \textsc{compute} the centralizer $Z(B)$, \textsc{as follows:}\\\upshape
  \begin{enumerate}
  \item If $B_*\neq C_*$ as maps on peripheral
    conjugacy classes, then return \texttt{fail}.
  \item Using Proposition~\ref{prop:EndOfG}(A), identify $G$ with
    $\Z^2\rtimes\{\pm1\}$, and with Algorithm~\ref{algo:2222} present
    $B$ and $C$ as $B_{M^v}$ and $B_{N^w}$ respectively, see~\eqref{eq:InjEndOfK}.
  \item Using Theorem~\ref{thm:MatrConj} check whether $M$ and $N$ are
    conjugate. If not, return \texttt{no}; otherwise find a conjugator $X$
    and compute the centralizer $Z(M)$ of $M$.
  \item Check whether there is a $Y\in Z(M)$ such that
    $B_{(Y X)^0}$ conjugates $B_{M^v}$ to $B_{N^w}$, as follows. The orbit of $B_{M^v}$ under conjugation of \[\{B_{Y^0}\mid Y\in Z(M), B_{(Y X)^0}\in \Mod(G)\}\] is finite and hence computable; so is its image under
    $X^0$. Check whether $B_{N^w}$ belongs to it; if not, return
    \texttt{no}, and if yes, return \texttt{yes} and the conjugator
    $B_{(Y X)^0}$.
  \item The centralizer of $B$ is
    \[\{B_{Y^0}\mid Y\in Z(M),B_{Y^0}\in \Mod(G), \text{ and }B_{Y^0}\text{ centralizes }B_{M^v}\},\] naturally embeds as a subgroup of
    finite index in $Z(M)$.
  \end{enumerate}
\end{algo}

\noindent Here is an algorithmic version of
Proposition~\ref{prp:ShadIn2222Case}:
\begin{algo}\label{algo:shadowing}
   \textsc{Given} a minimal \Tor\ orbisphere biset $\subscript G B_G$ with $G=\langle \Z^2,-z+r \rangle$, a base point $b_0\in B$ specifying the map
  \begin{equation}
  \label{eq:Phi:algo:shadowing}
   \Phi\colon\subscript G B_G\otimes\subscript{\langle \Z^2,-z+r \rangle}{\R^2}\to\subscript{\langle \Z^2,-z+r \rangle}{\R^2}\colon(b_0, z)\mapsto M^{-1}z
  \end{equation}
  (see Proposition~\ref{prp:ShadIn2222Case}), an extension $f_*\colon\wtA\to\wtA$ of the dynamics of $B$ on its peripheral classes, and a portrait of bisets $(G_a,B_a)_{a\in\wtA}$,\\
  \textsc{Compute} $(r_a)_{ a\in \wtA}$ shadowed by  $(G_a,B_a)_{a\in\wtA}$: $r_a= \Phi(B_a \otimes r_{f_*(a)})$, the local groups $G_{\bar r_a}$ where $\bar r_a$ is the image of $r_a$ in $\R^2/G$ (see~\eqref{eq:ell-1 alpha ell}), and the relative centralizer $Z_D((G_a,B_a)_{a\in\wtA})$, which is a finite
  abelian group, \textsc{as follows:}\\\upshape
   \begin{enumerate}
   \item Choose $b_a\in B_a$. Proceed through all periodic cycles $E$ of
    $I$. Choose $e\in E$ and solve the linear equation
    \[r_e =\Phi ^{|E|}\left(b_e\otimes b_{f_*(e)}\otimes \dots\otimes
        b_{f_*^{|E|-1}}(e) \otimes r_e\right);
    \]
    the equation takes form $ (\one+\theta M^n) r_i = t$ with
    $\theta\in\{\pm1\}$ (see~\eqref{eq:r_iInR2}) and has a unique
    solution by Lemma~\ref{lem:EigenValOfM-I}.
  \item Inductively compute $r_a=\Phi(b_a\otimes r_{f_*(a)})$ for all
    $a\in \wtA$.
  \item For $a\in A$ we have $G_{\bar r_a}=G_{a}$.
  \item For $a\in\wtA\setminus A$ check whether $r_a-z\in
    G$. If $r_a-z\not \in G$, then $G_{\bar r_a}$ is the trivial
    subgroup; otherwise $G_{\bar r_a}=\langle r_a-z\rangle$.
  \item By a finite check compute $Z_D((G_a,B_a)_{a\in\wtA})$:
    by Lemma~\ref{lem:Zportrait=Zhpo} it is the set of
    self-conjugacies of the corresponding homotopy pseudo-orbit, and
    by Lemma~\ref{lem:orb shad the same pnt} it is a subgroup of
    $\prod_{a\in\wtA}G_{\bar r_a}$, and is therefore an easily
    computable finite group.
  \end{enumerate}
\end{algo}

\begin{algo}\label{algo:conjportraittor}
  \textsc{Given} a minimal \Tor\ orbisphere biset $\subscript G B_G$, an extension $f_*\colon\wtA\to\wtA$ of the dynamics of $B$ on its peripheral classes, and two portraits of bisets $(G_a,B_a)_{a\in\wtA}$ and $(G'_a,B'_a)_{a\in\wtA}$,\\
  \textsc{Decide} whether $(G_a,B_a)_{a\in\wtA}$ and
  $(G'_a,B'_a)_{a\in\wtA}$ are conjugate \textsc{as follows:}
  \begin{enumerate}
  \item Normalize the portraits in such a manner that $G_a=G'_a$ and
    $B_a=B'_a$ for all $a\in A$; by Lemma~\ref{lem:CritPortrUnique}
    this follows from the conjugacy for subgroups: find
    $(\ell_a)_{a\in A}\in G^A$ with $G^{\ell_a}_a=G'_a$ and conjugate
    $(G_a,B_a)_{a\in\wtA}$ by $(\ell_a)_{a\in \wtA}$
    with $\ell_a=1$ if $a\not\in A$.
  \item Identify $G$ with $\langle \Z^2,-z+r \rangle$ and choose
    $b_0\in B$ characterizing the map $\Phi$,
    see~\eqref{eq:Phi:algo:shadowing}.
  \item Using Algorithm~\ref{algo:shadowing} compute the points $r_a$
    and $r'_a$ shadowed by $(G_a,B_a)_{a\in\wtA}$ and
    $(G'_a,B'_a)_{a\in\wtA}$ respectively.
  \item Check if $\bar r_a=\bar r'_a$ in $\R^2/G$. If
    $\bar r_a\neq\bar r'_a$ for some $a\in\wtA$, then return
    \texttt{no}.  Otherwise find $\ell_a\in G$ with $\ell_a r'_a= r_a$
    with $\ell_a=1$ for $a\in A$ and conjugate $(G'_a,B'_a)$ by
    $(r_a)_{a\in \wtA}$. This reduces to original problem to
    the case $r'_a=r_a$.
  \item Using Algorithm~\ref{algo:shadowing} compute the (finite
    abelian) local groups $G_{\bar r_a}$ and by a finite check decide
    if an element in $\prod_{a\in\wtA}G_{\bar r_a}$ conjugates
    $(G_a,B_a)_{a\in\wtA}$ to
    $(G'_a,B'_a)_{a\in\wtA}$.
  \label{algo:conjportraittor:CentrStep}
  \end{enumerate}
\end{algo}
If $(\bar r_a)_{a\in \wtA}$ is an actual orbit, then
$G_{\bar r_a}$ are trivial groups for all
$a\not\in \wtA\setminus A$ and
Step~\eqref{algo:conjportraittor:CentrStep} can be omitted. This is
the case when the algorithm is called from Algorithm~\ref{algo:red}.

\begin{algo}\label{algo:listportraittor}
  \textsc{Given} a minimal \Tor\ orbisphere biset $\subscript G B_G$ and an extension $f_*\colon\wtA\to\wtA$ of the dynamics of $B$ on its peripheral classes,\\
  \textsc{Produce a list} of all conjugacy classes of portraits of bisets $(G_a,B_a)_{a\in\wtA}$ in $B$ with dynamics $f_*$ \textsc{as follows:}\\\upshape
  \begin{enumerate}
  \item Write $G=\langle\Z^2,-z+r\rangle$ and $B=B_{M^v}$, using
    Algorithm~\ref{algo:2222}. Choose $b_0\in B$ characterizing the map $\Phi$, see~\eqref{eq:Phi:algo:shadowing}.
  \item Using Algorithm~\ref{algo:shadowing} compute the orbit $(\bar r_a)_{a\in  A}$ of $M\colon \R^2/G\selfmap$. Produce a list of all possible poly-orbits $(\bar r_a)_{a\in \wtA}$ extending $(\bar r_a)_{a\in  A}$.
  \item For every poly-orbit $(\bar r_a)_{a\in \wtA}$ find a portrait $(G_a,B_a)_{a\in \wtA}$ that shadow $(\bar r_a)_{a\in \wtA}$.
  \item Using Algorithm~\ref{algo:shadowing} compute the (finite) local groups $G_{\bar r_a}$. By Lemma~\ref{lem:orb shad the same pnt} the finite set
    \[\{(G_a,B_ah_a)_{a\in \wtA}\mid h_a\in G_{\bar r_a},
      h_a=1 \text{ if }a\in A\}
    \]
    contains a representative of every conjugacy class of portraits
    that shadows $(\bar r_a)_{a\in \wtA}$. Using
    Algorithm~\ref{algo:shadowing} produce a list of all conjugacy
    classes of portraits of bisets that shadow $(\bar r_a)_{a\in A}$.
  \end{enumerate}
\end{algo}

\subsection{\Exp\ maps}
We now turn to expanding maps, and start by a short example showing
how the conjugacy problem may be solved algorithmically. We note, by
following the proof of Lemma~\ref{lem:cc in NS}, that every portrait
of bisets can be algorithmically conjugated to one in which the terms
belong to $N X$, for $X$ a basis and $N$ the associated nucleus.
\begin{exple}
\label{ex:alg:conj of fixed s o}
  Suppose $E=\{e\}$ with $f_*(e)=e$, and let
  $(G_a,B_a)_{a\in A\sqcup E}$ and $(G_a,C_a)_{a\in A\sqcup E}$ be two
  portraits of bisets. Then $G_e=1$ and $B_e=\{b\}$ and $C_e=\{c\}$.

  The portraits $(G_a,B_a)_{a\in A\sqcup E}$ and $(G_a,C_a)_{a\in
    A\sqcup E}$ are conjugate if and only if there exists $\ell \in G$
  such that $\ell^{-1} b \ell=c$.

  Write $B_e=\{g x\}$ in a basis $X$ of $B$, with associated nucleus
  $N$; recall that $N$ is symmetric, contains $1$ and generates
  $G$. Then $g x$ is conjugate to $x g=(g@x)x^g=g' x'$. After
  iterating finitely many times the process $g x\sim x g=g' x'$, we
  can assume $g\in N$, as in Lemma~\ref{lem:cc in NS}. Similarly, we may replace $C_e$ by a conjugate
  biset $\{h y\}$ with $h\in N$.

  Find, by direct search, a $t\in\N$ with $N^{t-3}X\supseteq X N^t$;
  such a $t\ge 4$ exists because $B$ is contracting. Then $B_e$ and
  $C_e$ are conjugate if and only if there exists $\ell\in G$ with
  $b\ell = \ell c$, namely $g x\ell=\ell h y$. Let $u\in\N$ be such
  that $\ell\in N^u\setminus N^{u-1}$; if $u\ge t$ then
  $g x\ell\in N N^{u-3}X$ while $\ell h y\not\in N^{u-1} N X$, a
  contradiction. Therefore the search for a conjugator $\ell$ is
  constrained to $\ell\in N^{t-1}$.
\end{exple}

As in Example~\ref{ex:alg:conj of fixed s o} we have
\begin{lem}\label{lem:t is min}
  Let $\subscript G B_G$ be a contracting biset; choose a basis $X$ of
  $B$, and let $N\subseteq G$ be the associated nucleus. Suppose that
  $t\in\N$ is such that $N^{t-3}X\supseteq X N^t$.

  If two symbolic orbits $(b_i)_{i\in I}$ and $(c_i)_{i\in I}$ with
  $b_i,c_i\in N X$ are conjugate by $g_i$, then $g_i\in N^{t-1}$.
\end{lem}
\begin{proof}
  Suppose $g_i\in N^u\setminus N^{u-1}$ for $u\ge t$ and some $i$. We
  can also assume that $g_j\in N^u$ for all $j\in I$. Then on one
  hand, $g_i b_i\in N^u N X$; but on the other hand
  $g_i b_i=c_i g_{f_*(i)}\in N^{u-3}N X$. This is a contradiction.
\end{proof}

\begin{algo}\label{algo:conjportraitexp}
  \textsc{Given} a minimal \Exp\ orbisphere biset $\subscript G B_G$, an extension $f_*\colon\wtA\to\wtA$ of the dynamics of $B$ on its peripheral classes, and two portraits of bisets $(G_a,B_a)_{a\in\wtA}$ and $(G'_a,B'_a)_{a\in\wtA}$,\\
  \textsc{Decide} whether $(G_a,B_a)_{a\in\wtA}$ and
  $(G'_a,B'_a)_{a\in\wtA}$ are conjugate, and \textsc{Compute}
  the centralizer of $(G_a,B_a)_{a\in\wtA}$, which is a finite
  group, \textsc{as follows:}\\\upshape
  \begin{enumerate}
  \item Write $\wtA=A\sqcup J\sqcup I$ with
    $f_*^n(J)\subseteq A$ and $f_*(I)\subseteq I$. Normalize
    $(G_a,B_a)_{a\in\wtA}$ and
    $(G'_a,B'_a)_{a\in\wtA}$ such that $G_a=G'_a$ and
    $B_a=B'_a$; by Lemma~\ref{lem:CritPortrUnique} this follows from
    the conjugacy for subgroups: find $(\ell_a)_{a\in A}\in G^A$ with
    $G^{\ell_a}_a=G'_a$ and conjugate $(G_a,B_a)_{a\in\wtA}$
    by $(\ell_a)_{a\in \wtA}$ with $\ell_a=1$ if
    $a\not\in A$.
  \item Check whether there is $(\ell_a)_{a\in A\sqcup J}$ with
    $\ell_a=1$ for $a\in A$ conjugating $(G_a,B_a)_{a\in A\sqcup J}$
    and $(G'_a,B'_a)_{a\in A\sqcup J}$. If not, return
    \texttt{no}. This reduces the conjugacy problem of portraits to
    the conjugacy problem of symbolic orbits: writing $B_i=\{b_i\}$
    and $B'_i=\{b'_i\}$ for $i\in I$ solve a conjugacy problem of
    $(b_i)_{i\in I}$ and $(b'_i)_{i\in I}$.

  \item Write the biset $B$ in the form $G X$ for a basis $X$, and let
    $N$ be its nucleus. Find, by direct search, a $t\in\N$ with
    $N^{t-3}X\supseteq X N^t$; such a $t\ge4$ exists because $B$ is
    contracting.
  \item Write $b_i=g_ix_i$ and replace $b_i$ with $x_i g_i=g'_i
    x'_i$. After iterating finitely many times this process, we obtain
    $b_i\in N X$ by Lemma~\ref{lem:cc in NS}. By a similar iteration,
    we can assume $b'_i\in N X$.\label{st:4:algo:conjportraitexp}
  \item Answer whether $(b_i)_{i\in I}$ and $(b'_i)_{i\in I}$ are conjugate by elements in $N^{t-1}$. This is correct by Lemma~\ref{lem:t is min}.
  \item By a direct search compute the centralizer of
    $(G_a,B_a)_{a\in \wtA}$: the centralizer of
    $(G_a,B_a)_{a\in A\sqcup J}$ is trivial, while elements
    centralizing $(G_a,B_a)_{a\in I}$ are within $N^{t-1}$.
  \end{enumerate}
\end{algo}

\begin{algo}\label{algo:listportraitexp}
  \textsc{Given} a minimal \Exp\ orbisphere biset $\subscript G B_G$ and an extension $f_*\colon\wtA\to\wtA$ of the dynamics of $B$ on its peripheral classes,\\
  \textsc{Produce a list} of all conjugacy classes of portraits of bisets $(G_a,B_a)_{a\in\wtA}$ in $B$ with dynamics $f_*$ \textsc{as follows:}\\\upshape
  \begin{enumerate}
  \item Write $\wtA=A\sqcup J\sqcup I$ with $f_*^n(J)\subseteq A$
    and $f_*(I)\subseteq I$. Produce a list $\mathcal L_{A\sqcup J}$ of
    all conjugacy classes of portraits of bisets
    $(G_a,B_a)_{a\in A\sqcup J}$.
  \item Write the biset $B$ in the form $G X$ for a basis $X$, and let
    $N$ be its nucleus. Produce a list of all symbolic orbits
    $(b_i)_{i\in I}$ with $b_i\in N X$. Using
    Algorithm~\ref{algo:conjportraitexp} produce a list $\mathcal L_I$
    of all conjugacy classes of symbolic orbits $(b_i)_{i\in I}$ with
    $b_i\in N X$.
  \item Combine $\mathcal L_{A\sqcup J}$ and $\mathcal L_{I}$ to
    produce a list $\mathcal L_{A\sqcup J}\times \mathcal L_{I}$ of
    all conjugacy classes of portraits of bisets
    $(G_a,B_a)_{a\in\wtA}$ by setting $B_i\coloneq \{b_i\}$
    and $G_i\coloneq \{1\}$ for $i\in I$.
  \end{enumerate}
\end{algo}

\subsection{Decidability of conjugacy and centralizer problems}\label{ss:decidability}
The statements of Theorem~\ref{thm:RedConjCentrProb} can be turned
into an algorithm:
\begin{algo}\label{algo:red}
  \textsc{Given} $\subscript{\wtG}{\wtB}_{\wtG}$, $\subscript{\wtG}{\wtC}_{\wtG}$ two geometric bisets,\\
  \textsc{And given} an oracle that decides whether two minimal \Exp$\setminus$\Tor\ bisets are conjugate, and computes their centralizers,\\
  \textsc{Decide} whether $\wtB$ and $\wtC$ are conjugate by an element of $\Mod(\wtG)$, and if so \textsc{construct} a conjugator; and   \textsc{compute} the centralizer $Z(\wtB)$, \textsc{as follows:}\\\upshape
  \begin{enumerate}
  \item If $\wtB_*\neq \wtC_*$ as maps on peripheral
    conjugacy classes, then return \texttt{no}.
  \item Let
    $\subscript{\wtG}{\wtB}_{\wtG}\to
    \subscript G B_G$ and
    $\subscript{\wtG}{\wtC}_{\wtG}\to
    \subscript{G'}C_{G'}$ be the maximal forgetful morphisms. If
    $G\neq G'$, then return \texttt{no}.
  \item From now on assume $G=G'$. Using
    Algorithm~\ref{algo:minimaltor} or the oracle, decide whether
    $B,C$ are conjugate; if not, return \texttt{no}. Using
    Algorithm~\ref{algo:minimaltor} or the oracle compute the
    centralizer $Z(B)$.
  \item From now on, assume $B^\phi=C$ for some
    $\phi\in\Mod(G)$. Compute the action of $Z(B)$ on the list of all
    conjugacy classes of portraits of bisets (this list is finite by
    Corollary~\ref{cor:shaw is finite to one};
    Algorithms~\ref{algo:listportraittor}
    and~\ref{algo:listportraitexp} compute the list) and check whether
    $(G_a,B_a)_{a\in \wtA}$ and
    $(G^{\phi^{-1}}_a,C^{\phi^{-1}}_a)_{a\in \wtA}$ belong to
    the same orbit. Return \texttt{no} if the answer is
    negative. Otherwise replace $\phi$ by an element in $Z(B)\phi$
    such that $(G^\phi_a,B^\phi_a)_{a\in \wtA}$ and
    $(G_a,C_a)_{a\in \wtA}$ are conjugate portraits.
  \item Choose an arbitrary lift $\widetilde\phi\in\Mod(\wtG)$
    of $\phi$, and compute $k\in\Mod(\wtG)$ such that
    $k\wtB^{\widetilde\phi}\cong\wtC$. It follows that
    $k$ is a knitting element. Compute inductively $k^{(n)}$ with
    $k^{(n)} \wtB^{\widetilde\phi k' k''\cdots
      k^{(n-1)}}\cong\wtC$; proceed until $k^{(n)}$ is
    identity (this is guaranteed by
    Lemma~\ref{lmm:TrvOfLift2Cases}). Return
    $\widetilde\phi k' k''\cdots k^{(n-1)}$.\label{st:5:algo:red}
  \item To compute the centralizer of $\wtB$, consider first
    $Z(B)$. Compute the action of $Z(B)$ on the list of all conjugacy
    classes of portraits of bisets (this list is finite by
    Corollary~\ref{cor:shaw is finite to one};
    Algorithms~\ref{algo:listportraittor}
    and~\ref{algo:listportraitexp} compute the list) and compute the
    stabilizer $Z_0(B)$ of the $Z(B)$-action. Note that $Z_0(B)$ is a
    finite-index subgroup of $Z(B)$.
  \item List a generating set $S$ of $Z_0(B)$. For every $\phi\in S $
    compute its lift $\widehat\phi\in \Mod(\wtG)$ with
    $\wtB^{\widehat\psi}\cong \wtB$: compute first an
    arbitrary lift $\widetilde\phi$ of $\phi$, then inductively
    define $k^{(n)}$ with
    $k^{(n)} \wtB^{\widetilde\phi k' k''\cdots
      k^{(n-1)}}\cong\wtB$ until $k^{(n)}$ is identity, and
    set
    $\widehat\phi \coloneqq \widetilde\phi k' k''\cdots
    k^{(n-1)}$. Return $\{\widehat\phi\mid \phi\in S\}$.\label{st:7:algo:red}
  \end{enumerate}
\end{algo}

\begin{proof}[Proof of Corollary~\ref{cor:conjZpb}]
  We claim that
  Algorithms~\ref{algo:shadowing},~\ref{algo:conjportraittor},~\ref{algo:listportraittor},
  ~\ref{algo:conjportraitexp},~\ref{algo:listportraitexp} are
  efficient.  Indeed, the efficiency of
  Step~\ref{st:4:algo:conjportraitexp} of
  Algorithm~\ref{algo:conjportraitexp} follows from expansion. The
  efficiency of Steps~\ref{st:5:algo:red} and~\ref{st:7:algo:red} of
  Algorithm~\ref{algo:red} follows from
  Lemma~\ref{lmm:TrvOfLift2Cases}. All the remaining steps are
  obviously efficient.
\end{proof}

We note that, in the case of $(2,2,2,2)$ bisets, the oracle itself is
efficient. Indeed the oracle reduces to solving conjugacy and
centralizer problems in the group $\SL_2(\Z)$, and
Theorem~\ref{thm:MatrConj} is efficient, because $\SL_2(\Z)$ has a
free subgroup of index $12$.

\begin{cor}\label{cor:expdecidable}
  There is an algorithm that, given two geometric bisets
  $\subscript G B_G$ and $\subscript H C_H$, decides whether $B$ and
  $C$ are conjugate, and computes the centralizer of $B$.
\end{cor}
\begin{proof}
  We briefly sketch an algorithm that justifies an existence of an
  oracle: given two minimal \Exp\ orbisphere bisets, whether they are
  conjugate and computes their centralizers; details will appear
  in~\cite{bartholdi-dudko:bc5}.

  Let $B,C$ be two minimal \Exp\ orbisphere $G$-$G$-bisets. They admit
  a decomposition into rational maps along the \emph{canonical
    obstruction}, which is computable
  by~\cite{selinger:canonical}. The graph of bisets along this
  decomposition is computable, and rational maps may be computed for
  each of the small bisets in the decomposition, e.g. by giving their
  co\"efficients as algebraic numbers with floating-point enclosures
  to distinguish them from their Galois conjugates. The bisets $B,C$
  are conjugate precisely when their respective rational maps are
  conjugate and the twists along the canonical obstruction agree; the
  first condition amounts to finite calculations with algebraic
  numbers, while the second is the topic
  of~\cite{bartholdi-dudko:bc2}*{Theorem~\ref{bc2:thm:A}}.

  The centralizer of a rational map is trivial,
  and~\cite{bartholdi-dudko:bc2}*{Theorem~\ref{bc2:thm:A}} shows that
  the centralizer of $B$ is computable.
\end{proof}

\begin{proof}[Proof of Corollary~\ref{cor:G}]
  If the rational map is $(2,2,2,2)$, then the oracle is efficient, as
  we noted above, so Corollary~\ref{cor:conjZpb} applies. In the other
  case, the rational map is hyperbolic, so it has trivial centralizer
  and no oracle is needed in the application of
  Algorithm~\ref{algo:red}.
\end{proof}

\section{Examples}
Finally, in this brief section, we consider some examples of portraits
of bisets. The first ones come from marking points on the Basilica map
$f(z)=z^2-1$, and more generally maps with three post-critical
points. The second ones from cyclic bisets (which are particularly
simple, but to which our main results apply with restrictions because
these bisets have automorphisms).

\subsection{Twisted marked Basilica}
Consider the Basilica polynomial
\[f(z)\coloneqq z^2-1\colon (\widehat\C,\{0,-1,\infty\})\selfmap.
\]
It has two fixed points $\alpha$ and $\beta$, with $\alpha\in(-1,0)$
and $\beta>1$. Let us take $\alpha$ to be the basepoint and let
$\subscript G B_G$ be the biset of $f$. Denote respectively by
$\gamma_{-1}$ and $\gamma_0$ the loops circling around $-1$ and $0$
as in Figure~\ref{Fig:BasDynPl}, and let
$\gamma_\infty=(\gamma_{-1} \gamma_0)^{-1}$ be the loop around
infinity. Then
\[G=\langle\gamma_{-1},\gamma_0,\gamma_\infty\mid \gamma_\infty \gamma_{-1}\gamma_0\rangle. \]

The basepoint $\alpha$ has two preimages $\alpha$ and $-\alpha$. Let
$x_1$ be the constant path at $\alpha$ and let $x_2$ be a path
slightly below $0$ connecting $\alpha$ to $-\alpha$. The presentation
of $B$ in the basis $S\coloneqq \{x_1,x_2\}$ is
\begin{align*}
  \gamma_{-1}=&\pair{1,\gamma_0}(1,2),\\
  \gamma_0=&\pair{\gamma_{-1},1},\\
  \gamma_\infty=&\pair{\gamma_{-1}^{-1}\gamma^{-1}_0,1}(1,2).
\end{align*}

\begin{figure}
\begin{tikzpicture}[>=stealth',scale=3,inner sep=1mm]
  \node at (-0.618,0) [label={90:$\alpha$}] (alpha) {$\bullet$};

  \node at (-1,0) [label={[label distance=-2mm]180:$-1$}] (mone) {$\bullet$};
  \coordinate[xshift=4mm] (monex) at (mone);
  \draw[->] (monex) .. controls +(145:6mm)  and +(215:6mm) .. node[left] {$\gamma_{-1}$} (monex);

  \node at (0,0) [label=0:{$0$}] (zero) {$\bullet$};
  \coordinate[xshift=-4mm] (zerox) at (zero);
  \draw (monex) -- (zerox);
  \draw[->] (zerox) node[above=2mm] {$\gamma_0$} .. controls +(-35:6mm)  and +(35:6mm) .. (zerox);

  \node[inner sep=0mm] at (1.618,0) [label={[label distance=-1mm]45:$\beta$}] (beta) {$\bullet$};
  \coordinate[shift={(225:4mm)}] (betax) at (beta.center);
  \draw (alpha.center) edge[bend right=45] node[above] {$\gamma_\beta$} (betax);
  \draw[->] (betax) .. controls +(10:6mm)  and +(80:6mm) .. (betax);

  \node at (0.618,0) [label=$-\alpha$] (malpha) {$\bullet$};
  \draw (alpha.center) edge[bend right=35] node[below] {$x_2$} (malpha.center);

  \draw [dashed] (1.3,0) .. controls +(0.3,0.5) and +(0.3,0.5) .. (2,0)
  .. controls +(-1.3,-1.3) and +(0.2,-1.7) .. node[above] {$t$} (-1.5,0)
  .. controls +(0,0.5) and +(0,0.5) .. (-0.75,0)
  .. controls +(0,-0.9) and +(-0.6,-1.0) .. (1.3,0);
  \end{tikzpicture}
  \caption{The dynamical plane of $z^2-1$. Loops
    $\gamma_{-1},\gamma_0, \gamma_\beta$ circle around
    $0,-1,\beta$ respectively. The curve $x_2$ connects $\alpha$ to
    its preimage $-\alpha$. The simple closed curve $t$ surrounds
    $\{-1,\beta\}$.}
  \label{Fig:BasDynPl}
\end{figure}
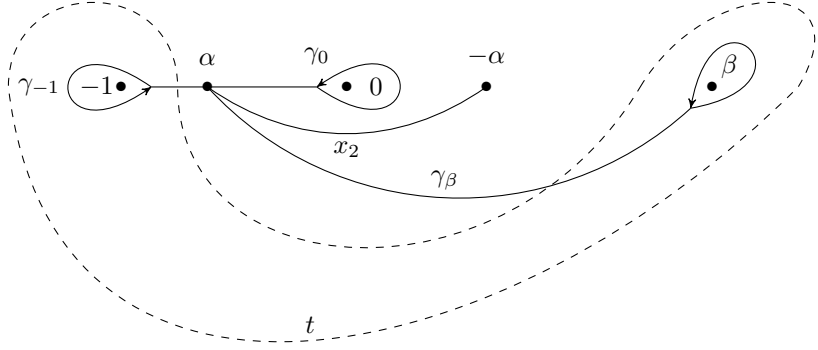

We shall consider first the effect of marking a fixed point, and then
of marking some preimages of the post-critical set.

\subsubsection{Marking $\beta$}\label{sss:MarkBeta}
Consider a marked Basilica
$\wtf(z)=z^2-1\colon(\widehat\C, \{\infty,0,-1,\beta\})\selfmap$. Let
$\gamma_\beta$ be the loop around $\beta$ depicted in
Figure~\ref{Fig:BasDynPl}. The fundamental group of
$(\widehat\C, \{\infty,0,-1,\beta\})$, based at $\alpha$, is
\[\wtG=\langle\gamma_{-1},\gamma_0,\gamma_\beta,\gamma_\infty\mid \gamma_\infty \gamma_{-1}\gamma_\beta\gamma_0\rangle
\]
and the forgetful map $\wtG\twoheadrightarrow G$ sends
$\gamma_\beta$ to $1$. The presentation of the biset
$\wtB=B(\wtf)$ in the basis $S$ is
\begin{align*}
\gamma_{-1}=&\pair{1,\gamma_0}(1,2),\\
\gamma_\beta=&\pair{1,\gamma_\beta},\\
\gamma_0=&\pair{\gamma_{-1},1},\\
\gamma_\infty=&\pair{\gamma_{-1}^{-1}\gamma^{-1}_0,\gamma_\beta^{-1}}(1,2).
\end{align*}
Write $A=\{\infty,-1,0\}$ and $\wtA=\{\infty,-1,0,\beta\}$;
then $\wtB$ has a minimal portrait of bisets
\begin{xalignat*}{2}
  \wtG_{-1}&=\langle\gamma_{-1}\rangle,&\wtB_{-1}&=\wtG_{-1}x_1,\\
  \wtG_\beta&=\langle\gamma_\beta\rangle,&\wtB_\beta&=\wtG_{-1}x_2,\\
  \wtG_0&=\langle\gamma_0\rangle,&\wtB_0&=\wtG_0x_1\sqcup \wtG_0x_2=x_1\wtG_{-1},\\
  \wtG_\infty&=\langle\gamma_\infty\rangle,&\wtB_\infty&=\wtG_\infty (\gamma_{-1}x_1)\sqcup
  \wtG_\infty (\gamma_{-1}x_2\gamma_\beta).\\
  \intertext{Under the forgetful intertwiner
    $\wtB\twoheadrightarrow B$ the portrait
    $(\wtG_a,B_a)_{a\in \wtA}$ projects to}
  G_{-1}&=\langle\gamma_{-1}\rangle,& B_{-1}&= G_{-1}x_1,\\
  G_\beta&=1,& B_\beta&=\{x_2\},\\
  G_0&=\langle\gamma_0\rangle,& B_0&= G_0x_1\sqcup  G_0x_2=x_1 G_{-1},\\
  G_\infty&=\langle\gamma_\infty\rangle,& B_\infty&= G_\infty (\gamma_{-1}x_1)\sqcup  G_\infty (\gamma_{-1}x_2).
\end{xalignat*}
Note that $B_\beta=\{x_2\}$ encodes $\beta$ in the sense that $x_2$
shadows $\beta$. This can be seen directly as follows: set $x_2^0=x_2$
and let $x_2^{i+1}$ be the $f$-lift of $x_2^i$ that starts at
$x_2^i(1)$. Then the sequence of endpoints $x_2^i(1)$ converges to
$\beta$, and the infinite concatenation $x_2^0\#x_2^1\#\cdots$ is a
path from $\alpha$ to $\beta$.

Let now $T$ be the clockwise Dehn twist around the simple closed curve
$t$ surrounding an interval $0$ and $\beta$ as in
Figure~\ref{Fig:BasDynPl}. The action of $T$ on $\wtG$ is
given by
\[T_*\colon\gamma_{-1}\to
  \gamma_{-1}^{(\gamma_{-1}\gamma_\beta)^{-1}},\gamma_\beta\to
  \gamma_\beta^{(\gamma_{-1}\gamma_\beta)^{-1}},\gamma_0\to
  \gamma_0,\gamma_\infty\to\gamma_\infty.
\]
The biset $\wtC$ of $g\coloneqq T \circ f $ is
$\wtB\otimes B_{T_*}$.  Let us show that $g$ is conjugate to
${z^2-1\colon (\widehat\C,\{\infty,-1,0,\alpha\})\selfmap}$.  By
Theorem~\ref{thm:portrait},
\[(\wtC_a,\wtG_a)_{a\in \wtA}= (\wtB_a,\wtG_a)_{a\in \wtA} \Antipush(T).
\]
Note that we have $C_a=B_a$ for all $a\in A$, so it remains to compute
$C_\beta$.

Let $\bar t\in \pi_1(\C\setminus\{0,-1\},\beta)$ be the simple loop
below $0$ circling around $-1$; then $T=\Push(\bar t)$. Let
$\ell_\beta\colon [0,1]\to\C\setminus\{-1,0\}$ (using notations of
Lemma~\ref{lmm:connect:ell_a}) be an arc from $\alpha$ to $\beta$
slightly below $0$ so that $\gamma_\beta$ may be homotoped to a small
neighborhood of $\ell_\beta$. We have
$C_\beta=B_\beta\gamma_{-1}^{-1}$; the claim then follows from
$x_2\gamma_{-1}^{-1}=x_1$ and the fact that $x_1$ shadows $\alpha$.

More generally, suppose that
$\wtg= \Push (\bar s) \wtf\Push (\bar t)$ for some
motions $\bar s, \bar t\in \pi_1(\C\setminus\{0,-1\},\beta)$ of
$\beta$; then
\[C_\beta=(\ell_\beta \#\bar s \#\ell_\beta^{-1})B_\beta(\ell_\beta \#
  \bar t \#\ell_\beta^{-1}).
\]
The process
\[(\ell_\beta \#\bar s \#\ell_\beta^{-1})x_2 (\ell_\beta \#\bar t \#\ell_\beta^{-1})=g_1x_{i(1)}\sim x_{i(1)}g_1= g_2 x_{i(2)}\sim x_{i(2)}g_2=\dots
\]
eventually terminates in either $x_1$ or $x_2$. In the former case,
$\wtg$ is conjugate to
$z^2-1\colon (\widehat\C, \{\infty,0,-1,\alpha\})\selfmap$ and in
the latter case $\wtg$ is conjugate to
$z^2-1\colon (\widehat\C, \{\infty,0,-1,\beta\})\selfmap$.

\subsubsection{Marking $1$ and $\sqrt2$}\label{sss:MarkSqrt2}
Consider now
\[\wtf(z)=z^2-1\colon(\widehat\C,\{\infty, -1,0,1,\sqrt2
  \})\selfmap
\]
with $\wtA\setminus A=\{1,\sqrt2\}$ and
$A=\{\infty,-1,0\}$. The dynamics are
$\sqrt2\mapsto1\mapsto0\leftrightarrow-1$. As in~\S\ref{sss:MarkBeta}
we readily compute a presentation of $\wtB=B(\wtf)$ in the
basis $S$:
\begin{align*}
  \gamma_{-1}=&\pair{1,\gamma_0}(1,2),\\
  \gamma_{\sqrt2}=&\pair{1,1},\\
  \gamma_1=&\pair{1,\gamma_{\sqrt2}},\\
  \gamma_0=&\pair{\gamma_{-1},\gamma_1},\\
  \gamma_\infty=&\pair{\gamma_{-1}^{-1}\gamma^{-1}_0,\gamma_1^{-1}\gamma_{\sqrt2}^{-1}}(1,2),
\end{align*}
with
$\wtG=\langle\gamma_\infty,\gamma_{-1},\gamma_{\sqrt2},\gamma_1,\gamma_0\mid
\gamma_\infty\gamma_{-1}\gamma_{\sqrt2}\gamma_1\gamma_0\rangle$.

As in Lemma~\ref{lmm:connect:ell_a}, let $\ell_{-1}$ and
$\ell_0\subset \R$ respectively be simple arcs from $\alpha$ to $-1$
and to $0$ so that $\gamma_{-1}$ and $\gamma_0$ may be homotoped to
small neighborhoods of $\gamma_{-1}$ and $\gamma_0$
respectively. Let $\ell_1$ and $\ell_{\sqrt2}$ be simple arcs from
$\alpha$ to $1$ and $\sqrt2$ slightly below $0$ as in
Figure~\ref{Fig:BasDynPl:2}.

\begin{figure}
\begin{tikzpicture}[>=stealth',scale=4.8,inner sep=1mm]
  \node (alpha) at (-0.618,0) [label={above:$\alpha$}] {$\bullet$};
  \node[inner sep=0] (mone) at (-1,0) [label={above:$-1$}] {$\bullet$};
  \draw (alpha.center) edge[->] node[above] {$\ell_{-1}$} (mone);
  \node[inner sep=0] (zero) at (0,0) [label={above:$0$}] {$\bullet$};
  \draw (alpha.center) edge[->] node[above] {$\ell_0$} (zero);
  \coordinate (malpha) at (0.618,0);
  \node[inner sep=0] (one) at (1,0) [label={above:$1$}] {$\bullet$};
  \draw (alpha.center) edge[->,bend right=20] node[above] {$\ell_1$} (one);
  \draw[dotted,thick] (one) edge[->] node[above] {$\ell^{-1}_0\lift{f}{1}$} (malpha) (malpha) node[above] {$-\alpha$};

  \node[inner sep=0] at (1.414,0) [label={above:$\sqrt2$}] (sqrt2) {$\bullet$};
  \draw (alpha.center) edge[->,bend right=45] node[above] {$\ell_{\sqrt2}$} (sqrt2);
  \draw[dotted,thick] (sqrt2) edge[->,bend left=45] node[above] {$\ell^{-1}_1\lift{f}{\sqrt2}$} (malpha);
  \end{tikzpicture}
  \caption{Curves $\ell_{-1},\ell_0,\ell_1,\ell_{\sqrt2}$ and their lifts.}
  \label{Fig:BasDynPl:2}
\end{figure}
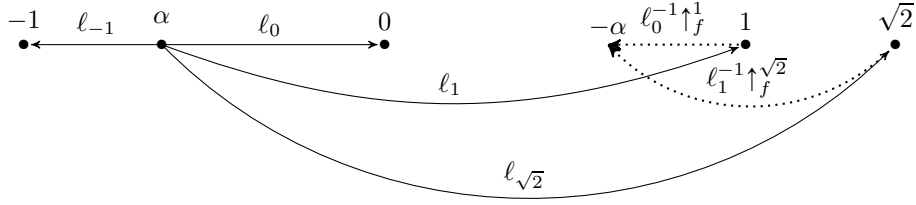

Denote by $(G_a,B_a)_{a\in \wtA}$ the portrait of bisets in
$B$ induced by $\wtB=B(\wtf)$. Then $G_a$ and $B_a$
are the same as in~\S\ref{sss:MarkBeta} for all $a\in A$, while (by
Lemma~\ref{lmm:connect:ell_a:2})
\[B_1=\{\ell_1\#\ell_0^{-1}\lift{f}{1}\}=\{x_2\}\text{ and }B_{\sqrt2}=\{\ell_{\sqrt2}\#\ell_1^{-1}\lift{f}{\sqrt2}\}=\{x_2\}.\]

We consider again some twists of $\wtf$.  Consider
$\wtg=m_1\wtf m_2$ with $m_1,m_2$ trivial rel $A$. Suppose
that $m_1$ moves $1$ and $\sqrt2$ along
$s_1\in \pi_1(\C\setminus\{-1,0\},1)$ and
$s_{\sqrt2}\in \pi_1(\C\setminus\{-1,0\},\sqrt2)$ respectively
while $m_2$ moves $1$ and $\sqrt2$ along
$t_1\in \pi_1(\C\setminus\{-1,0\},1)$ and
$t_{\sqrt2}\in \pi_1(\C\setminus\{-1,0\},\sqrt2)$ respectively. If
$(G_a,C_a)_{a\in \wtA}$ is the portrait of bisets in $B$
induced by $B(\wtg)$, then
\begin{align*}
  C_1=&\left\{(\ell_1\#s_1\#\ell^{-1}_1)x_2\right\},\\
  C_{\sqrt2}=&\left\{(\ell_{\sqrt2}\#s_{\sqrt2}\#\ell^{-1}_1)x_2 (\ell_{\sqrt2}\#t_1\lift{f}{\sqrt2} \#\ell^{-1}_{\sqrt2})\right\}.
\end{align*}
Write $C_1=\{h_1 x_2\}$ and $C_{\sqrt2}=\{h_2 x_j\}$. We may
conjugate $(C_1,C_2)$ to $(\{x_2\},\{x_j h'\})$, and write
$x_j h'=h'' x_k$. If $k=1$, then $\wtg$ is conjugate to
${z^2-1\colon(\widehat\C,\{\infty, -1,0,1,\sqrt2\})\selfmap}$;
otherwise $\wtg$ is conjugate to
${z^2-1\colon(\widehat\C,\{\infty, -1,0,1,-\sqrt2\})\selfmap}$.

\subsubsection{Mapping class bisets}
We continue the discussion from~\S\ref{sss:MarkSqrt2}, and consider
the related mapping class bisets. The biset $M(B)$ is of course
reduced to $\{B\}$, with $\Mod(G)=1$.

On the other hand, the biset $M(\wtB)$ is a left-free
$\Mod(\wtG)$-biset of degree $2$. This can be seen in various
ways: analytically, the point $\sqrt2$ may be moved to $-\sqrt2$, and
a basis of $M(\wtB)$ may be chosen as
$\{(z^2-1,\{\infty,-1,0,1,\sqrt2\}),(z^2-1,\{\infty,-1,0,1,-\sqrt2\})\}$. More
symbolically, the action of exchanging $\sqrt2$ with $-\sqrt2$ amounts
to changing, in the wreath recursion of $\wtf$, the entry
`$\gamma_1=\pair{1,\gamma_{\sqrt2}}$' into
`$\gamma_1=\pair{\gamma_{\sqrt2},1}$'.

Note that the biset $M(\wtB)$ is nevertheless connected, and
$M(\wtB)=M^*(\wtB)$: indeed the right action by the
mapping class that pushes $1$ once along the circle $\{|z|=1\}$ has
the effect of exchanging the two left orbits.

\begin{exple}\label{ex:notexact}
  Let us now consider the sets $A=\{\infty,-1,0,1\}$ and
  $D=\{\sqrt2\}$, still with the map
  $f(z)=z^2-1\colon(\widehat\C,A)\selfmap$ and
  $\wtf=f\colon(\widehat\C,A\sqcup D)\selfmap$. Then the fibres of the
  map $\Forget_{D,D}\colon M^*(\wtf)=M(\wtf)\to M(f)$ are not
  connected (by the second claim of Proposition~\ref{prop:FibBiset}).

  Indeed the left action of $\Mod(\widehat\C\setminus A,D)$ has two
  orbits, while the right action does not identify these orbits since
  $1$ is not allowed to move around a critical value.
\end{exple}

\subsection{Belyi maps}
The Basilica map $f(z)=z^2-1$ is an example of dynamical \emph{Belyi
  map}, namely a map whose post-critical set consists of $3$
points. All such maps are realizable as holomorphic maps, and the
three points may be normalized as $\{0,1,\infty\}$.

In this subsection, we briefly state how the main results of this
article simplify considerably. We concentrate on the dynamical
situation, namely a map $f\colon(S^2,A)\selfmap$ covered by
$\wtf\colon(S^2,A\sqcup D)\selfmap$, with $\#A=3$.

In that case, $\Mod(S^2,A)=1$, and we have a short exact sequence
\[1\to\knBraid(S^2,A\sqcup D)\to\Mod(S^2,A\sqcup D)\to\pi_1(S^2\setminus A,D)\to1.
\]
Assume first that $D$ consists only of $\wtf$-fixed
points. Then the extension of bisets decomposition of
$M(\wtf)$ from Theorem~\ref{mainthm:extension} reduces to the
statement that $M(\wtf)$ is an inert extension of
$B(f)^D$. The case of $D$ consisting of a single fixed point was
considered in~\cite{bartholdi-dudko:bc2}*{\S\ref{bc2:ss:BelyaMaps}}.

More concretely: if we are given a wreath recursion $G\to G\wr S\perm$
for $B(f)$, with $G=\pi_1(S^2\setminus A,*)$, then a generating set
for $\Mod(S^2,A\sqcup D)$ can be chosen to consist of $\#D$ copies of
a generating set of $G$, corresponding to point pushes of $D$ in
$S^2\setminus A$, together with some additional knitting elements. A
wreath recursion for $M(\wtf)$ will then be of the form
$\Mod(S^2,A\sqcup D)\to\Mod(S^2,A\sqcup D)\wr (S^D)\perm$, consisting
of $D$ parallel copies of the wreath recursion of $B(f)$. The wreath
recursion associated with the knitting elements is trivial.

In case $D$ consists of periodic points of period $>1$, then the
actions of $G$ on the left and the right should be appropriately
permuted. If $D$ contains $n_i$ cycles of period $i$, then abstractly
$M(\wtf)$ will consist in the direct product of $n_i$ copies
of $B(f^i)=B(f)^{\otimes i}$.

\begin{exple}\label{ex:M*neqM}
  Consider the map $\wtf(z)=f(z)=z^3$ with $A=\{0,1,\infty\}$ and
  $D=\{\omega=\exp(2\pi i/3)\}$. The biset $M(f)$ is of course reduced
  to $\{f\}$, while
  $M(\wtf)=\Mod(\widehat\C\setminus A,D)\cong\pi_1(\widehat\C\setminus
  A)$ with trivial right action. Indeed $D$ is not in the image of
  $f$, so pushing $\omega$ has no effect.

  However, $M^*(\wtf)$ has two left
  $\Mod(\widehat\C\setminus A,D)$-orbits by Lemma~\ref{lem:twisting bar E}. Representatives may be
  chosen as $\{(z^3,A\sqcup\{\omega\}),(z^3,A\sqcup\{\omega^2\})\}$,
  or equivalently (if the marked set is to remain $A\sqcup D$) as the
  maps $\{z^3,z^3\circ m\}$ for a homeomorphism
  $m\colon(\widehat\C,A)\selfmap$ that pushes $\omega$ to $\omega^2$.
\end{exple}

\subsection{Cyclic bisets}
\label{ss:CyclBis}
We finally consider the easy case of a cyclic biset, namely the biset
$\subscript G B_G$ of a monomial map
$f(z)=z^d\colon(\widehat\C,\{0,\infty\})\selfmap$ for some
$d\in\Z\setminus\{-1,0,1\}$. Then $G\cong\Z$ and $B$ is a left-free
right-principal biset. Choosing $b_0\in B$ we can identify $B=b_0G$
with $\Z$ with the actions are given by
\[m\cdot b\cdot n=d m+b+n.\]

Observe first that $b\to b+k$ is an automorphism of $B$ for all
$k\in \Z$. Therefore, contrarily Proposition~\ref{mainprop:NoGhostAut}
we have $\Aut(B)\cong\Z$.

Suppose that $(G_a,B_a)_{a\in\wtA}$ is a portrait of bisets in
$B$. Then $\Aut(B)$ acts on portraits by
\[(G_a,B_a)_{a\in\wtA} +k \coloneqq
  (G_a,B_a+k)_{a\in\wtA}.
\]
We denote by $(G_a,B_a)_{a\in\wtA}/\Aut(B)$ the orbit of this
action.  We can now adjust Theorem~\ref{main:thm:A} for cyclic bisets:
\begin{lem}\label{lem:main:thm:A:cyclic case}
  Let $\wtG\twoheadrightarrow G$ be a forgetful morphism of
  groups with $G\cong\Z$, and let $\subscript G B_G$ be the biset of
  $z^d\colon(\widehat\C,\{0,\infty\})\selfmap$.

  There is then a bijection between, on the one hand, conjugacy
  classes of portraits of bisets $(B_a)_{a\in\wtA}$ in $B$
  considered up to the action of $\Aut(B)$ and, on the other hand,
  $\wtG$-$\wtG$-bisets projecting to $B$ under
  $\wtG\twoheadrightarrow G$ considered up to composition with
  the biset of a knitting element. This bijection maps every minimal
  portrait of bisets of $\wtB$ to
  $(B_a)_{a\in\wtA}/\Aut(B)$.
\end{lem}
\begin{proof}
  Mark an extra fixed point in $(\widehat\C,\{0,\infty\})$ so as to
  remove automorphisms, and apply Theorem~\ref{main:thm:A}.
\end{proof}

To illustrate Lemma~\ref{lem:main:thm:A:cyclic case}, consider
$\wtf(z)=z^d\colon (\widehat\C,\{0,\infty\}\cup D)\selfmap$
and let $(G_a,B_a)_{a\in \wtA}$ be the induced portrait of
bisets on $(\widehat\C,\{0,\infty\})$. Let us also assume that
$d>1$. For $a\in D$ write $B_a=\{b_a\}$ with $b_a\in B\cong \Z$.

Let us consider a twisted map $\wtg=m\wtf n$. For
$a\in D$ let $m_a$ and $n_a$ be the number of times $m$ and $n$ push
$a$ around $0$. If $(G_a,C_a)_{a\in \wtA}$ is the induced
portrait of bisets, then $C_a=\{d m_a+b_d+n_{f_*(a)}\}=\{c_a\}$ for
$a\in A$. For $a\in D$ set
\[x_a\coloneq c_a/d+c_{f_*(a)}/d^2+ c_{f_*^2(a)}/d^3+\dots \mod \Z\in
  \R/\Z.\] Then $(C_a)_{a\in D}$ shadows $(x_a)_{a\in D}$. The map
$\wtg$ is unobstructed if and only if all points $x_a$ are
pairwise different. If $\wtg$ is unobstructed, then
$(x_a)_{d\in D}$ considered up to the action
\[(x_a)_{a\in D} \to (x_a+k)_{a\in D},k\in\Z/(d-1)\Z\]
is a complete conjugacy invariant of $\wtg$.


\end{document}